\documentclass[hidelinks,12pt, reqno]{amsart}
\usepackage{amsmath}
\usepackage{amssymb,amscd}
\usepackage{graphicx}
\usepackage{extarrows}
\usepackage{latexsym} 
\usepackage{tikz-cd}
\usepackage{enumitem}
\usepackage{float}
\usepackage[colorlinks]{hyperref}
\usepackage{verbatim} 
\usepackage{amsfonts}
\usepackage{mathrsfs}
\usepackage{nccmath}
\usepackage{mathtools}
\usepackage{tabularray}

\usepackage{etoolbox}
\usepackage{titlecaps}
\patchcmd{\subsection}{\bfseries}{\titlecap}{}{}
\patchcmd{\subsection}{-.5em}{.5em}{}{}

\usepackage{xpatch}
\makeatletter   
\xpatchcmd{\@tocline}
{\hfil\hbox to\@pnumwidth{\@tocpagenum{#7}}\par}
{\ifnum#1<0\hfill\else\dotfill\fi\hbox to\@pnumwidth{\@tocpagenum{#7}}\par}
{}{}
\makeatother 

\makeatletter
 \def\l@subsection{\@tocline{2}{0pt}{4pc}{6pc}{}}
\def\l@subsubsection{\@tocline{3}{0pt}{8pc}{8pc}{}}
 \makeatother

\numberwithin{equation}{section}

\newcommand{\C}{\mathbb{C}}
\newcommand{\N}{\mathbb{N}}
\newcommand{\R}{\mathbb{R}}
\newcommand{\Z}{\mathbb{Z}}
\newcommand{\T}{\mathbb{T}}

\newcommand{\Q}{\mathbb{Q}}
\newcommand{\D}{\mathbb{D}}

\DeclareMathOperator{\pr}{Pr}
\DeclareMathOperator{\id}{Id}

\DeclareMathOperator{\img}{img}
\DeclareMathOperator{\en}{End}

\DeclareMathOperator{\diff}{Diff}
\DeclareMathOperator{\spn}{Span}
\DeclareMathOperator{\ad}{Ad}
\DeclareMathOperator{\Hom}{Hom}
\DeclareMathOperator{\rank}{Rank}
\DeclareMathOperator{\type}{Type}
\DeclareMathOperator{\codim}{Codim}
\DeclareMathOperator{\im}{Im}
\DeclareMathOperator{\re}{Re}
\DeclareMathOperator{\sym}{Sym}
\DeclareMathOperator{\tr}{Trace}
\DeclareMathOperator{\pic}{Pic}
\DeclareMathOperator{\hol}{Hol}
\DeclareMathOperator{\sign}{Sign}

\theoremstyle{plain}
\newtheorem{theorem}{Theorem}[section]
\newtheorem{corollary}[theorem]{Corollary}
\newtheorem{lemma}[theorem]{Lemma}
\newtheorem{prop}[theorem]{Proposition}

\theoremstyle{definition}
\newtheorem{definition}[theorem]{Definition}
\newtheorem{remark}[theorem]{Remark}
\newtheorem{example}[theorem]{Example}
\newtheorem{question}[theorem]{Question}
\newtheorem*{thm}{Theorem}

\def \mc{\mathcal}

\begin{document}

\title[Strong generalized holomorphic principal bundles]{Strong generalized holomorphic principal bundles}
\author[D. Pal]{Debjit Pal}

\address{Department of Mathematics, Indian Institute of Science Education and Research, Pune, India}

\email{mathdebjit@gmail.com}

\author[M. Poddar]{Mainak Poddar}

\address{Department of Mathematics, Indian Institute of Science Education and Research, Pune, India}

\email{mainak@iiserpune.ac.in}



\begin{abstract}
We introduce the notion of a strong generalized holomorphic (SGH) fiber bundle and develop connection and curvature theory for an SGH principal $G$-bundle over a regular generalized complex (GC) manifold, where $G$ is a complex Lie group. We develop a de Rham cohomology for regular GC manifolds, and a Dolbeault cohomology for SGH vector bundles. Moreover, we establish a Chern-Weil theory for SGH principal $G$-bundles under certain mild assumptions on the leaf space of the GC structure. We also present a Hodge theory along with associated dualities and vanishing theorems for SGH vector bundles. Several examples of SGH fiber bundles are given.
\end{abstract}

\maketitle

\renewcommand\contentsname{\vspace{-1cm}}
{
  \hypersetup{linkcolor=black}
  \tableofcontents
}

\section{Introduction}\label{intro}
Generalized complex (GC) geometry presents a unified framework for a range of geometric structures whose two extreme cases are complex and symplectic structures. The notion was introduced by Hitchin \cite{Hit} and developed to a large extent by his doctoral students Gualtieri \cite{Gua, Gua2} and Cavalcanti \cite{Cavth}. The appropriate generalization of holomorphic bundles in complex geometry to GC geometry has received much attention \cite{Gua, Gua2, Hitchin11,  wang14, lang2023}.

\vspace{0.2em}
The study of principal bundles or vector bundles involves four basic differential geometric aspects, namely: $(1)$ The exploration of connection and curvature, $(2)$ Chern-Weil theory and characteristic classes, $(3)$ Hodge theory and its associated dualities and vanishing theorems, and $(4)$ Deformation theory. Notably, in the case of holomorphic principal bundles or vector bundles, these four aspects reveal rich geometric properties
\cite{atiyah57,cartan58,hubrechts05,kobayashi14,kodaira58,kodaira58-2,voisin02}. Naturally, one can ask the following:
\begin{question}\label{qsn}
~    
\begin{enumerate}
\setlength\itemsep{0.5em}
    \item What kind of vector or principal bundle theory arises within GC geometry?
    \item How are these four classical geometric components represented within the framework of GC geometry?
\end{enumerate}
\end{question}
In \cite{Gua, Gua2}, Gualtieri introduced generalized holomorphic (GH) vector bundles, which are complex vector bundles defined over a GC manifold equipped with a Lie algebroid connection. Wang further extended this in \cite{wang14}, introducing GH principal bundles by extending the structure group action to an exact Courant algebroid. These provide an answer to $(1)$ in Question \ref{qsn}. Additionally, in \cite{yicao14}, Wang explored the deformation of GH vector bundles, covering one of the geometric aspects mentioned above.
However, the absence of a generalized complex structure (GCS) on the total space of the bundle in these notions is a hindrance to the investigation of the other aspects. Recently, in \cite{lang2023}, Lang et al. tried to address this by considering a new notion of GH vector bundles equipped with a GCS on the total space. Their GCS on the total space is locally a product the GCS on the base and the fiber. They also introduced the Atiyah sequence of such bundles and defined its splitting as a generalized holomorphic connection. This new notion of GH vector bundle is more rigid than the notion due to Gualtieri \cite{Gua, Gua2} and Hitchin \cite{Hitchin11}, and yields a strict subclass.  But, it has the potential of being more amenable to methods from complex geometry. 

\vspace{0.2em}
In this article, we generalize the work of \cite{lang2023} on vector bundles to fiber and principal bundles. To distinguish these bundles from the earlier notions due to Gualtieri and Wang \cite{wang14}, we refer to them as {\it strong generalized holomorphic} or {\it SGH} bundles. A regular GCS induces a regular foliation with symplectic leaves and a transverse complex structure. The SGH bundles are intuitively characterized by the fact they are flat along the leaves and transversely holomorphic. However, they form a bigger category than the category of holomorphic bundles on the leaf space (when the leaf space is a manifold or an orbifold), see Examples \ref{counter eg1}, and \ref{counter eg2}.    

\vspace{0.2em}
 Both the base and fiber of an SGH fiber bundle are GC manifolds, and the total space admits a GCS that is locally a product GCS derived from the base and the fiber, see Definition \ref{main def}. In the context of vector bundles, SGH vector bundles correspond precisely to the GH vector bundles of Lang et al. (cf. \cite{lang2023}). Similarly, in the realm of principal bundles, they are a subclass of the GH principal bundles analyzed by Wang (cf. \cite[Example 4.2]{wang14}). 
\vspace{0.2em}

The main contribution of this article is in adapting the methods of complex geometry to introduce a suitable Dolbeault cohomology theory for SGH vector bundles and in using it to develop suitable generalizations of Chern-Weil theory and Hodge theory for these bundles. To do so, an assumption that the leaf space of the symplectic foliation is a complex (K\"{a}hler) orbifold is often necessary. A more detailed outline of the paper is given below.       

\vspace{0.2em}
In Section \ref{prelim}, we describe the basic facts on GCS and generalized holomorphic (GH) maps. In Sections \ref{sgh fiber}-\ref{sec4}, we introduce the notions of SGH fiber bundle and SGH principal bundle.    

\vspace{0.2em}
In Section \ref{sec gh pb}, we follow Atiyah's approach to defining holomorphic connection of a holomorphic principal bundle \cite{atiyah57}, to construct the Atiyah sequence of an SGH principal $G$-bundle $P$ over a regular GC manifold $M$, where $G$ is a complex Lie group:
    \begin{equation*}
        \begin{tikzcd}
0 \arrow[r] & Ad(P) \arrow[r] & At(P) \arrow[r] & \mathcal{G}M \arrow[r] & 0 \,.
\end{tikzcd}
    \end{equation*}
Here, $\mathcal{G}M$ is the GH tangent bundle of $M$, $Ad(P)$ is the adjoint bundle of $P$, and $At(P)$ is the Atiyah bundle of $P$. A GH connection on $P$ is a  GH splitting of the above short exact sequence, and the Atiyah class is the obstruction to such a splitting; see Definition \ref{def:gh conn} and Theorem \ref{main3}. 
Furthermore, in Section \ref{sec atiyah}, \`{a} la Atiyah, we establish that the Atiyah class of an SGH vector bundle and the Atiyah class of its associated SGH principal bundle agree up to a sign (see Theorem \ref{main4}). 

\vspace{0.2em}
In Section \ref{sec orbi}, we discuss the leaf space associated to the regular symplectic foliation $\mathscr{S}$ of a GCS. In general, the leaf space $M/\mathscr{S}$ might lack the Hausdorff property, as illustrated in \ref{counter eg}. Nonetheless, assuming $M/\mathscr{S}$ is a smooth orbifold, we provide a structured description of $\mathscr{S}$ in Theorem \ref{orbi thm}. Moreover, in Section \ref{sec cohomo}, utilizing Theorem \ref{orbi thm}, we  develop the de Rham cohomology $H^{\bullet}_{D}(M)$ for regular GC manifolds in Proposition \ref{prop2}, and the Dolbeault cohomology $H^{\bullet,\star}_{d_{L}}(M,E)$ of an SGH vector bundle $E$ in Corollary \ref{cor5}. This leads to the notion of the curvature of a smooth generalized connection (see Definition \ref{def:gen conn}) on an SGH principal bundle in Section \ref{curv sec}, and also, provides a crucial relationship between the curvature and its Atiyah class, as follows:
\begin{thm}(Theorem \ref{main7})
 For an SGH principal $G$-bundle $P$ over a regular GC manifold $M$, with $G$ as a complex Lie group and $M/\mathscr{S}$ as a smooth orbifold, the $(1,1)$-component of the curvature of a smooth generalized connection of type $(1,0)$ on $P$, which is constant along the leaves, corresponds to the Atiyah class of $P$ in $H^{1}(M,\mathbf{\mathcal{G}^{*}M}\otimes_{\mathcal{O}_{M}}\mathbf{Ad(P)})$.   
\end{thm}
In Section \ref{sec c-w}, we establish the generalized Chern-Weil homomorphism for SGH principal bundles in Definition \ref{def:c-w} using the generalized connection of Theorem \ref{main7}, and define the generalized characteristic classes. 
In Section \ref{sec sgh vb}, we develop the theory of smooth generalized connection and its curvature for SGH vector bundles. We also introduce the notion of a transverse connection and its curvature in Definition \ref{transvrs connectn} and present a Chern-Weil theory of SGH vector bundles similar to Section \ref{sec c-w}. In particular, we prove that the existence of a GH connection on an SGH bundle is equivalent to the existence of a GH connection on its associated SGH principal bundle; see Theorem \ref{main8}. We also provide an analogue of the holomorphic Picard group in Subsection \ref{sec picard}.

\vspace{0.2em}
In Section \ref{sec duality}, we develop generalized versions of classical results such as Serre duality and Poincar\'{e} duality.  We also introduce a Hodge decomposition for the $D$-cohomology and $d_L$-cohomology (see Section \ref{sec cohomolgy}) of a regular GC manifold. We establish the following result.
\begin{thm}(Theorem \ref{harmonic thm 2})
For any compact regular GC manifold $M$ of type $k$, given that $M/\mathscr{S}$ is a smooth orbifold, the following holds. 
 \begin{enumerate}
 \setlength\itemsep{1em}
     \item $H^{\bullet}_{D}(M)\cong (H^{2k-\bullet}_{D}(M))^{*}\,,$ and $H^{\bullet,\bullet}_{d_{L}}(M)\cong (H^{k-\bullet,k-\bullet}_{d_{L}}(M))^{*}\,.$
     \item If $\mathscr{S}$ is also transversely K\"{a}hler, we have,
      $H^{\bullet}_{D}(M)=\bigoplus_{p+q=\bullet}H^{p,q}_{d_{L}}(M)\,.$
 \end{enumerate}    
\end{thm}
Extending Theorem \ref{harmonic thm 2}, when the coefficient of $H^{\bullet,\star}_{d_{L}}(M)$ is an SGH vector bundle, we establish a generalized Hodge decomposition in Theorem \ref{harmonic thm 3} and provide a generalized Serre duality in Theorem \ref{gc serre}, under the assumption of Theorem \ref{harmonic thm 2}. Additionally, we establish the following vanishing theorem.
\begin{thm}(Theorem \ref{vanishing thm})
 Let $M$ be a  compact regular GC manifold  of type $k$ such that $M/\mathscr{S}$ is a K\"{a}hler orbifold. For a positive SGH line bundle $E$ over $M$,  we have the following.
 
\vspace{0.2em}
\begin{enumerate}
 \setlength\itemsep{1em}
    \item  $H^{q}(M,(\mathbf{\mathcal{G}^{*}M})^{p}\otimes_{\mathcal{O}_{M}}\mathbf{E})=0\quad\text{for $p+q>k$}\,.$
    \item For any SGH vector bundle $E^{'}$ on $M$, there exists a constant $m_0$ such that 
    $H^{q}(M,\mathbf{E^{'}}\otimes_{\mathcal{O}_{M}}\mathbf{E}^{m})=0\quad\text{for $m\geq m_0$ and $q>0$}$.
\end{enumerate}   
\end{thm}

In Section \ref{sec calabi}, we give some criteria on the GCS so that the leaf space of the associated symplectic foliation is a smooth torus, and therefore, satisfies the
hypothesis that the leaf space be an orbifold, used in most of our results in previous sections. This is a generalization of a result of Bailey et al.  \cite[Theorem1.9]{cav17}.
Then, in Section \ref{sec nilpotent}, we give a complete characterization of the leaf space of a left invariant GCS on a simply connected nilpotent Lie group and its associated nilmanifolds. Finally, we construct some examples of nontrivial SGH bundles on the Iwasawa manifolds which show that the category of SGH bundles is in general different from the category of holomorphic bundles on the leaf space.

\subsection{Notation} 
\begin{itemize}
\setlength\itemsep{1em}
    \item If $E$ denotes an SGH fiber bundle over $M$, then $\mathbf{E}$ denotes the corresponding sheaf of GH sections of $E$, and for any open set $U\subseteq M$, we denote the set of GH sections of $E$ over $U$ by $\Gamma(U,\mathbf{E})$ or by $\mathbf{E}(U)$.
    \item If $E$ denotes a smooth fiber bundle over $M$, then $C^{\infty}(E)$ denotes the corresponding sheaf of smooth sections of $E$ and for any open set $U\subseteq M$, we denote the set of smooth sections of $E$ over $U$ by $C^{\infty}(U,E)$ or by $C^{\infty}(E)(U)$. In particular, if $E$ is a complex vector bundle, then $C^{\infty}(E)$ is the sheaf of $\C$-valued sections. Similarly, if $E$ is only a real vector bundle, then $C^{\infty}(E)$ is the sheaf of $\R$-valued sections.
    \item  Given a smooth manifold $M$, for any open set $U\subseteq M$, $C^{\infty}(U)$ denotes the ring of $\C$-valued smooth functions on $U$ and $C^{\infty}_{M}$ denotes the sheaf of $\C$-valued smooth functions on $M$. Also $C^{\infty}_{M,\R}$ denotes the sheaf of $\R$-valued smooth functions on $M$.
\end{itemize}

\medskip
{\hspace{-1em}\bf Acknowledgement.} The authors would like to thank Ajay Singh Thakur and Chandranandan Gangopadhyay for many stimulating and helpful discussions. The research of the first-named author is supported by a CSIR-UGC NET research fellowship, file no.  09/936(0256)/2019-EMR-I. They also thank the referees for their helpful comments and suggestions.

\section{Preliminaries}\label{prelim}
\subsection{Generalized complex structure (GCS)} We first start by recalling some basic notions of generalized complex (in short GC) geometry. In this subsection, we shall rely upon  \cite{Gua} and \cite{Gua2} for most of the definitions and results.  

\medskip
Given any $2n$-dimensional smooth manifold $M$, the direct sum of tangent and cotangent bundles of $M$, which we denote by $TM\oplus T^{*}M$, is endowed with a natural symmetric bilinear form of signature $(2n,2n)$,
\begin{equation}\label{bilinear}
    \langle X+\xi,Y+\eta\rangle\,:=\,\frac{1}{2}(\xi(Y)+\eta(X))\,.
\end{equation}
It is also equipped with the \textit{Courant Bracket}  defined as follows.
\begin{definition}
The Courant bracket is a skew-symmetric bracket defined on smooth sections of $TM\oplus T^{*}M$, given by
\begin{equation}\label{bracket}
    [X+\xi,Y+\eta] :=
    \mathcal{L}_{X}Y+\mathcal{L}_{X}\eta-\mathcal{L}_{Y}\xi-\frac{1}{2}d(i_{X}\eta-i_{Y}\xi),
\end{equation}  
where $X,Y\in C^{\infty}(TM)$, $\xi,\eta\in C^{\infty}(T^{*}M)$, and $\mathcal{L}_{X},\,\, i_{X}$ denote the Lie derivative and the interior product of forms with respect to the vector field $X$, respectively.
\end{definition}
We are now ready to define the notion of GCS on a $2n$-dimensional smooth manifold $M$ in two equivalent ways.
\begin{definition}(cf. \cite{Gua})\label{gcs}
A \textit{generalized complex structure (GCS)} is determined by any of the following two equivalent sets of data:

\vspace{0.5em}
\begin{enumerate}
\setlength\itemsep{1em}
    \item  A bundle automorphism $\mathcal{J}_{M}$ of $TM\oplus T^{*}M$ which satisfies the following conditions:

    \vspace{0.5em}
    \begin{itemize}
 \setlength\itemsep{1em}
        \item[(a)] $\mathcal{J}_{M}^{2}=-1\,,$
        \item[(b)] $\mathcal{J}_{M}^{*}=-\mathcal{J}_{M}\,,$ that is, $\mathcal{J}_{M}$ is orthogonal with respect to the natural pairing in \eqref{bilinear}\,,
        \item[(c)] $\mathcal{J}_{M}$ has vanishing {\it Nijenhuis tensor}, that is, for all $C, D \in C^{\infty}(TM\oplus T^{*}M)$,
   $$ N(C, D) :=[\mathcal{J}_{M}C, \mathcal{J}_{M}D]-\mathcal{J}_{M} [\mathcal{J}_{M}C, D] - \mathcal{J}_{M} [C, \mathcal{J}_{M} D] 
    - [C, D]=0\,. $$ 
    \end{itemize}
     \item A subbundle, say $L_{M}$, of $(TM\oplus T^{*}M)\otimes\C$ which is maximal isotropic with respect to the natural bilinear form \eqref{bilinear}, involutive with respect to the Courant bracket \eqref{bracket}, and satisfies $L_{M}\cap\overline{L_{M}}=\{0\}$.
\end{enumerate}
\vspace{0.3em}
When $M$ is a real vector space $V$, a GCS on $V$ is an endomorphism $\mathcal{J}_{V}$ of $V\oplus V^*$ satisfying $\mathcal{J}_{V}^{2}=-1$ and $\mathcal{J}_{V}^{*}=-\mathcal{J}_{V}$. The pair $(V,\mathcal{J}_{V})$ is then called a generalized complex (GC) linear space.
\end{definition}
In Definition \ref{gcs}, the two equivalent conditions are related to each other by the fact that the subbundle $L_M$ may be obtained as the $+i$-eigenbundle of the automorphism $\mathcal{J}_M$.

\vspace{0.5em}
Given any GC manifold  $(M,\,\mathcal{J}_{M})$, we can deform $\mathcal{J}_{M}$ by a real closed $2$-form $B$ to get another GCS on $M$,  \begin{equation}\label{B transformation}
        (\mathcal{J}_{M})_{B}=e^{-B}\circ\mathcal{J}_{M}\circ e^{B}\quad\text{where}\quad e^{B}=\begin{pmatrix} 
	       1 & 0 \\
	       B & 1 \\
	    \end{pmatrix}\,.
    \end{equation}
This operation is called a \textit{$B$-field transformation} (or a \textit{$B$-transformation}). Then, the corresponding $+i$-eigenbundle of $(\mathcal{J}_{M})_{B}$ is 
\begin{equation}\label{L-B}
(L_{M})_{B}=\{X+\xi-B(X,\,\cdot)\,|\,X+\xi\in L_{M}\}\,.    
\end{equation}
Let us consider some simple examples of GCS.
\begin{example}\label{complx eg}
Let $(M,\, J_{M})$ is a complex manifold with a complex structure $J_{M}$. Then the natural GCS on $M$ is given by the bundle automorphism  
\[
\mathcal{J}_{M}:=
\begin{pmatrix}
 -J_{M}     &0 \\
    
    0        &J^{*}_{M}
\end{pmatrix}: TM\oplus T^{*}M\longrightarrow TM\oplus T^{*}M\,.
\] The corresponding $+i$-eigen bundle is
$$L_{M}=T^{0,1}M\oplus(T^{1,0}M)^{*}\,.$$
\end{example}
\begin{example}\label{symplectic eg}
Let $(M,\,\omega)$ be a symplectic manifold with a symplectic structure $\omega$. Then, the bundle automorphism   
\[
\mathcal{J}_{M}:=
\begin{pmatrix}
    0    &-\omega^{-1}\\
    \omega    &0
\end{pmatrix}: TM\oplus T^{*}M\longrightarrow TM\oplus T^{*}M\,,
\] gives a natural GCS on $M$. The $+i$-eigen bundle of this GCS is 
$$L_{M}=\{X-i\omega(X)\,|\,X\in TM\otimes\C\}\,.$$
\end{example}  

\subsection{Generalized holomorphic (GH) map}  We recall some basic facts about generalized holomorphic maps, an analogue of holomorphic maps in the case of GC manifolds. For most of the definitions, we shall rely upon \cite{Gua} and \cite{ornea2011}. 

\medskip
Let $V$ be a real vector space.  Given a subspace $E\leq V\otimes\C$ and an element $\sigma\in\wedge^2 E^{*}$, consider the subspace 
\begin{equation}\label{isotropic set}
    L(E,\sigma):=\{X+\xi\in E\oplus V^{*}\otimes\C\,| \,\xi|_{E}=\sigma(X)\}\,,
\end{equation}
of $(V \oplus V^{\ast}) \otimes \mathbb{C} $.
By \cite[Proposition 2.6]{Gua}, $L(E,\sigma)$ is a maximal isotropic subspace of $(V\oplus V^{*})\otimes\C$ with respect to the bilinear pairing \eqref{bilinear}, and any maximal isotropic subspace of $(V\oplus V^{*})\otimes\C$ is of this form. Let $(V,\mathcal{J}_{V})$ be a GC linear space with $+i$-eigenspace $L_{V}$. Consider the projection map 
$$\rho:(V\oplus V^{*})\otimes\C\longrightarrow V\otimes\C\,.$$  Let $\rho(L_{V})=E_{V}$ and let $E_{V}\cap\overline{E_{V}}=\Delta_{V}\otimes\C$ where $\Delta_{V} \le V$ is a real subspace. Then by \cite[Proposition 4.4]{Gua}, we have
\begin{enumerate}
\setlength\itemsep{1em}
    \item $L_{V}=L(E_{V},\sigma)$ for some $\sigma\in\wedge^2 E_{V}^{*}$ ;
    \item $E_{V}+\overline{E_{V}}=V\otimes\C$ with a non-degenerate real $2$-form $\Omega_{\Delta_{V}}:=\im(\sigma|_{\Delta_{V}\otimes\C})$ on $\Delta_{V}\otimes\C$.
\end{enumerate}
Following \cite[Section 3]{ornea2011}, $\widetilde{P}_{V}:=L(\Delta_{V}\otimes\C,\Omega_{\Delta_{V}})$ is called the associated { \it linear Poisson structure} of $\mathcal{J}_{V}$ on $V\otimes\C$.
\begin{definition}\label{GC map} (\cite{ornea2011})
Let $\psi:(V,\mathcal{J}_{V})\longrightarrow (V^{'},\mathcal{J}_{V^{'}})$ be a linear map between two GC linear spaces. Then $\psi$ is called a \textit{generalized complex (GC) map} if

\vspace{0.5em}
\begin{enumerate}
\setlength\itemsep{1em}
    \item $\psi(E_{V})\subseteq E_{V^{'}}$ ,
    \item $\psi_{\star}(\widetilde{P}_{V})=\widetilde{P}_{V^{'}}$ where $\psi_{\star}$ denotes the pushforward of a Dirac structure, as in \cite[Section 1]{ornea2011}, namely,
      $$\psi_{\star}(\widetilde{P}_{V})=\{\psi(Y)+\eta\in (V^{'}\oplus V^{'*})\otimes\C\,|\,Y+\psi^{*}(\eta)\in\widetilde{P}_{V}\}\,.$$
\end{enumerate} 
\end{definition}

\begin{definition}\label{GH map} (\cite{ornea2011})
 A smooth map $\psi:(M,\mathcal{J}_{M})\longrightarrow (M^{'},\mathcal{J}_{M^{'}})$ between two GC manifolds is called a \textit{generalized holomorphic (GH) map} if for each $x\in M$,
    $$(\psi_{*})_{x}:T_{x}M\longrightarrow T_{\psi(x)}M^{'}$$ is a GC map.

\medskip
Let $J_{\R^{2}}$ be the standard complex structure on $\mathbb{R}^2$ so that $(\R^2,J_{\R^{2}})$ is identified with $\C$.Consider $(M^{'},\mathcal{J}_{M^{'}})=(\R^{2},\mathcal{J}_{\R^{2}})$ where $\mathcal{J}_{\R^{2}}$ is as in Example \ref{complx eg}. In this case, a GH map
$\psi$ is called a \textit{GH function}.
\end{definition}

\begin{remark}
 Given a $B$-field transformation of $\mathcal{J}_{V}$, we see that $\rho((L_{V})_{B})=\rho(L_{V})$ where $(L_{V})_{B}$ is as in \eqref{L-B}. Since the imaginary part of $\sigma$ is also preserved, the associated linear Poisson structures are the same for both GCS. This shows that the notions of GC map and GH map are insensitive to $B$-field transformations.   
\end{remark}

Let $(V,\mathcal{J}_{V})$ be a generalized complex (GC) linear space. Then $\mathcal{J}_{V}$ can be written as 
\begin{equation*}
    \mathcal{J}_{V}=
    \begin{pmatrix}
    -J_{V}     &\beta_{V} \\
    
    B_{V}         &J^{*}_{V}
    \end{pmatrix}
\end{equation*}
where $J_{V}\in\en(V)$, $B_{V}\in\Hom_{\R}(V,V^{*})$ and $\beta_{V}\in\Hom_{\R}(V^{*},V)\,.$ Using 
$\mathcal{J}^{*}_{V}=-\mathcal{J}_{V}$ (cf. Definition \ref{gcs}), we get $B_{V}\in\wedge^{2}V^{*}$ and $\beta_{V}\in\wedge^{2}V$.
\begin{lemma}\label{ev-ew lemma}
Let $\psi:V\longrightarrow W$ be a GC map between two GC linear spaces. Then $\psi(E_{V}\cap\overline{E_{V}})=E_{W}\cap\overline{E_{W}}$. 
\end{lemma}
\begin{proof}
The first part, $E_{W}\cap\overline{E_{W}}\subseteq\psi(E_{V}\cap\overline{E_{V}})$ follows in a straightforward manner from Definition \ref{GC map}. For the converse part, let $v\in E_{V}\cap\overline{E_{V}}$ be a non-zero element. Let $\psi(v)=w\in W\otimes\C\,,$ and let $\widetilde{\Omega}_{\Delta_{W}}:W\otimes\C\longrightarrow\Delta^{*}_{W}\otimes\C$ be an extension of $\Omega_{\Delta_{W}}:\Delta_{W}\otimes\C\longrightarrow\Delta_{W}^{*}\otimes\C$ such that $\widetilde{\Omega}_{\Delta_{W}}\in\wedge^2 W^{*}\otimes\C\,.$ Then, we have $\psi^{*}(\widetilde{\Omega}_{\Delta_{W}}(w))\in V^{*}\otimes\C$ and
\begin{equation}\label{ev-eq}
 \Omega^{-1}_{\Delta_{V}}(\psi^{*}(\widetilde{\Omega}_{\Delta_{W}}(w))|_{\Delta_{V}\otimes\C})+\psi^{*}(\widetilde{\Omega}_{\Delta_{W}}(w))\in\widetilde{P}_{V}\,.   
\end{equation}

Denote $  \Omega^{-1}_{\Delta_{V}}(\psi^{*}(\widetilde{\Omega}_{\Delta_{W}}(w))|_{\Delta_{V}\otimes\C}) \in {\Delta_{V}\otimes\C} $ by $v'$. Then,
\begin{align*}
 &v'=\Omega^{-1}_{\Delta_{V}}(\psi^{*}(\widetilde{\Omega}_{\Delta_{W}}(w))|_{\Delta_{V}\otimes\C})\,\\
 \implies
 &\Omega_{\Delta_{V}}(v')(v)=\widetilde{\Omega}_{\Delta_{W}}(w)(\psi(v))\,\\
 \implies
 &\Omega_{\Delta_{V}}(v',v)=0\,\,\,\,\,(\text{as}\,\,\,\psi(v)=w)\,\\
 \implies &v=kv'\,\,\,\,\,(\text{as}\,\,\,\Omega_{\Delta_{V}}\,\,\,\text{is non-degenerate on $E_{V}\cap\overline{E_{V}}$ and}\,\,\,k\in\C\backslash\{0\})\,\\
 \implies &v=k\,\Omega^{-1}_{\Delta_{V}}(\psi^{*}(\widetilde{\Omega}_{\Delta_{W}}(w))|_{\Delta_{V}\otimes\C})\,.
\end{align*}

Note that $\psi(\Omega^{-1}_{\Delta_{V}}(\psi^{*}(\widetilde{\Omega}_{\Delta_{W}}(w))|_{\Delta_{V}\otimes\C})+(\widetilde{\Omega}_{\Delta_{W}}(w))\in\widetilde{P}_{W}$ by \eqref{ev-eq}, and also $$\psi(\Omega^{-1}_{\Delta_{V}}(\psi^{*}(\widetilde{\Omega}_{\Delta_{W}}(w))|_{\Delta_{V}\otimes\C})=\frac{1}{k}\psi(v)=\frac{1}{k}w\,.$$ Thus $w\in E_{W}\cap\overline{E_{W}}$ and $\psi(E_{V}\cap\overline{E_{V}})\subseteq E_{W}\cap\overline{E_{W}}$. This proves the lemma. 
\end{proof}
\begin{remark}
    The assertion made in the statement of Lemma \ref{ev-ew lemma} is claimed in the proof of \cite[Proposition 3.2]{ornea2011}. However, the argument given there is not very explicit.
\end{remark}

The proof of the next lemma follows similar arguments in \cite[Lemma 2.12]{lang2023} and provides a generalization to the case $W=\C$ established there.

\begin{lemma}\label{gh map equi}
Let $(V,\mathcal{J}_{V})$ and $(W,\mathcal{J}_{W})$ are  GC linear spaces with
\[\mathcal{J}_{V}=
\begin{pmatrix}
  -J_{V}     &\beta_{V} \\
    
    B_{V}         &J^{*}_{V}  
\end{pmatrix}\,\,\,\,\,
\text{and}\,\,\,\,\,
\mathcal{J}_{W}=
\begin{pmatrix}
  -J_{W}     &0 \\
    
    0        &J^{*}_{W}  
\end{pmatrix}
\]
where $J_{W}$ is a complex structure on $W$.  Then $\psi:V\longrightarrow W$ is a GC map if and only if $$\psi\circ J_{V}=J_{W}\circ \psi\,\,\,\,\,,\,\,\,\,\,\psi\circ \beta_{V}=0\,.$$
\end{lemma}
\begin{proof}
 Let $\dim V=2n$ and let the type of $\mathcal{J}_{V}$ be $k\in\N\cup\{0\}$. Since the definition of GC map is invariant under a $B$-transformation,  we can assume without loss of generality that
$$(V,\mathcal{J}_{V})=(V_{1},J_{1})\oplus (V_{2},J_{2})\,,$$
  where $(V_{1},J_{1})=(\R^{2k},J_{\R^{2k}})=\C^{k}$  and $(V_{2},J_{2})=(\R^{2n-2k},\omega_0)$. Here, 
  $J_{\R^{2k}}$ and $\omega_0$ denote the standard complex and symplectic structures on the corresponding spaces. It follows that $E_{V}= V_1^{0,1}\oplus (V_2\otimes\C)$. Since $E_{V}\cap\overline{E_{V}}=V_2\otimes\C$, the Poisson bivector on $V$ is 
$$\widetilde{\beta}_{V}=\begin{cases}
     0, & \text{on } V^{*}_1\otimes\C\,,\\
    \omega_{0}^{-1}, & \text{on } V^{*}_2\otimes\C\,.
  \end{cases}$$
 Hence,
$$\widetilde{P}_{V}=L(V_2\otimes\C,\omega_{0})=L(V^{*}\otimes\C,\widetilde{\beta}_{V})\,.$$
  Similarly, as $W$ is a complex vector space, we have $E_{W}=W^{0,1}$ and so, $E_{W}\cap\overline{E_{W}}=\{0\}$. Thus, $\beta_{W}$, the Poisson bivector on $W$ is $0$ and we get $$\widetilde{P}_{W}=W^{*}\otimes\C=L(W^{*}\otimes\C,0)\,.$$ Then, by Lemma \ref{ev-ew lemma},
  $\psi_{\star}(\widetilde{P}_{V})=\widetilde{P}_{W}$ if and only if $\psi\circ\omega_{0}^{-1}=0$. Thus, $\psi$ is a GC map if and only if 
  $$\psi(V_1^{0,1}\oplus (V_2\otimes\C))\subset W^{0,1}\,\,\,{\rm and}\,\,\,\psi\circ\omega_{0}^{-1}=0\,.$$  Hence, for any $v_1\in V_1$ and $v_2\in V_2$,  we have 
  $$\psi((-J_1(v_1)))=(-J_{W})(\psi(v_1))\,\,\,\, {\rm and}\,\,\,\psi(v_2)=0\,.$$
  This implies $$\psi\circ J_{V}=J_{W}\circ \psi\,\,\,\,\,\text{and}\,\,\,\,\,\psi\circ\beta_{V}=0$$ where 
  \[J_{V}=
  \begin{pmatrix}
       J_1    &0\\
        0     &0
\end{pmatrix}\,\,\,\,\,\text{and}\,\,\,\,\,\beta_{V}=-\widetilde{\beta}_{V}=
  \begin{pmatrix}
      0     &0\\
      0     &-\omega_{0}^{-1} 
  \end{pmatrix}\,.
  \]
\end{proof}
Let $\psi$ be any complex valued linear function on $V$. Considered as an element of $V^{\ast}
\otimes \mathbb{C}$, $\psi$ has two components corresponding to the decomposition $(V\oplus V^{*})\otimes\C=L_{V}\oplus\overline{L_{V}}$, namely $\psi_{L_{V}}$ and $\psi_{\overline{L_{V}}}$. The following lemma is obtained by straightforward modification of \cite[Proposition 2.13]{lang2023} together with Lemma \ref{gh map equi}.

\begin{lemma}\label{imp lemma}
 A linear map $\psi:(V,\mathcal{J}_{V})\longrightarrow\C=(\R^2, J_{\R^2})$ between two GC linear spaces, is a GC map if and only if $\psi\in (L_{V}\cap (V^{*}\otimes\C))$ that is, $\psi_{\overline{L_{V}}}=0$.   
\end{lemma}
Let $(M,\mathcal{J}_{M})$ be a GC manifold with $+i$-eigenbundle $L_{M}$ so that we have,
$$(TM\oplus T^{*}M)\otimes\C=L_{M}\oplus\overline{L_{M}}\,.$$ Let $d$ be the exterior derivative on $M$. 

\begin{lemma}\label{imp lemma2}
Given an open set $U\subseteq M$, a smooth map $\psi:(U,\mathcal{J}_{U})\longrightarrow\C=(\R^2,J_{\R^{2}})$ is a GH function if and only if for each $x\in U$, $d\psi_{x}\in (L_{M}\cap (T^{*}M\otimes\C))_{x}$\,.
\end{lemma}
\begin{proof}
Follows from Lemma \ref{imp lemma} and Definition \ref{GH map}.     
\end{proof}

\begin{definition}\label{def:O_M}
    Let $\mathcal{O}_{M}$ denote the sheaf of  $\C$-valued GH functions on $M$.
\end{definition} 

By Lemma \ref{imp lemma2}, $\mathcal{O}_{M}$ is a subsheaf of the sheaf of smooth  $\C$-valued functions on $M$. To begin with, we consider some simple examples of $\mathcal{O}_{M}$.
\begin{enumerate}
\setlength\itemsep{1em}
    \item Let $(M,J_{M})$ be a complex manifold with $J_{M}$ as its complex structure. Consider the natural GCS induced by $J_{M}$, as given in Example \ref{complx eg}. By Lemma \ref{imp lemma2}, we can see that, given any GH map $\psi$, $d\psi\in\Omega^{1,0}(M)$, that is, $\psi$ is a holomorphic function. So $\mathcal{O}_{M}$ is the sheaf of holomorphic functions on $M$.
\item Let $(M,\omega)$ be a symplectic manifold with a symplectic structure $\omega$. Consider the induced GCS, as given in Example \ref{symplectic eg}, with the $+i$-eigenbundle $L_{M}$. Note that, $L_{M}$ is naturally identified with $TM\otimes\C$. So, $d\psi=0$ for any GH map $\psi$,  which implies $\psi$ is locally constant. Hence $\mathcal{O}_{M}$ is a sheaf of locally constant functions. 
\end{enumerate}

\begin{definition}(cf. \cite{lang2023})
A diffeomorphism $\phi:(M,\mathcal{J}_{M})\longrightarrow (N,\mathcal{J}_{N})$ between two GC manifolds is called a \textit{generalized holomorphic (GH) homeomorphism} if
\begin{equation}\label{gh homeo}
    \begin{pmatrix}
     \phi_{*}    &0 \\
     0           &(\phi^{-1})^{*} 
    \end{pmatrix}
    \circ \mathcal{J}_{M} = \mathcal{J}_{N}\circ 
    \begin{pmatrix}
     \phi_{*}    &0 \\
     0           &(\phi^{-1})^{*} 
    \end{pmatrix}\,.
\end{equation}
When $N=M$, $\phi$ is called GH automorphism.
\end{definition}
\begin{remark}
    Note that a GH homeomorphism $\phi$ and its inverse $\phi^{-1}$ are both GH maps. This can be observed as follows. Let $L_{M},L_{N}$ denote the $+i$-eigen bundles of $\mathcal{J}_{M},\mathcal{J}_{N}$, respectively. For every point $p$ in $M$, consider the subset of $(T_{\phi(p)}N\oplus T_{\phi(p)}^{*}N)\otimes\C$
    $$\phi_{\star}(L_{M}\mid_p):=\{\phi_{*}(X)+\eta\,|\,X+\phi^{*}\eta\in L_{M} \mid_p \}\,.$$ Then, for any $X\in T_pM\otimes\C$ and $\eta\in T_{\phi(p)}^{*}N\otimes\C$
    $$\mathcal{J}_{N}(\phi_{*}(X)+\eta)=\mathcal{J}_{N}\left( 
    \begin{pmatrix}
     \phi_{*}    &0 \\
     0           &(\phi^{-1})^{*} 
    \end{pmatrix}(X+\phi^{*}\eta)\right)\,.$$  By \eqref{gh homeo}, we get $Y+\xi\in\phi_{\star}(L_{M})$ if and only $Y+\xi\in L_{N}$, that is $\phi_{\star}(L_{M})=L_{N}\,,$ and using \cite[Corollary 3.3]{ornea2011}, we conclude that both $\phi$ and $\phi^{-1}$ are GH maps.

    \vspace{0.2em}
    But the converse is not true always. A GH map which is a diffeomorphism may not always be a GH homeomorphism. The reason is that a GH map is defined up to a $B$-transformation whereas a GH homeomorphism between two GC manifolds shows that their GC structures are the same.
\end{remark}
\begin{definition}\label{def:type} Let $\rho: (TM \oplus T^{\ast}M) \otimes \C \longrightarrow TM \otimes \C $ denote the natural projection. We denote the image of 
$L_M$ under $\rho$ by $E_M$. Let $\Delta_M \otimes \C := E_M \bigcap \overline{E_M}$. 

\vspace{0.2em}
    For each $x\in M$, type of $\mathcal{J}_{M}$ at $x$ is defined as 
    $$\type(x):=\codim_{\C}((E_{M})_{x})=\frac{1}{2}\codim_{\R}((\Delta_{M})_{x})\,.$$ $x$ is called a regular point of $M$ if $\type(x)$ is constant in a neighborhood of $x$ and $M$ is called a regular GC manifold if each point of $M$ is a regular point.
\end{definition}
 We have the generalized Darboux theorem around any regular point.
\begin{theorem}(\cite[Theorem 4.3]{Gua2})\label{darbu thm}
For a regular point $x\in (M,\mc{J}_{M})$ of $\type(x)=k$, there exists an open neighborhood $U_{x}\subset M$ of $x$ such that, after a $B$-transformation, $U_{x}$ is GH homeomorphic to $U_1\times U_2$, where $U_1\subset(\R^{2n-2k},\omega_0), U_2\subset\C^{k}$ are  open subsets with $\omega_0$ being the standard symplectic structure.
\end{theorem}
In a simpler terms, Theorem \ref{darbu thm} implies that, for some real closed form $2$-form $B_{\phi}\in\Omega^2(U_1\times U_2)$, there exists a GH homeomorphism $$\phi:(U_{x},\mc{J}_{U_{x}})\longrightarrow (U_1\times U_2,(\mathcal{J}_{U_1\times U_2})_{B_{\phi}})$$
where  $(\mathcal{J}_{U_1\times U_2})_{B_{\phi}}$ is  the $B_\phi$-transformation, as in \eqref{B transformation}, of the product GCS, denoted by $\mathcal{J}_{U_1\times U_2}\,.$ Let $p=(p_1,\ldots,p_{2n-2k})$ and $z=(z_1,\ldots,z_{k})$ represent coordinate systems for $\R^{2n-2k}$ and $\C^{k}$, respectively, and consider the corresponding local coordinates around $x$
\begin{equation}\label{loc coordi}
    (U_{x},\phi,p,z):=(U_{x},\phi\,;p_1,\ldots,p_{2n-2k},z_1,\ldots,z_{k})\,.
\end{equation}
We note that the subspaces $(E_M)_x$ and $(\overline{ E_M})_x$ admit the following
description,
\begin{equation}\label{eq:E,Ebar} 
\begin{array}{c}

 (E_M)_x = \spn\left\{ {\partial_{p_i}}\mid_x, 
\partial_{\overline{z_j}}\mid_x : 1 \le i \le 2n-2k,  1\le j \le k \right\} \,, \\
\\
(\overline{ E_M})_x =\spn\left\{ {\partial_{p_i}}\mid_x, 
\partial_{z_j}\mid_x : 1 \le i \le 2n-2k,  1\le j \le k \right\} \,.

\end{array}
\end{equation}

Using Theorem \ref{darbu thm} in the case of a regular GC manifold,  we obtain a nice description of coordinate transformations, as given in the following corollary. 
\begin{corollary}\label{cor:diffcharts}(\cite[Proposition 2.7]{lang2023})
    Let $(M,\mc{J}_{M})$ be a regular GC manifold of type $k$. Let's assume that $(U,\phi,p,z)$ and $(U',\phi',p',z')$ are two local coordinate systems, as in \eqref{loc coordi}, with $U\cap U'\neq\emptyset\,.$ Then,
    $$\frac{\partial z'_{i}}{\partial\overline{z_{j}}}=\frac{\partial z'_{i}}{\partial p_{l}}=0\quad\text{for all}\quad i,j\in\{1,\ldots,k\}\,, l\in\{1,\ldots,2n-2k\}\,.$$
\end{corollary}
Furthermore, we have the following characterization of GH functions on a regular GC manifold in terms of the local coordinates in \eqref{loc coordi}.  
\begin{prop}(\cite[Example 2.8]{lang2023})\label{prop GH}
Let $(M,\mc{J}_{M})$ be a regular GC manifold of type $k$. Then, $f: M\longrightarrow\C$ is a GH function if and only if at every point on $M$, expressed in terms of local coordinates, as shown in \eqref{loc coordi}, $f$ satisfies the following 
$$\frac{\partial f}{\partial\overline{z_{j}}}=\frac{\partial f}{\partial p_{l}}=0\quad\text{for all}\quad i,j\in\{1,\ldots,k\}\,, l\in\{1,\ldots,2n-2k\}\,.$$
\end{prop}

This implies that for a regular GC manifold of type $k$, the sheaf $\mathcal{O}_M$ is locally given by the ring of convergent power series in the coordinates $z=(z_1,\ldots,z_{k})$ in \eqref{loc coordi}.
 
\section{Strong Generalized Holomorphic Fiber bundle}\label{sgh fiber}
Let $(M,\mathcal{J}_{M})$ be a generalized complex (GC) manifold. Then $\mathcal{J}_{M}$ can be written as 
\begin{equation}\label{matix}
    \mathcal{J}_{M}=
    \begin{pmatrix}
    -J_{M}     &\beta_{M} \\
    
    B_{M}         &J^{*}_{M}
    \end{pmatrix}
\end{equation}
where $J_{M}\in\en(TM)$, $B_{M}\in\Omega^{2}(M)$ and $\beta_{M}\in\mathfrak{X}^{2}(M)$. Let $\diff_{\mathcal{J}_{M}}(M)$ denote the subgroup of $\diff(M)$  defined by
\begin{equation}
 \diff_{\mathcal{J}_{M}}(M):=\{\phi\in\diff(M)\,|\,\phi\,\,\text{is a GH automorphism of}\,\,(M,\mathcal{J}_{M})\} \,. 
\end{equation}

\begin{definition}\label{main def} Let $G$ be a Lie group. A smooth fiber bundle $F\hookrightarrow E\xrightarrow{\pi} M$ over a GC manifold $(M,\mathcal{J}_{M})$ with a GC manifold $(F,\mathcal{J}_{F})$ as  fiber and structure group $G$  is called a \textit{strong generalized holomorphic (SGH) fiber bundle} if

\vspace{0.1em}
    \begin{enumerate}
\setlength\itemsep{1em}
        \item $E$ is a GC manifold,
        \item there is an open cover $\{U_{\alpha}\}$ of $M$ and a family of local trivializations $\{\phi_{\alpha}\}$ of $E$
        $$\{\phi_{\alpha}:\pi^{-1}(U_{\alpha})\longrightarrow U_{\alpha}\times F\}$$
        such that every $\phi_{\alpha}$ is a GH homeomorphism when  $U_{\alpha}\times F$ is endowed with the standard product GC structure.
    \end{enumerate}
    In addition, if $F$ is a vector space and $G$ is a subgroup of $GL(F)$, then we say that $E$ is an SGH vector bundle over $M$. In this case, $\mathcal{J}_{F}$ is constant (i.e. a generalized complex structure on the vector space $F$, rather than the manifold $F$); see Definition \ref{gcs}. 
\end{definition}
The following theorem is a generalization of \cite[Proposition 3.2]{lang2023}.
\begin{theorem}\label{main1}    
Let $E$ be a fiber bundle over $(M, \mathcal{J}_{M})$ with typical fiber $(F,\mathcal{J}_{F})$ and structure group $G$. Let $\{U_{\alpha},\phi_{\alpha}\}$ be a family of local trivializations with  transition functions $\phi_{\alpha  \beta}$ as follows, 
$$\{\phi_{\alpha}:\pi^{-1}(U_{\alpha})\longrightarrow U_{\alpha}\times F\}\,,\,\phi_{\alpha\beta}:U_{\alpha\beta}=U_{\alpha}\cap U_{\beta}\longrightarrow G\,,$$
where $\phi_{\alpha\beta}(x)=\phi_{\alpha}|_{\pi^{-1}(x)}\circ\phi^{-1}_{\beta}(x,\cdot)$ for all $x\in U_{\alpha\beta}\,.$
Then, $E$ is SGH fiber bundle over $M$ with local trivializations $\{U_{\alpha},\phi_{\alpha}\}$ if and only if

\vspace{0.3em}
\begin{enumerate}
\setlength\itemsep{1em}
    \item $\phi_{\alpha\beta}(m)\in\diff_{\mathcal{J}_{F}}(F)$ for all $m\in U_{\alpha\beta}$,
    \item For each $(m,f)\in M\times F$, the following equations hold:
    \begin{equation*} (\rho_{f})_{*}\circ(\phi_{\alpha\beta})_{*m}\circ J_{U_{\alpha\beta}}\,=\,J_{F}\circ(\rho_{f})_{*}\circ(\phi_{\alpha\beta})_{*m}\,,
    \end{equation*}
    \begin{equation*} (\rho_{f})_{*}\circ(\phi_{\alpha\beta})_{*m}\circ\beta_{U_{\alpha\beta}}\,=\,0\,, 
    \end{equation*}
    \begin{equation*}
     B_{F}\circ (\rho_{f})_{*}\circ(\phi_{\alpha\beta})_{*m}\,=\,0\,,   
    \end{equation*}
\end{enumerate}
for all $(m,f)\in M\times F$, where $J_{U_{\alpha\beta}}\,,\,J_{F}\,,\,\beta_{U_{\alpha\beta}}\,,\,B_{F}$ are as in equation \eqref{matix}, and the map $\rho_{f}:G\longrightarrow F$ is defined as 
$\rho_{f}(g) := g\cdot f$.
\end{theorem}

\begin{proof}
Consider the map 
\begin{equation}\label{psi alpha beta 2} \psi_{\alpha\beta}=\phi_{\alpha}\circ\phi^{-1}_{\beta}:U_{\alpha\beta}\times F\longrightarrow U_{\alpha\beta}\times F\,.
\end{equation}
Note that $\psi_{\alpha\beta}(m,f)=(m,\phi_{\alpha\beta}(m)\cdot f)$ for all $(m,f)\in U_{\alpha\beta}\times F\,.$ 

\medskip
First, we claim that $E$ is an SGH fiber bundle if and only if $\psi_{\alpha\beta}$ is a GH automorphism for any fixed $\alpha, \beta$. Indeed, if $\psi_{\alpha\beta}$ is a GH automorphism, then
\begin{equation}\label{GHdiff}
    \begin{pmatrix}
     (\psi_{\alpha\beta})_{*}    &0 \\
     0           &(\psi_{\beta\alpha})^{*} 
    \end{pmatrix}
    \circ \mathcal{J}_{U_{\alpha\beta}\times F} = \mathcal{J}_{U_{\alpha\beta}\times F}\circ 
    \begin{pmatrix}
     (\psi_{\alpha\beta})_{*}    &0 \\
     0           &(\psi_{\beta\alpha})^{*} 
    \end{pmatrix}\,,
\end{equation}
where $\mathcal{J}_{U_{\alpha\beta}\times F}= (J_{ij})_{2\times 2}$ is the product GC structure on $U_{\alpha\beta}\times F$. Then \begin{equation}
    \begin{pmatrix}
     (\phi_{\alpha}^{-1})_{*}    &0 \\
     0           &(\phi_{\alpha})^{*} 
    \end{pmatrix}
    \circ (J_{ij})_{2\times 2}\circ 
     \begin{pmatrix}
     (\phi_{\alpha})_{*}    &0 \\
     0           &(\phi_{\alpha}^{-1})^{*} 
    \end{pmatrix}\,
\end{equation}
is an endomorphism of $T\pi^{-1}(U_{\alpha})\oplus T^{*}\pi^{-1}(U_{\alpha})$ that produces the GC structure on $\pi^{-1}(U_{\alpha})$. By equation \eqref{GHdiff} this structure is independent of the choice of $\phi_{\alpha}$. Hence, we obtain a GC structure on $E$ such that $\phi_{\alpha}$ becomes GH homeomorphism. The converse is obvious.

\vspace{0.2em}
Now, it is enough to show that $\psi_{\alpha\beta}$ is a GH automorphism if and only if $(1)$ and $(2)$ are satisfied. Observe that, the product GC structure on $U_{\alpha\beta}\times F$ can be expressed as
\begin{equation*}
   J_{11}=
   \begin{pmatrix}
    -J_{U_{\alpha\beta}}     &0 \\
    
    0        &-J_{F}
   \end{pmatrix}
   \,, \,J_{21}=
   \begin{pmatrix}
      B_{U_{\alpha\beta}}   &0\\
      0                  &B_{F}
   \end{pmatrix}
   \,, \, J_{12}=
   \begin{pmatrix}
    \beta_{U_{\alpha\beta}}     &0 \\
    
    0        &\beta_{F}\,,
   \end{pmatrix}
   \,, \,J_{22}=
   \begin{pmatrix}
      J^{*}_{U_{\alpha\beta}}   &0\\
      0                  &J^{*}_{F}
   \end{pmatrix}\,.
\end{equation*}
\\
Upon simplification, the expression for equation \eqref{GHdiff}, at $(m,f)\in U_{\alpha\beta}\times F$, can be represented as: 
\begin{equation}\label{eqA}
  (\psi_{\alpha\beta})_{*(m,f)}\circ J_{11}=J_{11}\circ(\psi_{\alpha\beta})_{*(m,f)}\,,  
\end{equation}
\begin{equation}\label{eqB}
   (\psi_{\alpha\beta})_{*(m,f)}\circ J_{12}=J_{12}\circ (\psi_{\beta\alpha})^{*}_{(m,f)}\,, 
\end{equation}
\begin{equation}\label{eqC}
   (\psi_{\beta\alpha})^{*}_{(m,f)}\circ J_{21}=J_{21}\circ(\psi_{\alpha\beta})_{*(m,f)}\,,  
\end{equation}
\begin{equation}\label{eqD}
  (\psi_{\beta\alpha})^{*}_{(m,f)}\circ J_{22}=J_{22}\circ(\psi_{\beta\alpha})^{*}_{(m,f)}\,.   
\end{equation}
Since $\psi_{\alpha\beta}(m,f)=(m,\phi_{\alpha\beta}(m)\cdot f)$ where $\phi_{\alpha\beta}(m)\in G$, the map 
$$(\psi_{\alpha\beta})_{*(m,f)}: T_{m}U_{\alpha\beta}\oplus T_{f}F\longrightarrow T_{m}U_{\alpha\beta}\oplus T_{\phi_{\alpha\beta}(m)\cdot f}F$$ can be expressed as 
\begin{equation}\label{eq1}
   (\psi_{\alpha\beta})_{*(m,f)}=
   \begin{pmatrix}
       Id_{U_{\alpha\beta}}   &0\\
       (\rho_{f})_{*}\circ(\phi_{\alpha\beta})_{*m}    &(\phi_{\alpha\beta}(m))_{*}
   \end{pmatrix}\,,
\end{equation}
and  the map 
$$(\psi_{\beta\alpha})^{*}_{(m,f)}: T^{*}_{m}U_{\alpha\beta}\oplus T^{*}_{f}F\longrightarrow T^{*}_{m}U_{\alpha\beta}\oplus T^{*}_{\phi_{\alpha\beta}(m)\cdot f}F$$ can be expressed as 
\begin{equation}\label{eq2}
 (\psi_{\beta\alpha})^{*}_{(m,f)}=
   \begin{pmatrix}
       Id_{U_{\alpha\beta}}   &(\phi_{\beta\alpha})^{*}_{m}\circ\rho^{*}_{\phi_{\alpha\beta}(m)\cdot f}\\
       0    &(\phi_{\beta\alpha}(m))^{*} 
   \end{pmatrix} \, .
\end{equation}
From equations \eqref{eqA} and \eqref{eq1}, we have
\begin{equation}\label{eq3}
   (\phi_{\alpha\beta}(m))_{*}\circ (-J_{F})=(-J_{F})\circ(\phi_{\alpha\beta}(m))_{*} 
\end{equation}
and
\begin{equation}\label{eq4}
(\rho_{f})_{*}\circ(\phi_{\alpha\beta})_{*m}\circ (-J_{U_{\alpha\beta}})=(-J_{F})\circ(\rho_{f})_{*}\circ(\phi_{\alpha\beta})_{*m}\,. 
\end{equation}
 Using equations \eqref{eqB}, \eqref{eq1} and \eqref{eq2}, we get 
\begin{equation}\label{eq5}
 (\phi_{\alpha\beta}(m))_{*}\circ\beta_{F}=\beta_{F}\circ(\phi_{\beta\alpha}(m))^{*}\,,   
\end{equation}
\begin{equation}\label{eq6}
(\rho_{f})_{*}\circ(\phi_{\alpha\beta})_{*m}\circ\beta_{U_{\alpha\beta}}=0\,,
\end{equation}
and 
\begin{equation}\label{eq7} \beta_{U_{\alpha\beta}}\circ(\phi_{\beta\alpha})^{*}_{m}\circ\rho^{*}_{\phi_{\alpha\beta}(m)\cdot f}=0\,.  
\end{equation}
From equations \eqref{eqC}, \eqref{eq1} and \eqref{eq2}, we have
\begin{equation}\label{eq8}
(\phi_{\beta\alpha}(m))^{*}\circ B_{F}=B_{F}\circ(\phi_{\alpha\beta}(m))_{*}\,,    
\end{equation}
\begin{equation}\label{eq9} (\phi_{\beta\alpha})^{*}_{m}\circ\rho^{*}_{\phi_{\alpha\beta}(m)\cdot f}\circ B_{F}=0\,,  
\end{equation}
and
\begin{equation}\label{eq10} B_{F}\circ(\rho_{f})_{*}\circ(\phi_{\alpha\beta})_{*m}=0\,. 
\end{equation}
From equations \eqref{eqD} and \eqref{eq2}, we have
\begin{equation}\label{eq11}
  (\phi_{\beta\alpha}(m))^{*}\circ J^{*}_{F}=J^{*}_{F}\circ (\phi_{\beta\alpha}(m))^{*}
\end{equation}
and
\begin{equation}\label{eq12} (\phi_{\beta\alpha})^{*}_{m}\circ\rho^{*}_{\phi_{\alpha\beta}(m)\cdot f}\circ J^{*}_{F}=J^{*}_{U_{\alpha\beta}}\circ(\phi_{\beta\alpha})^{*}_{m}\circ\rho^{*}_{\phi_{\alpha\beta}(m)\cdot f}\,.
\end{equation}
Now, we can see that equations \eqref{eq3}, \eqref{eq5}, \eqref{eq8} and \eqref{eq11} hold if and only if 
\begin{equation}\label{eq13}
\begin{pmatrix}
(\phi_{\alpha\beta}(m))_{*}    &0 \\
    0    &(\phi_{\beta\alpha}(m))^{*}
\end{pmatrix}
\circ\mathcal{J}_{F}=\mathcal{J}_{F}\circ
\begin{pmatrix}
(\phi_{\alpha\beta}(m))_{*}    &0 \\
    0    &(\phi_{\beta\alpha}(m))^{*}
\end{pmatrix}
\end{equation}
where $\mathcal{J}_{F}=
\begin{pmatrix}
    -J_{F}   &\beta_{F}
    \\
    B_{F}    &J^{*}_{F}
\end{pmatrix}$  as in \eqref{matix}.
Therefore, equations \eqref{eq3}, \eqref{eq5}, \eqref{eq8} and \eqref{eq11} hold if and only if $\phi_{\beta\alpha}(m)\in\diff_{\mathcal{J}_{F}}(F)\,.$ Note that, by skew-symmetry, $\beta^{*}_{U_{\alpha\beta}}=-\beta_{U_{\alpha\beta}}$ and $B^{*}_{F}=-B_{F}$. Since $(m,f)$ is arbitrary, considering duals, we observe that

\vspace{0.3em}
\begin{itemize}
\setlength\itemsep{1em}
    \item equation \eqref{eq6}  holds if and only if equation \eqref{eq7} holds, 
    \item equation \eqref{eq10} holds if and only if equation  \eqref{eq9} holds, 
    \item equations \eqref{eq4} and \eqref{eq12} are equivalent to each other.
\end{itemize}
Therefore, $\psi_{\alpha\beta}$ is a GH automorphism if and only if equations \eqref{eq4}, \eqref{eq6},
\eqref{eq10} and \eqref{eq13} hold. Hence, $\psi_{\alpha\beta}$ is a GH automorphism if and only if $(1)$ and $(2)$ are satisfied as desired.
\end{proof}

\begin{definition}
Let $E$ be an SGH fiber bundle over a GC manifold $(M,\mathcal{J}_{M})$ and $U\subseteq M$ be an open set. A smooth section $s:U\longrightarrow E$ is called a GH section if $s$ is a GH map from $(U,\mathcal{J}_{M}|_{U})$ to $(E,\mathcal{J}_{E})$.    
\end{definition}
\begin{definition}
Given any two SGH fiber bundles $E$ and $E^{'}$ over $M$, a smooth map $\phi:E\longrightarrow E^{'}$ is called an SGH bundle homomorphism if
\begin{enumerate}
\setlength\itemsep{1em}
    \item $\phi$ is a bundle homomorphism between $E$ and $E^{'}$ as smooth fiber bundles.
    \item $\phi$ is a GH map.
\end{enumerate} If, in addition, $\phi$ is a GH homeomorphism, then $\phi:E\longrightarrow E^{'}$ is called SGH bundle isomorphism.
\end{definition}

\begin{example}
Let $E$ be an SGH fiber bundle over a GC manifold $(M,\mathcal{J}_{M})$ with typical fiber $(F,\omega)$, a symplectic manifold. Consider the GC structure on $F$ induced by $\omega$ as given in Example \ref{symplectic eg}. Note that $J_{F}=J^{*}_{F}=0$ and $B_{F}=-\beta^{-1}_{F}=\omega$. Since $\omega$ is non-degenerate and $f\in F$ is arbitrary, for each $m\in M$, equation \eqref{eq10} holds if and only if  $(\phi_{\alpha\beta})_{*m}=0$,  i.e., $\phi_{\alpha\beta}$ is a locally constant map on $U_{\alpha\beta}$. From equation \eqref{eq8}, for $X,Y\in TM$, we have 
\begin{align*}
  &\omega=(\phi_{\alpha\beta}(m))^{*}\circ\omega \circ(\phi_{\alpha\beta}(m))_{*}\,\\ 
\implies&\omega(X,Y)=(\phi_{\alpha\beta}(m))^{*}\omega(X,Y)\,.
\end{align*}

\begin{lemma}\label{sym lmma1}
Any SGH fiber bundle $E$ over a GC manifold $M$ with a symplectic fiber $(F,\omega)$ is a smooth symplectic fiber bundle with a flat connection.
\end{lemma}
\end{example}
\begin{example}
   Let $E$ be an SGH fiber bundle over a symplectic manifold $(M,\omega)$ with a typical fiber $(F,\mathcal{J}_{F})$. Consider the GC structure on the base induced by $\omega$ as given in Example \ref{symplectic eg}. Note that $J_{M}=J^{*}_{M}=0$ and $B_{M}=-\beta^{-1}_{M}=\omega$. Since $\omega$ is nondegenarate, by equation \eqref{eq6},  for any $(m,f)\in U_{\alpha\beta}\times F$, $(\rho_{f})_{*}\circ(\phi_{\alpha\beta})_{*m}=0$. Thus, $(\phi_{\alpha\beta})_{*m}=0$, i.e., $\phi_{\alpha\beta}$ is locally constant on $U_{\alpha\beta}$. So, equations \eqref{eq4} and \eqref{eq10} are also satisfied. Hence, we have the following.
\begin{lemma}\label{sym lmma2}
  $E$ be a smooth fiber bundle over a symplectic manifold $(M,\omega)$ with a typical fiber $(F,\mathcal{J}_{F})$. Then, $E$ is SGH fiber bundle over $M$ with local trivializations $\{U_{\alpha},\phi_{\alpha}\}$ if and only if

  \vspace{0.3em}
\begin{enumerate}
\setlength\itemsep{1em}
    \item $\phi_{\alpha\beta}(m)\in\diff_{\mathcal{J}_{F}}(F)$ for all $m\in U_{\alpha\beta}$,
    \item $\phi_{\alpha\beta}$ is constant on $U_{\alpha\beta}$, that is, $E$ admits a flat connection.
\end{enumerate}
\end{lemma}
\end{example}

\begin{prop}\label{gh map equi2}
  Let $(M,\mathcal{J}_{M})$ be a GC manifold and $(N,J_{N})$ be a complex manifold with a complex structure $J_{N}$. Then, for any smooth map $\psi:M\longrightarrow N$, the following are equivalent.

  \vspace{0.3em}
  \begin{enumerate}
  \setlength\itemsep{1em}
      \item $\psi$ is a GH map.
      \item For any open set $U\subset M$, $\psi: U\longrightarrow N$ is a GH map. 
      \item $\psi_{*}\circ J_{M}=J_{N}\circ \psi_{*}\,, \quad \psi_{*}\circ \beta_{M}=0 \,.$
  \end{enumerate}
  Here, $\mathcal{J}_{M}$ is as in \eqref{matix} and $N$ is considered as a GC manifold with the natural GC structure induced by $J_{N}$.
\end{prop}
\begin{proof}
    Follows from Lemma \ref{gh map equi}.
\end{proof}
\begin{example}
{Let $E$ be an SGH fiber bundle over a GC manifold $(M,\mathcal{J}_{M})$ with typical fiber a complex manifold $(F,J_{F})$. Then, consider the naturally induced GC structure on $F$ as given in Example \ref{complx eg}. Observe that $B_{F}=\beta_{F}=0$. Also, by equation \eqref{eq13}, for any $m\in U_{\alpha\beta}$, $\phi_{\alpha\beta}(m)$ is a biholomorphic automorphism. By Proposition \ref{gh map equi2} and equation \eqref{eq6}, for any $f\in F$,  $\rho_{f}\circ\phi_{\alpha\beta}$ is a GH map. Thus we have the following result.}
\begin{lemma}\label{cmplx fiber}
    Let $E$ be a smooth fiber bundle over a GC manifold $(M,\mathcal{J}_{M})$ with typical fiber a complex manifold $(F,J_{F})$. Let $\{U_{\alpha},\phi_{\alpha}\}$ be a family of local trivializations of $E$. Then, $E$ is an SGH fiber bundle over $M$ if and only if
    \begin{enumerate}
\setlength\itemsep{1em}
        \item for each $m\in U_{\alpha\beta}$, $\phi_{\alpha\beta}(m)$ is a biholomorphic map, 
        \item for any $f\in F$,
       $\rho_{f}\circ\phi_{\alpha\beta}$ is a GH map.
    \end{enumerate}
In particular, when $M$ is a complex manifold and the GCS $\mathcal{J}_{M}$ is induced by the complex structure (cf. Example \ref{complx eg}), $E$ is an SGH fiber bundle if and only if $E$ is a holomorphic fiber bundle over $M$.
\end{lemma}
\end{example}
\begin{remark}\label{VB rmk}
    Note that when $E$ denotes a vector bundle over a GC manifold $M$, by Lemma \ref{cmplx fiber}, $E$ is an SGH vector bundle if and only if it is a GH vector bundle in the sense described by Lang et al. (\cite[Definition 3.1]{lang2023}).
\end{remark}
\begin{example}\label{imp exmple}
    Let $M$ be a GC manifold and $\tilde{M}$ be a covering space. Let $K\leq\pi_1(M)$ be a subgroup corresponding to $\tilde{M}$ such that $\tilde{M}/K\cong M$. Note that $K\hookrightarrow\tilde{M}\xrightarrow{\pi}M$ is a principal $K$-bundle where $\pi$ is the covering map. Since $\pi$ is a local diffeomorphism, $M$ induces a GC structure (of the same type) on $\tilde{M}$ such that $\pi$ becomes a GC map. Let $\rho:K\longrightarrow GL_{l}(\C)$ be a representation. Define 
    $$\tilde{M}\times_{\rho}\C^{l}:=\tilde{M}\times\C^{l}/\sim\,,$$ where $(m,z)\sim (n,w)$ if and only if $n=m\cdot g^{-1}$ and $w=\rho(g)\cdot z$ for some $g\in K$. Since $K$ is discrete, the transition maps of the associated vector bundle $\tilde{M}\times_{\rho}\C^{l}\longrightarrow M$ are locally constant. Hence, by Lemma \ref{cmplx fiber}, $\tilde{M}\times_{\rho}\C^{l}\longrightarrow M$ is a (flat) SGH vector bundle over $M$.
\end{example}
\begin{example}\label{imp exmple2}
    Let $M_1$ be a complex manifold and $V_1$ be a holomorphic vector bundle over $M_1$. Let $M_2$ be a symplectic manifold and $V_2$ be a flat vector bundle over $M_2$. Then $\otimes_{i}\pr^{*}_{i}(V_{i})\longrightarrow M_1\times M_2$ is an SGH vector bundle where $\pr_{i}:M_1\times M_2\longrightarrow M_{i}$ is the natural projection map onto $i$-th component. Here $M_1\times M_2$ is considered with the product GCS.
\end{example}

\section{ SGH principal bundles, SGH vector bundles, and locally free sheaves}\label{sec4}
\subsection{SGH principal bundles and SGH vector bundles}
Let $G$ be a real connected Lie group with a GC structure $\mathcal{J}_{G}$. Let $G\hookrightarrow P\xrightarrow{\pi} M$ be a smooth principal $G$-bundle over a GC manifold $(M,\mathcal{J}_{M})$. Let $\{U_{\alpha},\phi_{\alpha}\}$ be a family of local trivializations 
\begin{equation}\label{phi alpha}
    \phi_{\alpha}:\pi^{-1}(U_{\alpha})\longrightarrow U_{\alpha}\times G\,,
\end{equation}
with transition functions 
\begin{equation}\label{phi alpha beta}
\phi_{\alpha\beta}:U_{\alpha\beta}=U_{\alpha}\cap U_{\beta}\longrightarrow G\,,    
\end{equation}
where $\phi_{\alpha\beta}(x)=\phi_{\alpha}|_{\pi^{-1}(x)}\circ\phi^{-1}_{\beta}(x,\cdot)$ for all $x\in U_{\alpha\beta}\,.$
\begin{definition}
  A SGH principal $G$-bundle $P$ is a principal $G$-bundle over $(M,\mathcal{J}_{M})$ which, as a fiber bundle, is SGH.
\end{definition}
As in \eqref{matix}, $\mathcal{J}_{M}$ and $\mathcal{J}_{G}$ can be written in the following form
$$\mathcal{J}_{M}=
    \begin{pmatrix}
    -J_{M}     &\beta_{M} \\
    
    B_{M}         &J^{*}_{M}
    \end{pmatrix}\quad\text{and}\quad
    \mathcal{J}_{G}=
    \begin{pmatrix}
    -J_{G}     &\beta_{G} \\
    
    B_{G}         &J^{*}_{G}
    \end{pmatrix}\,,\,\text{respectively}\,.$$
    \begin{remark}\label{rem:SGH principal bdl}
        Note that in the definition of an SGH principal $G$-bundle, we do not require that the 
        group operations on $G$ be GH maps, or that the left or right translations by elements of $G$ be GH homeomorphisms. However, if we assume that $G$ is a complex Lie group then these conditions hold.
    \end{remark}
    
\begin{prop}\label{GH principal bundle1} The following are equivalent.
 \begin{enumerate}
     \item $P$ is an SGH principal $G$-bundle over $(M, \mathcal{J}_{M} )$ with  local trivializations $\{U_{\alpha},\phi_{\alpha}\}$ and transition functions $\{ \phi_{\alpha \beta} \}$.
     
  \item  For all nonempty $U_{\alpha\beta}\subseteq M$ and $(m,f)\in U_{\alpha\beta}\times G\,,$ the map
  $$\psi_{\alpha\beta}:U_{\alpha\beta}\times G\longrightarrow U_{\alpha\beta}\times G\quad\text{defined as}\quad\psi_{\alpha\beta}(m,f)=(m,\phi_{\alpha\beta}(m)\cdot f)$$
  is a GH automorphism of $U_{\alpha\beta}\times G$.
  
  \item  The transition functions satisfy the following: For all $m\in U_{\alpha\beta}$,

  \vspace{0.3em}
  \begin{enumerate}
  \setlength\itemsep{1em}
      \item $\phi_{\alpha\beta}(m)\in\diff_{\mathcal{J}_{G}}(G)\,,$
      \item $(\phi_{\alpha\beta})_{*m}\circ J_{U_{\alpha\beta}}=J_{F}\circ(\phi_{\alpha\beta})_{*m}\,,$
      \item $(\phi_{\alpha\beta})_{*m}\circ\beta_{U_{\alpha\beta}}=0\,,$
      \item $B_{G}\circ(\phi_{\alpha\beta})_{*m}=0\,.$
  \end{enumerate} 
\end{enumerate}
\end{prop}
\begin{proof}
    Follows from Theorem \ref{main1}.
\end{proof}
\begin{prop}\label{GH principal bundle}
Let $G$ be a (connected) complex Lie group. Then, the following are equivalent.
 \begin{enumerate}
 \setlength\itemsep{1em}
     \item $P$ is an SGH principal $G$-bundle over $(M, \mathcal{J}_{M} )$ with  local trivializations $\{U_{\alpha},\phi_{\alpha}\}$ and transition functions $\{ \phi_{\alpha \beta} \}$.
 \item The transition maps $\phi_{\alpha\beta}$   satisfy the following:

  \vspace{0.3em}
  \begin{enumerate}
  \setlength\itemsep{1em}
      \item $\phi_{\alpha\beta}(m)$ is a biholomorphic map on $G$ $\forall\,\,m\in U_{\alpha\beta}$,
      \item each $\phi_{\alpha\beta}$  is a GH map.

      \end{enumerate}
  \end{enumerate} 
\end{prop}

\begin{proof}
 Follows from Proposition \ref{GH principal bundle1} and Lemma \ref{cmplx fiber}.   
\end{proof}
Let $M$ be a GC manifold and let $E$ be an SGH vector bundle of real rank $2l$ over $M$ with local trivializations $\{U_{\alpha}\,,\,\phi_{\alpha}\}$. Then, by Theorem \ref{main1} and \cite[Proposition 3.2]{lang2023}, we have 
\begin{enumerate}
\setlength\itemsep{1em}
    \item $\phi_{\alpha\beta}(m)\in GL_{l}(\C)$, i.e., $E$ is a complex vector bundle of of complex rank $l$, 
    \item each entry $B_{\lambda\gamma}: U_{\alpha\beta}\longrightarrow\C$ of $\phi_{\alpha\beta}=(B_{\lambda\gamma})_{l\times l}$ is a GH function, 
\end{enumerate}
where $\phi_{\alpha\beta}:U_{\alpha\beta}\longrightarrow GL_{2l}(\R)$ is the transition map as in  Theorem \ref{main1}. {Following the standard associated principal bundle construction
(cf. \cite[Chapter 3]{steenrod57}), we define the principal bundle, denoted by $P_{E}$, associated to $E$ as 
\begin{equation}\label{p-e} P_{E}:=\bigsqcup_{\alpha}U_{\alpha}\times GL_{l}(\C)\Big/\sim\,,
\end{equation} where $\sim$ is an equivalence relation, given by $$(b,h)\sim (a,g) \iff a=b\,\,\,\text{and}\,\,\,g=\phi_{\alpha\beta}(b)h$$ for $(b,h)\in U_{\beta}\times GL_{l}(\C)$ and  $(a,g)\in U_{\alpha}\times GL_{l}(\C)$. Considering $GL_{l}(\C)\subset GL_{2l}(\R)$, note that the transition map 
  $\phi_{\alpha\beta}:U_{\alpha\beta}\longrightarrow GL_{l}(\C)$ is a GH map if and only if each entry of $\phi_{\alpha\beta}$ is a GH function. Hence, by Proposition \ref{GH principal bundle}, $P_{E}$ is an SGH principal $GL_{l}(\C)$-bundle.}
\medskip

Given an SGH principal $GL_{l}(\C)$-bundle $P$ over $M$ with local trivializations $\{U_{\alpha},\phi_{\alpha}\}$, the associated vector bundle $E_P$ is defined as the identification space of the right action, that is, 
\begin{equation}\label{e-p}
    E_{P}:=P\times_{GL_{l}(\C)}\C^{l}\,,
\end{equation} where the right action of $GL_l(\mathbb{C})$ on $P\times\C^{l}$ is defined by
$$(p,f)\cdot g=(p\cdot g,g^{-1}( f))\,\,\,\forall\,\,\,p\in P\,,\,f\in\C^{l}\,\,\text{and}\,\,g\in GL_{l}(\C)\,.$$ Note that the transition map $\phi_{\alpha\beta}:U_{\alpha\beta}\longrightarrow GL_{l}(\C)$ of $E_{P}$, as in Theorem \ref{main1}, is a GH map by Proposition \ref{GH principal bundle}. This implies that each entry of $\phi_{\alpha\beta}$ is a GH function. Thus, by \cite [Proposition 3.2]{lang2023}, $E_{P}$ is an SGH vector bundle over $M$. 
The result below now follows using standard arguments.

\begin{prop}\label{G-V prop}
Let $(M,\mathcal{J}_{M})$ be a GC manifold and $l\in\N$. Consider the following two sets 
\begin{align*}
    \mathscr{E}_{l}:=&\,\text{Set of all isomorphism classes of SGH vector bundles of real rank $2l$}\\
    &\text{over $M$}\,,
\end{align*}
and 
\begin{align*}
   \mathscr{P}_{GL_{l}(\C)}:=&\,\text{Set of all  isomorphism classes of SGH principal}\\
   &GL_{l}(\C)\text{-bundles over $M$}\,. 
 \end{align*} 
If $P_{E}$ and $E_{P}$ are as in the equations \eqref{p-e} and \eqref{e-p} respectively, then the map 
\begin{equation}\label{category map}
    \Phi: \mathscr{E}_{l}\longrightarrow\mathscr{P}_{GL_{l}(\C)}
\end{equation} defined by $\Phi([E])=[P_{E}]$ gives a bijective map between the two sets with the inverse map defined as $\Phi^{-1}([P])=[E_{P}]$ where $[E]$ and $[P]$ denote the SGH bundle isomorphism classes of $E$ and $P$, respectively.
\end{prop}

\subsection{SGH vector bundles and locally free sheaves of finite rank}
{Let $M$ be a GC manifold. Let $E$ be an SGH vector bundle of real rank $2l$ over $M$. Consider the sheaf $\mathbf{E}$ of GH sections of $E$ over $M$, that is, for any open set $U\subseteq M$,
$$\Gamma(U,\mathbf{E}):=\{s\in C^{\infty}(U,E)\,|\,s\,\text{is a GH section of $E$ over $U$}\}\,.$$ Note that $\mathbf{E}$ is a sheaf of $\mathcal{O}_{M}$-modules. On a trivializing neighborhood $U$, $$E|_{U}\cong U\times\C^{l}\,,$$ so that $\Gamma(U,\mathbf{E})\cong\bigoplus_{l}\mathcal{O}_{M}(U)$\,. This implies that $\mathbf{E}$ is a locally free sheaf of complex rank $l$ over $M$. (We will henceforth follow the convention of denoting the sheaf of GH sections of an SGH vector bundle by the corresponding bold letter.)}
\vspace{0.5em}

Conversely, given any locally free sheaf $\mathcal{F}$ of $\mathcal{O}_{M}$-modules of rank $l$, one associates an SGH vector bundle $E_{\mathcal{F}}$ over $M$ of real rank $2l$ to it using standard arguments, such that $$\mathbf{E_{\mathcal{F}}}\cong \mathcal{F}\,\,\,\,\text{as $\mathcal{O}_{M}$-modules}\,.$$
Hence, we get the following.
\begin{prop}\label{S-V prop}
Let $M$ be a GC manifold and $l\in\N$. Consider the following set
\begin{align*}
   \mathscr{S}_{l}:=\text{Set of all isomorphism class of locally free $\mathcal{O}_{M}$-modules of complex rank $l$}\,. 
 \end{align*} Then the association $E\longrightarrow\mathbf{E}$ induces an one to one correspondence between $ \mathscr{E}_{l}$ and $\mathscr{S}_{l}$ where $\mathscr{E}_{l}$ as given in Proposition \ref{G-V prop}. The inverse map is given by the association $\mathcal{F}\longrightarrow E_{\mathcal{F}}$\,.
\end{prop}

\section{Generalized Holomorphic Connection on SGH Principal bundle}\label{sec gh pb}
\subsection{GH tangent and GH cotangent bundle}
Let $(M,\mathcal{J}_{M})$ be a regular GC manifold of type $k\in\N\cup \{0\}$. Let $L_{M}$ and $\overline{L_{M}}$ are its corresponding $+i$ and $-i$-eigenspace sub-bundles of $(TM\oplus T^{*}M)\otimes\C$ respectively.
 Define
\begin{equation}\label{GH cotangent bundle}
    \mathcal{G}^{*}M:= L_{M}\cap (T^{*}M\otimes\C)\,.
\end{equation}
By \cite[Proposition 3.13]{lang2023}, $\mathcal{G}^{*}M$ is an SGH vector bundle over $M$. It is called the \textit{GH cotangent bundle}. The GH sections of $\mathcal{G}^{*}M$ are called GH $1$-forms. Since $\mathcal{G}^{*}M$ is $B$-field transformation invariant, locally (cf. \eqref{loc coordi}, \eqref{eq:E,Ebar}), the space of GH $1$-forms is of the form
\begin{equation*}
 \spn_{\mathcal{O}_{U}}\{ dz_1\cdots dz_{k}\} \,.
\end{equation*}  
This shows that $\mathbf{\mathcal{G}^{*}M}$, the sheaf of GH sections of
$\mathcal{G}^{*}M$, is a locally free sheaf of $\mathcal{O}_{M}$-modules of finite rank $k$. Define 
 \begin{equation}\label{GH tangent sheaf} 
\mathbf{\mathcal{G}M}:=\Hom_{\mathcal{O}_{M}}(\mathbf{\mathcal{G}^{*}M},\mathcal{O}_{M})\,.
\end{equation}
Since $\mathbf{\mathcal{G}^{*}M}$ is a locally free sheaf of $\mathcal{O}_{M}$-modules of rank $k$, $\mathbf{\mathcal{G}M}$ will also be a locally free sheaf with the same rank. Then, by Proposition \ref{S-V prop}, the corresponding SGH vector bundle is defined as
\begin{equation}\label{GH tangent bundle} 
\mathcal{G}M:=E_{\mathbf{\mathcal{G}M}}\,.
 \end{equation}
 Here $\mathcal{G}M$ as given in Proposition \ref{S-V prop}. It is called the \textit{GH tangent bundle} of $M$. The GH sections of $\mathcal{G}M$ are called GH vector fields. Since $\mathcal{G}^{*}M$ is $B$-transformation invariant, $\mathcal{G}M$ is also invariant under $B$-transformation.  Thus, locally (cf. \eqref{loc coordi}), the space of GH vector fields is of the form
 \begin{equation*}
     \spn_{\mathcal{O}_{U}}\left\{\frac{\partial}{\partial z_1},\cdots,\frac{\partial}{\partial z_{k}}\right\}
 \end{equation*}
 Note that $\mathcal{G}M$ and $\mathcal{G}^{*}M$ are dual to each other as $\mathcal{O}_{M}$-modules of their GH sections. But, we can say more. Observe that $C^{\infty}(\mathcal{G}^{*}M)=\Hom_{\mathcal{O}_M}(\mathbf{\mathcal{G}M} , \mathcal{O}_M)\otimes_{\mathcal{O}_{M}} C^{\infty}_{M}$\,. {Then, by \cite[Proposition 7, Section 5, Chapter II]{bourbaki89}, we have
 \begin{equation}\label{iso GM-G*M}
\begin{aligned}
C^{\infty}(\mathcal{G}^{*}M)
\cong\Hom_{( \mathcal{O}_{M}\otimes_{\mathcal{O}_{M}} C^{\infty}_{M})}(\mathbf{\mathcal{G}M}\otimes_{\mathcal{O}_{M}} C^{\infty}_{M}\,,\,\mathcal{O}_{M}\otimes_{\mathcal{O}_{M}} C^{\infty}_{M})\,
\cong C^{\infty}((\mathcal{G}M)^{*})\,,
 \end{aligned}     
 \end{equation} that is, $\mathcal{G}M$ and $\mathcal{G}^{*}M$ are also dual to each other as $C^{\infty}_{M}$-modules of their smooth sections.}
\begin{remark} We note that the definition of 
$\mathcal{G}M$, given in  \cite[p.16]{lang2023}, as 
$\mathcal{G}M:= \overline{L} \cap TM \otimes \mathbb{C}$  is unsatisfactory as it varies with $B$-transformations. In other words, it is not always the case that $\mathcal{G}M$ and $e^{B}(\overline{L}_{M})\cap TM \otimes \mathbb{C}$ are same, while $\mathcal{G}^{*}M=e^{B}(L_{M})\cap T^{*}M \otimes \mathbb{C}$ for any $B$-transformation. Therefore, this does not guarantee duality with respect to $\mathcal{G}^{\ast}M$.
\end{remark}

\subsection{SGH principal bundles with complex fibers and GH connections}
 
There are some special properties of SGH principal bundles with a complex Lie group as a structure group which we similar to holomorphic principal bundles over complex manifolds. These properties do not hold in general. We list a few of them here that are important for our purposes.  

\begin{prop}\label{prop:SGH special}
Let $G\hookrightarrow P\xrightarrow{\pi} M$ be an SGH principal $G$-bundle over a regular GC manifold $(M,\mathcal{J}_M)$ where $G$ is a complex Lie group. Then

\vspace{0.2em}
\begin{enumerate}
\setlength\itemsep{1em}
\item $P$ admits GH sections over any  trivializing open set $U \subseteq M$. 
\item If $s:V \to P$ is a GH section of $P$ over an open subset $V \subseteq M$, then so is $  s \cdot \phi $ for any GH map $\phi: V \to G$. 
\item If $s_1$ and $s_2$ are any two GH sections of $P$ over $V$, then there exists a unique $GH$ map $\phi: V \to G $ such that $s_2 =   s_1  \cdot  \phi \,.$
\end{enumerate}
    
\end{prop}

 \begin{proof} 
 Follows from Proposition \ref{gh map equi2}.
 \end{proof}

\medskip
By Remark \ref{rem:SGH principal bdl}, $G$ acts on $P$ as a group of fiber preserving GH automorphisms, $P\times G \longrightarrow P\,$. The GCS induced by the complex structure on $G$ is regular, which implies that $P$ is a regular GC manifold. Let $\mathcal{G}^{*}P$ and 
 $\mathcal{G}P$ denote the GH cotangent and GH tangent bundles of $G$ as specified in \eqref{GH cotangent bundle} and \eqref{GH tangent bundle}, respectively. Since $G$ acts on $P$, it has an induced action on $(TP\oplus T^{*}P)\otimes\C$. For any $g\in G$, we have
\[(X+\xi)\cdot g =
\begin{pmatrix}
    g^{-1}_{*}    &0\\
    0        &g^{*}
\end{pmatrix} (X+\xi)\quad\text{for all}\quad X+\xi\in(TP\oplus T^{*}P)\otimes\C\,.
\]
 As $g:P\longrightarrow P\,,\,p\mapsto p\cdot g$, is a GH automorphism for every $g\in G$, it follows that  $G$ acts on $i$-eigen bundle $L_{P}$ of $\mathcal{J}_P$. This implies that $G$ acts on $\mathcal{G}^{*}P$ and hence on $\mathcal{G}P$. Define the {\em SGH Atiyah bundle} of $P$ by 
\begin{equation}\label{Q}
    At(P):=\mathcal{G}P/G\,.
\end{equation} Then, a point of $At(P)$ is a $G$-invariant field of GH tangent vectors, defined along one of the fibers of $P$. We shall show that $At(P)$ has a natural SGH vector bundle structure over $M$.

\medskip
Let $m\in M$ and let $U\subset M$ be a sufficiently small open neighborhood of $m$ such that there exist a GH section of $P$ over $U$, 
\begin{equation}\label{section}
    s:U\longrightarrow P\,.
\end{equation}
Let $(\mathcal{G}P)_{s}$ be the restriction of $\mathcal{G}P$ to $s(U)$. Now since $s: U\longrightarrow s(U)$ is a diffeomorphism, $s(U)$ can be endowed with the structure of a regular GC manifold such that $s$ becomes a GH homeomorphism between $U$ and $s(U)$. Since $s$ is a GH section, by \cite[Example 3.3]{lang2023}, $s^{*}(\mathcal{G}P)$ is an SGH vector bundle over $U$ and so, $(s^{-1})^{*}(s^{*}(\mathcal{G}P))$ is also an SGH vector bundle over $s(U)$ which coincides with $(\mathcal{G}P)_{s}$ as a smooth bundle. This defines a canonical SGH bundle structure on $(\mathcal{G}P)_{s}$.

\medskip
There is a natural one-to-one correspondence between $At(P)_{U}$ and $(\mathcal{G}P)_{s}$, 
\begin{equation}\label{gamma} \gamma_{s}:At(P)_{U}\longrightarrow(\mathcal{G}P)_{s}\,,
\end{equation}
where $\gamma_{s}$ assigns to each invariant GH vector field along $\pi^{-1}(x):=P_{x}$ its value at $s(x)$. This is easily seen to be an isomorphism of smooth vector bundles. Then, the SGH vector bundle structure of $(\mathcal{G}P)_{s}$ defines an SGH vector bundle structure of $At(P)_{U}$.

\medskip
It remains to show that this construction is independent of the choice of the GH section $s$. Let $s_1$ and $s_2$ be any two GH sections of $P$ over $U$. Then, by  Proposition \ref{prop:SGH special},  there exist a unique GH map $\phi:U\longrightarrow G$ such that 
$$s_{1}(x)\cdot\phi(x)=s_2(x)\,,\,\,\,\,\,\forall\,\,\,x\in U\,.$$
Note that the map $\psi:U\times G\longrightarrow U\times G$ defined as 
  $$\psi(x,g)=(x,\phi(x)\cdot g)\quad\text{for all $(x,g)\in U\times G$}$$
  is a GH automorphism of $U\times G$ by Proposition \ref{GH principal bundle}. Therefore, $\psi$ induces an isomorphism of SGH vector bundles, again denoted by $\psi\,,$
$$\psi:(\mathcal{G}P)_{s_1}\longrightarrow(\mathcal{G}P)_{s_2}\,,$$ satisfying
$$ \gamma_{s_2} = \psi \circ \gamma_{s_1}   \, .$$
Hence, the SGH vector bundle structure on $At(P)$ is well-defined.

\medskip
Let $\mathcal{T}$ denote the sub-bundle of $TP$ formed by vectors tangential to the fibers of $P$. Define $\mathcal{GT}=\mathcal{T}\cap \mathcal{G}P$. Since $G$ acts on $\mathcal{T}$, it also acts on $\mathcal{GT}$. Define
\begin{equation}\label{R}
    \mathcal{R}=\mathcal{GT}/G\,.
\end{equation}
 If $\gamma_{s}$ is defined as in \eqref{gamma}, then restricting to $\mathcal{R}_{U}$, we get 
 \begin{equation}\label{gamma'}
\gamma_{s}|_{\mathcal{R}_{U}}:=\gamma^{'}_{s}:\mathcal{R}_{U}\longrightarrow(\mathcal{GT})_{s}\,.
 \end{equation}
Note that $(\mathcal{GT})_{s}$ is also an SGH vector sub-bundle of $(\mathcal{G}P)_{s}$ as $(\mathcal{GT})$ is an SGH vector sub-bundle of $(\mathcal{G}P)$. Hence by above, $\mathcal{R}$ is also an SGH sub-bundle of $At(P)$.

\medskip
We now examine $\mathcal{R}$ more closely. Let $\mathfrak{g}$ denote the complex Lie algebra of $G$. As a vector space, $\mathfrak{g}$ is the holomorphic tangent space of $G$ at identity. In the SGH principal bundle $P$, for $x\in M$, each fiber $P_{x}$ can be identified with $G$ up to a left multiplication. Note that each smooth tangent vector at the point $p\in P$, tangential to the fiber, defines a unique left-invariant smooth vector field on $G$. Since the left (respectively, right) multiplication is biholomorphic, the vector space of left (respectively, right) invariant holomorphic vector fields on $G$ is then isomorphic with $\mathfrak{g}$ via left (respectively, right) multiplication. Note that by the locally product nature of the GCS on $P$, and the absence of a $B$ transformation in a GH homeomorphism, any holomorphic tangent vector to a fiber of $P$ is an element of 
$\mathcal{GT}$. 
Therefore,  any holomorphic tangent vector to the fiber at the point $p \in P$ defines a unique left invariant GH tangent vector field on $G$. Thus, we have an SGH vector bundle isomorphism 
$$\mathcal{GT}\cong P\times\mathfrak{g} \,.$$
Then, the action of $G$ on $\mathcal{GT}$ induces an action on $P\times\mathfrak{g}$ as follows,
\begin{equation}\label{action}
    (p,l)\cdot g=(p\cdot g\,,\ad(g^{-1})\cdot l)\,\,\,\,\,\forall\,\,\,(p,l)\in P\times\mathfrak{g}\,.
\end{equation}
Let $P \times_{G} \mathfrak{g}$ be the identification space defined by the action in \eqref{action}. The adjoint map is a biholomorphism due to the complex Lie group structure of $G$. Consequently,  the transition maps of the complex vector bundle $P\times_{G} \mathfrak{g}$ are GH maps. Therefore, by Theorem \ref{main1} and \cite[Proposition 3.2]{lang2023}, $P\times_{G} \mathfrak{g}$ is an SGH vector bundle over $M$ with fiber $\mathfrak{g}$ associated to $P$ by the adjoint representation. 
Hence, $\mathcal{R}=\mathcal{GT}/G\cong P\times_G \mathfrak{g}\,.$ We shall denote it by $Ad(P)$, that is,
\begin{equation}\label{atiyah GH bundle}
    Ad(P):= P\times_{G} \mathfrak{g}\,.
\end{equation}
The projection $\pi:P\longrightarrow M$ induces a bundle map $\mathcal{G}\pi:At(P)\longrightarrow\mathcal{G}M$. 
Using Definition \ref{main def}, we deduce that $\mathcal{G}\pi$ is an SGH vector bundle homomorphism. 

\medskip
Moreover, let $T_{s}$ denote the tangent bundle and $\mathcal{G}T_{s}$ denotes the GH tangent bundle of $s(U)$ respectively where $s(U)$ is the image of the GH section $s$ as in \eqref{section}. Then we have $(\mathcal{G}P)_{s}=(\mathcal{GT})_{s}\oplus(\mathcal{G}T_{s})$ where $(\mathcal{G}P)_{s}$ and $(\mathcal{GT})_{s}$ are as defined in \eqref{gamma} and \eqref{gamma'}, respectively. This implies the following commutative diagram:
\[\begin{tikzcd}
0 \arrow[r] & \mathcal{R}_{U} \arrow[r, "\xi",] \arrow[d, "\gamma^{'}_{s}"] & At(P)_{U} \arrow[r, "\mathcal{G}\pi"] \arrow[d, "\gamma_{s}"] & \mathcal{G}U \arrow[r] \arrow[d, "s^{\#}"] & 0 \\
0 \arrow[r] & (\mathcal{GT})_{s} \arrow[r]                                       & (\mathcal{G}P)_{s} \arrow[r]                                        & \mathcal{G}T_{s} \arrow[r]                 & 0
\end{tikzcd}\] where $\gamma_{s}$, $\gamma^{'}_{s}$ and $\xi$ are as in \eqref{gamma}, \eqref{gamma'} and the natural inclusion map, respectively. Also, the map $s^{\#}:\mathcal{G}U\longrightarrow\mathcal{G}T_{s}$, induced by $s$, is an isomorphism of  SGH vector bundles.  We conclude that 
\[\begin{tikzcd}
0 \arrow[r] & \mathcal{R} \arrow[r, "\xi"] & At(P) \arrow[r, "\mathcal{G}\pi"] & \mathcal{G}M \arrow[r] & 0
\end{tikzcd}\] is a short exact sequence of SGH vector bundles over $M$. We summarize our results in the following theorem.

\begin{theorem}\label{main2}
   Let $P$ be an SGH principal $G$-bundle over a regular GC manifold $(M,\mathcal{J}_{M})$ where $G$ is a complex Lie group. Then, there exists a canonical short exact sequence
    $\mathcal{A}(P)$ of SGH vector bundles over $M$:
    \begin{equation}\label{A(P)}
        \begin{tikzcd}
0 \arrow[r] & Ad(P) \arrow[r] & At(P) \arrow[r] & \mathcal{G}M \arrow[r] & 0
\end{tikzcd}
    \end{equation}
    where $\mathcal{G}M$ is the GH tangent bundle of $M$ as in \eqref{GH tangent bundle}, $Ad(P)$ is the SGH vector bundle associated to $P$ by the adjoint representation of $G$ as in \eqref{atiyah GH bundle}, and $At(P)$ is the SGH vector bundle of invariant GH tangent vector fields on $P$ as in \eqref{Q}.
\end{theorem}
\begin{definition}\label{def:gh conn}
Let $P$ be an SGH principal $G$-bundle over a regular GC manifold $M$ where $G$ is a complex Lie group. A generalized holomorphic (GH) connection on $P$ is a splitting of the short exact sequence $\mathcal{A}(P)$  in \eqref{A(P)} such that the splitting map is a GH map.    
\end{definition}

By \cite[Proposition 2]{atiyah57}, the extension $\mathcal{A}(P)$ defines an element 
$$a(P)\in H^{1}(M,\Hom_{\mathcal{O}_{M}}(\mathbf{\mathcal{G}M}\,,\,\mathbf{Ad(P)})\,,$$ and $\mathcal{A}(P)$ is a trivial extension if and only if $a(P)=0$.

\medskip
Note that $\mathbf{\mathcal{G}^{*}M}=\Hom_{\mathcal{O}_{M}}(\mathbf{\mathcal{G}M}\,,\,\mathcal{O}_{M})$. Hence,
$$\Hom_{\mathcal{O}_{M}}(\mathbf{\mathcal{G}M}\,,\,\mathbf{Ad(P)})=\mathbf{Ad(P)}\otimes_{\mathcal{O}_{M}}\mathbf{\mathcal{G}^{*}M}\,.$$
Thus we have the following result. 

\begin{theorem}\label{main3}
    An SGH principal $G$-bundle $P$ over a regular GC manifold $M$ defines an element 
    $$a(P)\in H^{1}(M,\mathbf{Ad(P)}\otimes_{\mathcal{O}_{M}}\mathbf{\mathcal{G}^{*}M})\,.$$ $P$ admits a GH connection if and only if $a(P)=0$.
\end{theorem}
\begin{definition}\label{def:atiyah}
    The element $a(P)$ in Theorem \ref{main3} is called the \textit{Atiyah class} of the SGH principal $G$-bundle $P$. The SGH vector bundle $At(P)$ in \eqref{A(P)} is called the \textit{SGH Atiyah bundle} of the SGH principal $G$-bundle $P$.
\end{definition}

\begin{definition}\label{def:gen conn}
  A smooth generalized connection in the principal bundle $P$ is a smooth splitting of the short exact sequence $\mathcal{A}(P)$  in \eqref{A(P)}.
  \end{definition}
  
\begin{remark}
In this case, when $\mathcal{A}(P)$ is considered as a short exact sequence of smooth vector bundles, again by \cite[Proposition 2]{atiyah57}, the smooth extension $\mathcal{A}(P)$ defines an element 
$$a'(P)\in H^{1}(M,\Hom_{C^{\infty}_{M}}(C^{\infty}({\mathcal{G}M})\,,\,C^{\infty}({Ad(P)}))\,,$$ and $\mathcal{A}(P)$ is a trivial smooth extension if and only if $a'(P)=0$. But due to the smooth partition of unity, $\Hom_{C^{\infty}_{M}}(C^{\infty}({\mathcal{G}M})\,,\,C^{\infty}({Ad(P)})$ is a fine sheaf. This implies $H^{1}(M,\Hom_{C^{\infty}_{M}}(C^{\infty}({\mathcal{G}M})\,,\,C^{\infty}({Ad(P)}))=0$ and $a'(P)$ is always zero. Thus a smooth generalized connection always exists.
\end{remark}

\subsection{Local coordinate description of the Atiyah class}\label{connection} 
In this section, we compute the Atiyah class $a(P)$ in local coordinates following Atiyah \cite{atiyah57}. Let $G$ be a (connected) complex Lie group with complex Lie algebra $\mathfrak{g}\,.$ Let $P$ be an SGH principal $G$-bundle over a regular GC manifold $M$ with  local trivializations $\{U_{\alpha}\,,\,\phi_{\alpha}\}$ and transition maps $\phi_{\alpha \beta}$ (see \eqref{phi alpha beta}).

\medskip
Let $M_{\mathfrak{g}}:=M\times\mathfrak{g}$ denote the trivial SGH vector bundle over $M$ where $\mathfrak{g}$ is the complex Lie algebra of $G$. Since $\phi_{\alpha}$ is a GH homeomorphism and it commutes with the action of $G$, it induces an SGH vector bundle isomorphism
\begin{equation}\label{Mg}
\widehat{\phi_{\alpha}}:At(P)|_{U_{\alpha}}\longrightarrow\mathcal{G}M|_{U_{\alpha}}\oplus M_{\mathfrak{g}}|_{U_{\alpha}}\,.
\end{equation}
Define the SGH vector bundle homomorphism 
\begin{equation}\label{a alpha}
   a_{\alpha}: \mathcal{G}M|_{U_{\alpha}}\longrightarrow At(P)|_{U_{\alpha}} 
\end{equation}
by $a_{\alpha}(X)=(\widehat{\phi_{\alpha}})^{-1}(X\oplus 0)$ for all $X\in \mathcal{G}M|_{U_{\alpha}}$. Then the map $a_{\alpha\beta}:\mathcal{G}M|_{U_{\alpha\beta}}\longrightarrow At(P)|_{U_{\alpha\beta}}$, defined as $$a_{\alpha\beta}=a_{\beta}-a_{\alpha}\,,$$ gives a representative $1$-cocycle for $a(P)$ in $H^{1}(M,\Hom_{\mathcal{O}_{M}}(\mathbf{\mathcal{G}M}\,,\,\mathbf{Ad(P)})$.

\medskip
Denote $G_{\mathfrak{g}}:=G\times\mathfrak{g}$. Note that  $\mathcal{G}G:=T^{1,0}G$.  Both right and left multiplication maps on $G$ are biholomorphic. Using them we have  SGH bundle isomorphisms
$$\xi:\mathcal{G}G\longrightarrow G_{\mathfrak{g}}\quad\text{and}\quad\eta:\mathcal{G}G\longrightarrow G_{\mathfrak{g}}\,,$$
 respectively. Thus, 
$$\xi\,,\,\eta\in H^0(G,\Hom_{\mathcal{O}_{G}}(\mathbf{T^{1,0}G},\mathbf{G_{\mathfrak{g}}}))\,.$$

\vspace{0.2em}
Now, $\phi_{\alpha\beta}$ is a GH map due to Proposition \ref{GH principal bundle}, thereby it induces elements $$\xi_{\alpha\beta}\,,\,\eta_{\alpha\beta}\in\Gamma(U_{\alpha\beta},\Hom_{\mathcal{O}_{M}}(\mathbf{\mathcal{G}M},\mathbf{M_{\mathfrak{g}}}))\,.$$  
Then, for each $X\in\mathcal{G}M|_{U_{\alpha\beta}}$,
\begin{align*}
 \widehat{\phi_{\alpha}}(a_{\alpha\beta}(X))&=\widehat{\phi_{\alpha}}((\widehat{\phi_{\beta}})^{-1}(X\oplus 0)-(\widehat{\phi_{\alpha}})^{-1}(X\oplus 0))\,\\
 &=\widehat{\phi_{\alpha}}((\widehat{\phi_{\beta}})^{-1}(X\oplus 0))-(X\oplus 0)\,\\
 &=(X\oplus\xi_{\alpha\beta}(X))-(X\oplus 0)\,\\
 &=(0\oplus\xi_{\alpha\beta}(X))\,.
\end{align*}
By the short exact sequence in \eqref{A(P)}, we can identify $Ad(P)|_{U_{\alpha}}$ as an SGH subbundle of $At(P)|_{U_{\alpha}}$. Then, the SGH vector bundle isomorphism between $Ad(P)|_{U_{\alpha}}$ and $M_{\mathfrak{g}}|_{U_{\alpha}}$ is identified with the restriction map
$$\widehat{\phi_{\alpha}}|_{Ad(P)|_{U_{\alpha}}}:Ad(P)|_{U_{\alpha}}\longrightarrow M_{\mathfrak{g}}|_{U_{\alpha}}\,.$$ 
Therefore, we get 
\begin{equation}\label{a alpha beta 2}
    a_{\alpha\beta}=(\widehat{\phi_{\alpha}})^{-1}\circ\xi_{\alpha\beta}\,,
\end{equation} and since $\xi_{\alpha\beta}=\ad(\phi_{\alpha\beta})\cdot\eta_{\alpha\beta}$,  we
can replace \eqref{a alpha beta 2} by
\begin{equation}\label{a alpha beta 3}
  a_{\alpha\beta}=(\widehat{\phi_{\beta}})^{-1}\circ\eta_{\alpha\beta}\,.  
\end{equation}
Now if $a(P)=0$, then the coboundary equation is 
$$a_{\alpha\beta}=\gamma_{\beta}-\gamma_{\alpha}$$ where $\gamma_{i}\in\Gamma(U_{i},\Hom_{\mathcal{O}_{M}}(\mathbf{\mathcal{G}M}\,,\,\mathbf{Ad(P)}))$ for $i\in\{\alpha,\beta\}$. For each $i\in\{\alpha,\beta\}$, if we denote 
$$\Theta_{i}:=\widehat{\phi_{i}}\circ\gamma_{i}\,,$$ then $\Theta_{i}\in\Gamma(U_{i},\Hom_{\mathcal{O}_{M}}(\mathbf{\mathcal{G}M}\,,\,\mathbf{M_{\mathfrak{g}}}))\,.$ Thus the 
coboundary equation becomes
\begin{equation}\label{co boundary equation}  \xi_{\alpha\beta}=\ad(\phi_{\alpha\beta})\cdot\Theta_{\beta}-\Theta_{\alpha}\,,
\end{equation}
or
\begin{equation}\label{co boundary equation 2}
\eta_{\alpha\beta}=\Theta_{\beta}-\ad(\phi_{\beta\alpha})\cdot\Theta_{\alpha}\,. 
\end{equation}
\begin{remark}
Note that, in case of smooth generalized connection, since we have $$H^{1}(M,\Hom_{C^{\infty}_{M}}(C^{\infty}({\mathcal{G}M})\,,\,C^{\infty}({Ad(P)})))=0\,,$$ the co boundary equation is 
$$a_{\alpha\beta}=\gamma'_{\beta}-\gamma'_{\alpha}$$ where $\gamma'_{i}\in C^{\infty}(U_{i},\Hom_{C^{\infty}_{M}}(C^{\infty}({\mathcal{G}M})\,,\,C^{\infty}({Ad(P)})))$ for $i\in\{\alpha,\beta\}$. Then for each $i$ in $\{\alpha,\beta\}$, if we again denote 
$$\Theta_{i}:=\widehat{\phi_{i}}\circ\gamma'_{i}\,,$$ we get that $\Theta_{i}\in C^{\infty}(U_{i},
\Hom_{C^{\infty}_{M}}(C^{\infty}({\mathcal{G}M})\,,\,C^{\infty}(M_{\mathfrak{g}})))\,.$
 Thus the co-boundary equation becomes
\begin{equation}\label{co boundary equation1-2}  \xi_{\alpha\beta}=\ad(\phi_{\alpha\beta})\cdot\Theta_{\beta}-\Theta_{\alpha}\,,
\end{equation}
or
\begin{equation}\label{co boundary equation2-2}
\eta_{\alpha\beta}=\Theta_{\beta}-\ad(\phi_{\beta\alpha})\cdot\Theta_{\alpha}\,. 
\end{equation}   
\end{remark}
\section{ Atiyah class of an SGH vector bundle}\label{sec atiyah}
Let $E$ be an SGH vector bundle over a regular GC manifold $M$ with local trivializations $\{U_{\alpha}\,,\,\phi_{\alpha}\}$. Let $J^{1}(E)$ be the first jet bundle of $E$ over $M$ as defined in \cite[Section 3.2]{lang2023}. Then by \cite[Theorem 3.17]{lang2023}, $J_{1}(E)$ is an  SGH vector bundle over $M$ and it fits into the following exact sequence, denoted by $\mathcal{B}(E)$, 
\begin{equation}\label{B(E)}
  \begin{tikzcd}
0 \arrow[r] & \mathcal{G}^{*}M\otimes E \arrow[r, "J"] & J_{1}(E) \arrow[r, "\pi_{1}"] & E \arrow[r] & 0
\end{tikzcd}  
\end{equation}
of SGH vector bundles over $M$.
\medskip

As a sheaf of $\C$-modules,
$$\mathbf{J_{1}(E)}=\mathbf{E}\oplus_{\C}(\mathbf{\mathcal{G}^{*}M}\otimes_{\mathcal{O}_{M}}\mathbf{E})\,.$$
Recall that, for each $m\in M$, $f\in\mathcal{O}_{M,m}$ if and only if $(df)_{m}\in(\mathcal{G}^{*}M)_{m}$. So, we can define the map
$$\phi_{m}:\mathcal{O}_{M,m}\times\mathbf{J_{1}(E)}_{m}\longrightarrow \mathbf{J_{1}(E)}_{m}$$ by 
$$\phi_{m}(f, s+ \delta ) = fs\oplus(f\delta+df\otimes s)$$
where $s\in E_{m}$,  $ \delta \in ((\mathcal{G}^{*}M)_{m}\otimes_{\mathcal{O}_{M,m}} E_{m})$, and $f\in\mathcal{O}_{M,m}$. 
This defines an action of $\mathcal{O}_{M}$ on $\mathbf{J_{1}(E)}$ making it a sheaf of $\mathcal{O}_{M}$-modules.  We obtain the following short exact sequence of $\mathcal{O}_{M}$-modules
\begin{equation}\label{B(E)2}
  \begin{tikzcd}
0 \arrow[r] & \mathbf{\mathcal{G}^{*}M}\otimes_{\mathcal{O}_{M}}\mathbf{E}\arrow[r, "\widehat{J}"] & \mathbf{J_{1}(E)} \arrow[r, "\widehat{\pi_1}"] & \mathbf{E} \arrow[r] & 0
\end{tikzcd}  
\end{equation}
where $\widehat{J}(\delta)=0+\delta$ and $\widehat{\pi_{1}}(s+\delta)=s$ are the morphisms of $\mathcal{O}_{M}$-modules induced by the maps $J$ and $\pi_1$  in \eqref{B(E)}, respectively.

\medskip
Since $\Hom_{\mathcal{O}_{M}}(\mathbf{E},\mathbf{\mathcal{G}^{*}M}\otimes_{\mathcal{O}_{M}}\mathbf{E})\cong\mathbf{\mathcal{G}^{*}M}\otimes_{\mathcal{O}_{M}}\mathbf{\en(E)}$,  by \cite[Proposition 2]{atiyah57} and using \eqref{B(E)2}, the extension $\mathcal{B}(E)$ defines an element
$$b(E)\in H^{1}(M,\mathbf{\mathcal{G}^{*}M}\otimes_{\mathcal{O}_{M}}\mathbf{\en(E)})\,.$$

\begin{definition} (\cite[Definition 4.4]{lang2023})\label{atiyah VB}
$b(E)$ is called the \textit{Atiyah class} of the SGH vector bundle $E$ over $M$. 
\end{definition}

The following result is standard in the holomorphic case (see \cite[Proposition 9]{atiyah57}) and follows similarly in the SGH setting.  

\begin{prop}\label{imp prop}
 Let $E$ be an SGH vector bundle of real rank $2l$ over $M$. Let $P_{E}$ be the corresponding SGH principal $GL_{l}(\C)$-bundle as in \eqref{p-e}. Then we have $$\en(E)\cong Ad(P_{E})$$ as SGH vector bundles where $Ad(P_{E})$ as in \eqref{atiyah GH bundle}.  
\end{prop}

\begin{corollary}
  $H^{1}(M,\mathbf{\mathcal{G}^{*}M}\otimes_{\mathcal{O}_{M}}\mathbf{\en(E)})\cong H^{1}(M,\mathbf{\mathcal{G}^{*}M}\otimes_{\mathcal{O}_{M}}\mathbf{Ad(P_{E})})$.  
\end{corollary}

\begin{theorem}\label{main4}
    Let $E$ be an SGH vector bundle over a regular GC manifold $M$. Let $P$ be the associated SGH principal $GL_{l}(\C)$-bundle over $M$, as in \eqref{e-p}, where $l$ is the complex rank of $E$. Let $b(E)$ and $a(P)$ be the obstruction elements defined by $\mathcal{B}(E)$ and $\mathcal{A}(P)$, as in the equations \eqref{B(E)} and \eqref{A(P)}, respectively. Then 
    $$a(P)=-b(E) \,.$$
\end{theorem}

\begin{proof}
    Let $E$ be an SGH vector bundle with local trivializations $\{U_{\alpha},\phi_{\alpha}\}$ where  
    \begin{equation}\label{phi alpha beta 3}
    \phi_{\alpha}: E|_{U_{\alpha}}\longrightarrow U_{\alpha}\times \C^l\,,
    \end{equation}
    are local GH homeomorphisms (cf. \eqref{phi alpha}). 
     Then $P$ is defined by the transition functions (cf. \eqref{phi alpha beta}), 
    \begin{equation}
        \phi_{\alpha\beta}: U_{\alpha\beta}\longrightarrow GL_{l}(\C)\,,
    \end{equation} 
    where
    \begin{equation}\label{psi alpha beta3}
\psi_{\alpha\beta}=\phi_{\alpha}\circ\phi^{-1}_{\beta}: U_{\alpha\beta}\times GL_{l}(C)\longrightarrow U_{\alpha\beta}\times GL_{l}(\C)
\end{equation}
is given by $\psi_{\alpha\beta}(m\,,g)=(m\,,\phi_{\alpha\beta}(m)g)$.

\medskip
Let $W =  \C^l$ so that $E \cong P\times_{GL_{l}(\C)} W$. Let $M_{W}=M\times W$, a trivial SGH vector bundle over $M$. The GH homeomorphism $\phi_{\alpha}$ induces a sheaf isomorphism of $\mathcal{O}_{M}|_{U_{\alpha}}$-modules
\begin{equation}\label{tilde phi alpha}
\widetilde{\phi}_{\alpha}:\mathbf{E}|_{U_{\alpha}}\longrightarrow\mathbf{M_{W}}|_{U_{\alpha}}\,.    
\end{equation} 
This also induces another canonical $\mathcal{O}_{M}|_{U_{\alpha}}$-module isomorphism, again denoted by $\widetilde{\phi}_{\alpha}$,
\begin{equation}\label{tilde phi alpha 2}
\widetilde{\phi}_{\alpha}=\widetilde{\phi}_{\alpha}\otimes_{\mathcal{O}_{M}}\id:\mathbf{E}|_{U_{\alpha}}\otimes_{\mathcal{O}_{M}|_{U_{\alpha}}}\mathbf{\mathcal{G}^{*}M}|_{U_{\alpha}}\longrightarrow\mathbf{M_{W}}|_{U_{\alpha}}\otimes_{\mathcal{O}_{M}|_{U_{\alpha}}}\mathbf{\mathcal{G}^{*}M}|_{U_{\alpha}}\,.
\end{equation}
Let $d$ be the exterior derivative. Note that $f\in\mathcal{O}_{M}$ if and only if $df\in\mathbf{\mathcal{G}^{*}M}$. Thus we can extend $d$ to a $\C$-linear sheaf homomorphism,
\begin{equation}\label{d}
d:\mathbf{M_{W}}|_{U_{\alpha}}\longrightarrow\mathbf{M_{W}}|_{U_{\alpha}}\otimes_{\mathcal{O}_{M}|_{U_{\alpha}}}\mathbf{\mathcal{G}^{*}M}|_{U_{\alpha}}\,.
\end{equation}
Define a $\C$-homomorphism of sheaves over $U_{\alpha}$,
\begin{equation}\label{D alpha}
D_{\alpha}:\mathbf{E}|_{U_{\alpha}}\longrightarrow\mathbf{E}|_{U_{\alpha}}\otimes_{\mathcal{O}_{M}|_{U_{\alpha}}}\mathbf{\mathcal{G}^{*}M}|_{U_{\alpha}}\quad\text{by}\quad D_{\alpha}(s)=(\widetilde{\phi}_{\alpha})^{-1}d\widetilde{\phi}_{\alpha}(s)\,,  
\end{equation} 
where the first $\widetilde{\phi}_{\alpha}$, $d$ and the second $\widetilde{\phi}_{\alpha}$ are as in the equations \eqref{tilde phi alpha 2}, \eqref{d} and \eqref{tilde phi alpha}, respectively. 
Now consider the sheaf homomorphism 
\begin{equation} b_{\alpha}:\mathbf{E}|_{U_{\alpha}}\longrightarrow\mathbf{J_{1}(E)}|_{U_{\alpha}}\,\,\,\text{defined by}\,\,\,b_{\alpha}(s)=s+D_{\alpha}(s) \,\, {\rm for \, all\,\,} s\in\mathbf{E}|_{U_{\alpha}}\,.
\end{equation} 
Then for any $f\in\mathcal{O}_{M}|_{U_{\alpha}}$ and $s\in\mathbf{E}|_{U_{\alpha}}$, we have 
\begin{align*}
 b_{\alpha}(fs)&=fs+(\widetilde{\phi}_{\alpha})^{-1}d\widetilde{\phi}_{\alpha}(fs)\,\\
 &=fs\oplus(s\otimes df+f(\widetilde{\phi}_{\alpha})^{-1}d\widetilde{\phi}_{\alpha}(s))\,\\
 &=f\cdot(s+D_{\alpha}(s))\,.
\end{align*}
Hence, $b_{\alpha}$ is an $\mathcal{O}_{M}$-module homomorphism. Consider the sheaf homomorphism 
\begin{equation}
b_{\alpha\beta}:\mathbf{E}|_{U_{\alpha\beta}}\longrightarrow\mathbf{J_{1}(E)}|_{U_{\alpha\beta}}\,\,\,\text{defined by}\,\,\,b_{\alpha\beta}:=b_{\beta}-b_{\alpha}\,.
\end{equation} Note that $b_{\alpha\beta}(s)=D_{\beta}(s)-D_{\alpha}(s)$. So
$$b_{\alpha\beta}\in\Gamma(U_{\alpha\beta},\Hom_{\mathcal{O}_{M}}(\mathbf{E},\mathbf{E}\otimes_{\mathcal{O}_{M}}\mathbf{\mathcal{G}^{*}M}))\,.$$
This shows that $\{b_{\alpha\beta}\}$ is a representative $1$-cocycle for $b(E)$ in $H^1(M,\mathbf{\en(E)}\otimes_{\mathcal{O}_{M}}\mathbf{\mathcal{G}^{*}M})$.

\medskip
Consider the following two sheaf homomorphisms over $U_{\alpha}$,
\begin{equation}
\widetilde{\phi}_{\alpha\beta}=\widetilde{\phi}_{\alpha}\circ(\widetilde{\phi}_{\beta})^{-1}:\mathbf{M_{W}}|_{U_{\alpha\beta}}\longrightarrow\mathbf{M_{W}}|_{U_{\alpha\beta}}\,,   
\end{equation}
and the second one, also denoted by $\widetilde{\phi}_{\alpha\beta}$,
\begin{equation}
   \widetilde{\phi}_{\alpha\beta}:  \mathbf{M_{W}}|_{U_{\alpha\beta}}\otimes_{\mathcal{O}_{M}|_{U_{\alpha\beta}}}\mathbf{\mathcal{G}^{*}M}|_{U_{\alpha\beta}}\longrightarrow\mathbf{M_{W}}|_{U_{\alpha\beta}}\otimes_{\mathcal{O}_{M}|_{U_{\alpha\beta}}}\mathbf{\mathcal{G}^{*}M}|_{U_{\alpha\beta}}\,.
\end{equation} 
We can see that $\widetilde{\phi}_{\alpha\beta}$ can be thought of as a $\mathcal{O}_{M}|_{U_{\alpha\beta}}$-valued matrix, again denoted by $$\widetilde{\phi}_{\alpha\beta}:\bigoplus_{r}\mathcal{O}_{M}|_{U_{\alpha\beta}}\longrightarrow\bigoplus_{r}\mathcal{O}_{M}|_{U_{\alpha\beta}}.$$ So $d(\widetilde{\phi}_{\alpha\beta})$ is well understood. Then for any $s\in\mathbf{M_{W}}|_{U_{\alpha\beta}}$, we get
\begin{equation}\label{imp b alpha beta}
\begin{aligned}
\widetilde{\phi}_{\alpha}b_{\alpha\beta}(\widetilde{\phi}_{\alpha})^{-1}(s)&=\widetilde{\phi}_{\alpha}(D_{\beta}((\widetilde{\phi}_{\alpha})^{-1}(s))-D_{\alpha}((\widetilde{\phi}_{\alpha})^{-1}(s)))\,\\
&=\widetilde{\phi}_{\alpha}((\widetilde{\phi}_{\beta})^{-1}d(\widetilde{\phi}_{\beta}(\widetilde{\phi}_{\alpha})^{-1}(s))-(\widetilde{\phi}_{\alpha})^{-1}(ds)) \\
&=\widetilde{\phi}_{\alpha\beta}(d(\widetilde{\phi}_{\alpha\beta}^{-1}(s)))-ds \\
&=\widetilde{\phi}_{\alpha\beta}\,d(\widetilde{\phi}_{\alpha\beta}^{-1})\cdot s \\
&=-d(\widetilde{\phi}_{\alpha\beta})\,\widetilde{\phi}_{\alpha\beta}^{-1}\cdot s\,.
\end{aligned}
\end{equation}
But, in the notation of Subsection \ref{connection}, $d(\widetilde{\phi}_{\alpha\beta})\,\widetilde{\phi}_{\alpha\beta}^{-1}=\xi_{\alpha\beta}$.
Here, using  $\mathfrak{g}=\mathfrak{gl}_{l}(\C)$, we  identify the three sheaves $\Hom_{\mathcal{O}_{M}}(\mathbf{M_{W}},\mathbf{M_{W}}\otimes_{\mathcal{O}_{M}}\mathbf{\mathcal{G}^{*}M})$, $\Hom_{\mathcal{O}_{M}}(\mathbf{\mathcal{G}M},\mathbf{M_{\mathfrak{g}}})$, and $\mathbf{M_{\mathfrak{g}}}\otimes_{\mathcal{O}_{M}}\mathbf{\mathcal{G}^{*}M}$ via their respective canonical $\mathcal{O}_{M}$-module isomorphisms where $\mathbf{\mathcal{G}M}$, $\mathbf{M_{\mathfrak{g}}}$ are as in the equations \eqref{GH tangent bundle} and \eqref{Mg} respectively.

\vspace{0.3em}
Now, $\Hom_{\mathcal{O}_{M}}(\mathbf{E},\mathbf{E}\otimes_{\mathcal{O}_{M}}\mathbf{\mathcal{G}^{*}M})$ is isomorphic to $\mathbf{Ad(P)}\otimes_{\mathcal{O}_{M}}\mathbf{\mathcal{G}^{*}M}$ by Proposition \ref{imp prop}. Therefore, upon identifying $\Hom_{\mathcal{O}_{M}}(\mathbf{\mathcal{G}M},\mathbf{Ad(P)})$, $\mathbf{Ad(P)}\otimes_{\mathcal{O}_{M}}\mathbf{\mathcal{G}^{*}M}\,,$ and also $\Hom_{\mathcal{O}_{M}}(\mathbf{E},\mathbf{E}\otimes_{\mathcal{O}_{M}}\mathbf{\mathcal{G}^{*}M})$ via  canonical $\mathcal{O}_{M}$-module isomorphisms, we have, via  equations \eqref{a alpha beta 2} and \eqref{imp b alpha beta}, that 
$$b_{\alpha\beta}=-a_{\alpha\beta}\,.$$ It follows that $a(P)=-b(E)$.
\end{proof}

\section{Generalized complex structure and orbifold}\label{sec orbi}
Let $M$ be a regular GC manifold of dimension $2m$ and type $k$. Let $\mathscr{S}$ denote the associated symplectic foliation of complex codimension $k$ which is transversely holomorphic. Let $T\mathscr{S}$ be the corresponding involutive subbundle of $TM$ of rank $2m-2k$, called the tangent bundle of the foliation. The normal bundle of the foliation, denoted by $\mathcal{N}$, is defined by 
$$\mathcal{N}:=TM/T\mathscr{S}\,.$$ By \cite[Proposition 4.2]{Gua2}, $\mathcal{N}$ is an integrable subbundle with a complex structure. Then $\mathcal{N}$ has a decomposition given by the complex structure
\begin{equation}\label{nrml bundle}
\mathcal{N}\otimes\C=\mathcal{N}^{1,0}\oplus\mathcal{N}^{0,1}\,.
\end{equation}
As the exact sequence 
\[\begin{tikzcd}[ampersand replacement=\&]
	0 \& {T\mathscr{S}} \& TM \& {\mathcal{N}} \& 0
	\arrow[from=1-2, to=1-3]
	\arrow[from=1-3, to=1-4]
	\arrow[from=1-1, to=1-2]
	\arrow[from=1-4, to=1-5]
\end{tikzcd}\,\]
splits smoothly,
$\mathcal{N}$ may be regarded as a subbundle of $TM$ complementary to $T\mathscr{S}$, and we may identify $\mathcal{N}^{1,0}$ with $\mathcal{G}M$.
Define
\begin{equation}\label{leaf sp}
    \mathscr{M}:=M/\mathscr{S}
\end{equation} to be the leaf space of the foliation $\mathscr{S}$. This is a topological space that has the  quotient topology induced by the quotient map
\begin{equation}\label{leaf sp map}
\tilde{\pi}:M\longrightarrow\mathscr{M} \,.    
\end{equation}
 The map $\tilde{\pi} $  is open (cf. \cite[Section 2.4]{moerdijk03}). 
\medskip

{In general, $\mathscr{M}$ could be rather wild. To have a reasonable theory, we assume that $\mathscr{M}$ admits a smooth orbifold structure.  Since $\mathscr{S}$ is transversely holomorphic, $\mathscr{M}$ then becomes a complex orbifold. Moreover, observe that $\tilde{\pi}$ is a smooth complete orbifold map (cf. \cite[Definition 3.1]{borz12}), and each point $y\in\mathscr{M}$ is a regular value of $\tilde{\pi}$. Thus, by the preimage theorem for orbifolds (cf.  \cite[Theorem 4.2]{borz12}), $\tilde{\pi}^{-1}(y)$ is an embedded submanifold of real dimension $2m-2k$ for all $y\in\mathscr{M}$. Hence, each leaf is not only an immersed but also a closed embedded submanifold of $M$.}

\begin{definition}\label{basic open set}
    An open set in $M$ is called a transverse open set if it is a union of leaves. An open cover $\mathcal{U}=\{U_{\alpha}\}$ is called a transverse open cover of $M$ if each $U_{\alpha}$ is a transverse open subset of $M$.
\end{definition}

Let $S$ be a leaf of $\mathscr{S}$. By the Tubular Neighborhood Theorem, there exists a transverse neighborhood (tubular neighborhood) of $S$ which is diffeomorphic to the normal bundle $\mathcal{N}_{S}$ of $S$. One can see that $\mathcal{N}_{S}$ is just the pullback of $\mathcal{N}$ via the inclusion map $S\hookrightarrow M$. Due to the transverse complex structure, $\mathcal{N}$, as well as $\mathcal{N}_{S}$, can be thought of as a complex vector bundle of complex rank $k$.
 Consider the partial connection, known as the Bott connection (cf. \cite[Section 6]{bott71}), on $\mathcal{N}$ which is flat along the leaves. Then its pullback on $\mathcal{N}_{S}$ gives a flat connection. Thus, considering $\mathcal{N}_{S}$ as a complex vector bundle, by \cite[Proposition 1.2.5]{kobayashi14}
$$\mathcal{N}_{S}\cong\tilde{S}\times_{\rho}\C^{k}$$ where $\rho:\pi_1(S)\longrightarrow GL_{k}(\C)$ is the linear holonomy representation of $\pi_1(S)$ and $\tilde{S}$ is the universal cover of $S$.
\begin{definition}
    A $2k$-dimensional embedded submanifold of $M$ is called a transversal section if it is transversal to the leaves of $\mathscr{S}$. 
\end{definition}
Note that by \cite[Proposition 2.20]{moerdijk03}, $\mathscr{M}$ admits a Riemannian metric which makes $\mathscr{S}$ into a Riemannian foliation. 
Since $S$ is an embedded submanifold, $T\cap S$ is discrete for any transversal section $T$.  Then, following the proof of \cite[Theorem 2.6]{moerdijk03}, one can show that the holonomy group of $S$, $\hol(S)$ is finite. By the differentiable slice theorem, we can indeed assume that the action of $\hol(S)$ on $T$ is linear, that is, $$\hol(S)=\img(\rho)\,.$$
We summarise our observations as follows.
\begin{theorem}\label{orbi thm}
    Let $M$ be a regular GC manifold and let $\mathscr{S}$ be the induced symplectic foliation. Assume that $M/\mathscr{S}$ has a smooth orbifold structure. Then, we have the following.
    \begin{enumerate}
    \setlength\itemsep{1em}
        \item Each leaf of $\mathscr{S}$ is an embedded closed submanifold of $M$.
        \item The holonomy group of each leaf is finite.
        \item $(M,\mathscr{S})$ is a regular Riemannian foliation.
        \item Around each leaf $S$, there exists a tubular neighborhood $U$ such that $U$   is diffeomorphic to $\tilde{S}\times_{\hol(S)}\C^{k}$ where $\tilde{S}$ is the universal cover of $S$ and $\hol(S)$ is the holonomy group of $S$. Here, $\hol(S)$ acts on $\C^{k}$ via a linear holonomy representation.      
    \end{enumerate}
\end{theorem}
\begin{example}\label{orbi eg}
    Let $F$ be a symplectic manifold and $\tilde{F}$ be its universal cover. Then as in Example \ref{imp exmple}, $\tilde{F}\times_{\rho}\C^{l}$ is a regular GC manifold of type $l$. The induced symplectic foliation $\mathscr{S}$ is the foliation of $F$-parameter submanifolds, that is, sets of the form $$S_{x}=\left\{[\tilde{m},y]\,|\,\tilde{m}\in\tilde{F},y\in [x]\right\}\quad\text{where}\quad [x]:=\{\rho(g)\cdot x\,|\,g\in\pi_1(F)\}\subset\C^{l}\,.$$ This implies that the leaf space $\tilde{F}\times_{\rho}\C^{l}/\mathscr{S}$ is exactly $$\C^{l}/\rho:=\{[x]\,|\,x\in\C^{l}\}\,.$$ The isotropy group at $0$ is $\img(\rho)$ which is the linear holonomy group. Therefore, we get that  $\tilde{F}\times_{\rho}\C^{l}/\mathscr{S}$ is a smooth orbifold if and only if the linear holonomy group is finite.
\end{example}
\begin{remark}
    It is tempting to think that the leaf space of a regular GCS is either manifold or an orbifold. But, it may not be even Hausdorff. The following example demonstrates this.
\end{remark}
\begin{example}\label{counter eg}
    Consider the product GCS on $M\times F$ where $M$ is a complex manifold and $F$ is a symplectic manifold. Let $N\subset F$ be a closed submanifold such that $F\backslash N$ is disconnected.  Fix $m\in M$, Consider the open submanifold $$X_{m}=M\times F\backslash\{m\times N\}\,.$$
    Consider the natural regular GCS on $X_{m}$ induced from $M\times F$. Let $(x,f)\in X_{m}$. Then, the leaf of the induced foliation $\mathscr{S}_{m}$, through $(x,f)$, is of the following form
    \begin{align*}
        S_{(x,f)}&=
        \begin{cases}
            F & \text{if }  x\neq m\,,\\
            (F\backslash N)_{\alpha}  & \text{if }  x=m\,,
        \end{cases}
    \end{align*}
where $(F\backslash N)_{\alpha}$ denotes the connected component of $F\backslash N$ that contains $f$ for $x=m$. One can see that the leaf space $X_{m}/\mathscr{S}_{m}$ is not Hausdroff. Thus, we obtain an infinite family of regular GC manifolds with non-Hausdorff leaf space.
\end{example}

\section{Dolbeault cohomology of SGH vector bundles}\label{sec cohomo}

\subsection{Cohomology Theory}\label{sec cohomolgy}
Let $(M,\mathcal{J}_{M})$ be a regular GC manifold with $i$-eigen bundle $L$. Then, $(TM\oplus T^{*}M)\otimes\C=L\oplus\overline{L}$.  We have a differential operator, 
\begin{equation}\label{d-L}
    d_{L}:C^{\infty}(\wedge^{\bullet}L^{*})\longrightarrow C^{\infty}(\wedge^{\bullet+1}L^{*})\,,
\end{equation}
defined as follows.
For any $\omega\in C^{\infty}(\wedge^{n} {L^*})$ and $X_{i}\in C^{\infty}(L)$ for all $i\in\{1,\cdots,n+1\}$, 
\begin{align*}
d_{L}\omega(X_{1},\cdots,X_{n+1})& :=\sum_{i=1}^{n+1}(-1)^{i+1}\rho(X_{i})(\omega(X_{1},\cdots,\hat{X_{i}},\cdots,X_{n+1}))\\
&+\sum_{i<j}(-1)^{i+j}\omega([X_{i},X_{j}],X_{1},\cdots,\hat{X_{i}},\cdots,\cdots,\hat{X_{j}},\cdots,X_{n+1})\,,
\end{align*}
where $\rho:(TM\oplus T^{*}M)\otimes\C\longrightarrow TM\otimes\C$ is the projection map and $[\,,\,]$ is the Courant bracket.
Similarly, we have another operator
\begin{equation}\label{d-L bar}
    d_{\overline{L}}:C^{\infty}(\wedge^{\bullet}\overline{L}^{*})\longrightarrow C^{\infty}(\wedge^{\bullet+1}\overline{L}^{*})\,.
\end{equation} 
Note that
\begin{equation}
\overline{\mathcal{G}^{*}M}=\overline{L}\cap(T^{*}M\otimes\C)\quad\text{and}\quad\overline{\mathcal{G}M}=(\overline{\mathcal{G}^{*}M})^{*}\,   
\end{equation}
are also smooth vector bundles over $M$ (cf. \eqref{GH cotangent bundle}, \eqref{GH tangent bundle}). Let $k$ be the type of $\mathcal{J}_{M}$.
So, on a coordinate neighborhood $U$ (cf. \eqref{loc coordi}, Corollary \ref{cor:diffcharts}), 
$$C^{\infty}(\overline{\mathcal{G}^{*}M}|_{U})=\spn_{C^{\infty}(U)}\{d\overline{z_1},\ldots,d\overline{z_{k}}\}\quad\text{and}\quad C^{\infty}(\overline{\mathcal{G}M}|_{U})=\spn_{C^{\infty}(U)}\{\frac{\partial}{\partial\overline{z_1}},\ldots,\frac{\partial}{\partial\overline{ z_{k}}}\}\,.$$
Let $\mathscr{S}$ denote the induced regular transversely holomorphic, symplectic foliation of complex codimension $k$ corresponding to $\mathcal{J}_{M}$. Let $d_{\mathscr{S}}$ denote the exterior derivative along the leaves. Define 
$$F_{M}:=\ker(d_{\mathscr{S}}:C^{\infty}_{M}\longrightarrow C^{\infty}(T^{*}\mathscr{S}\otimes\C))$$ as the sheaf of smooth $\C$-valued functions over $M$ which are constant along the leaves. Note that $\mathcal{O}_{M}\leq F_{M}\leq C^{\infty}_{M}$. For any vector bundle $E$ over $M$ whose transition maps are leaf-wise constant, we denote the sheaf of smooth leaf-wise constant sections of $E$ by $F_M(E)$. 

\medskip
The transition functions of $\mathcal{G}^{*}M$ and $\overline{\mathcal{G}^{*}M}$ are constant along the leaves of $\mathscr{S}$. On a coordinate neighborhood $U$ (cf. \eqref{loc coordi}), 
$$F_{M}(\overline{\mathcal{G}^{*}M}|_{U})=\spn_{F_{M}(U)}\{d\overline{z_1},\ldots,d\overline{z_{k}}\}\,,$$ and 
$$F_{M}(\overline{\mathcal{G}M}|_{U})=\spn_{F_{M}(U)}\{\frac{\partial}{\partial\overline{z_1}},\ldots,\frac{\partial}{\partial\overline{ z_{k}}}\}\,.$$
For any $p,q\geq 0$\,, define
\begin{equation}\label{A-pq}
\begin{aligned}
\tilde{A}^{p,q}:=C^{\infty}(\wedge^{p}\mathcal{G}^{*}M\otimes\wedge^{q}\overline{\mathcal{G}^{*}M})\,,\\
A^{p,q}:=F_{M}(\wedge^{p}\mathcal{G}^{*}M\otimes\wedge^{q}\overline{\mathcal{G}^{*}M})\,.   
\end{aligned}
\end{equation} More specifically, given any open set $U\subseteq M$,
\begin{equation}\label{Apq}
\begin{aligned}
\tilde{A}^{p,q}(U)=C^{\infty}(U,\wedge^{p}\mathcal{G}^{*}M)\otimes_{C^{\infty}(U)}C^{\infty}(U,\wedge^{q}\overline{\mathcal{G}^{*}M})\, ,\\
A^{p,q}(U)=F_{M}(\wedge^{p}\mathcal{G}^{*}M)(U)\otimes_{F_{M}(U)}F_{M}(\wedge^{q}\overline{\mathcal{G}^{*}M})(U)\,. 
\end{aligned}
\end{equation}
Note that $A^{p,q}\leq\tilde{A}^{p,q}$ and $\tilde{A}^{p,q}=A^{p,q}\otimes_{F_{M}} C_{M}^{\infty}$.
For any $l\in\{0,\ldots,2k\}$, denote $A^{l}=\bigoplus_{p+q=l}A^{p,q}$ and $\tilde{A}^{l}=\bigoplus_{p+q=l}\tilde{A}^{p,q}$.
Thus, we get two bigraded sheaves, namely,  
\begin{equation}\label{A}
 A:=\bigoplus_{p,q} A^{p,q}\,\,\,\,,\,\,\,\,\,\tilde{A}:=\bigoplus_{p,q} \tilde{A}^{p,q}\,.
\end{equation}
To summarize, $\tilde{A}$ and $A$ are the bigraded sheaves of germs of sections of $\bigoplus_{p,q}(\wedge^{p}\mathcal{G}^{*}M\otimes\wedge^{q}\overline{\mathcal{G}^{*}M})$, which are smooth and constant along the leaves, respectively.
\vspace{0.5em}

Let $d:C^{\infty}(\wedge^{\bullet}T^{*}M\otimes\C)\longrightarrow C^{\infty}(\wedge^{\bullet+1}T^{*}M\otimes\C)$ be the exterior derivative.  
By \cite[Proposition 4.2]{Gua2}, $\mathcal{G}M$ and $\overline{\mathcal{G}M}$ both are integrable smooth sub-bundle of $TM\otimes\C$\,. Thus we can restrict $d$ to $\tilde{A}^{\bullet}$, $A^{\bullet}$. We denote these restrictions by $\tilde{D}$ and $D$, respectively, that is, 
\begin{equation}\label{D-def}
\tilde{D}:=d|_{\tilde{A}^{\bullet}}\,\,\,\,\,, \,\,\,\,\,D:=d|_{A^{\bullet}}\,.  
\end{equation}
In particular,  any $\omega\in A^{p,q}\,(\text{respectively},\,\,\tilde{A}^{p,q})$, is locally (cf. \eqref{loc coordi}) of the form
$$\omega=\sum_{I,J}f_{IJ}\,dz_{I}\wedge d\overline{z_{J}}\,,$$
where $f_{IJ}\in F_{M}(U)\,(\text{respectively},\,\,C^{\infty}(U))$\,, 
$I, J$ are ordered subsets of $\{1, \ldots, k \}$, and
$dz_{I}= \bigwedge_{i \in I} dz_{i}$, $d\overline{z_J} = \bigwedge_{j \in J} d\overline{z_{j}} $ . 
Then,  
\begin{equation}\label{form of D}
D\omega\,(\text{respectively},\,\,\tilde{D}\omega)=\sum_{I,J}\partial f_{IJ}\,\,dz_{I}\wedge d\overline{z_{J}}+\sum_{I,J}\overline{\partial}f_{IJ}\,\,dz_{I}\wedge d\overline{z_{J}}\,,    
\end{equation}
where $\partial f_{IJ}$ and $\overline{\partial} f_{IJ}$ are defined by
\begin{equation}\label{del-delbar}
    \begin{aligned}
      \partial f_{IJ} :=\sum^{k}_{i=1}\frac{\partial f_{IJ}}{\partial z_{i}}\,dz_{i}\,\,,\,\,\,\,\,\,\,\, 
      \overline{\partial} f_{IJ} :=\sum^{k}_{i=1}\frac{\overline{\partial} f_{IJ}}{\partial \overline{z_{i}}}\,d\overline{z_{i}}\,.
    \end{aligned}
\end{equation}

We identify $L^{*}$ with $\overline{L}$ via the symmetric bilinear form defined in \eqref{bilinear}, and consider the restrictions of $d_{L}$  to $C^{\infty}(\wedge^{\bullet}\overline{\mathcal{G}^{*}M})$ and $d_{\overline{L}}$ to $C^{\infty}(\wedge^{\bullet}\mathcal{G}^{*}M)$. We denote these by $\tilde{d}_{L}$ and $\tilde{d}_{\overline{L}}$, respectively. In particular, for any $\omega\in C^{\infty}(\wedge^{p}\mathcal{G}^{*}M)$, locally we can write
 $$\omega=\sum_{I}f_{I}\,dz_{I}\,.$$ Then, 
 \begin{equation*}
\tilde{d}_{\overline{L}}\omega=\sum_{I}d_{\overline{L}}f_{I}|_{C^{\infty}(\mathcal{G}^{*}M)}\,\,dz_{I}\,.   
 \end{equation*}
We know that, for any $f\in C^{\infty}(U)$, $d_{\overline{L}}f\in L$ and $d_{L}f\in\overline{L}$. Therefore, if we restrict them to $C^{\infty}(\mathcal{G}^{*}M)$ and $C^{\infty}(\overline{\mathcal{G}^{*}M})$, respectively, we get that
\begin{equation*}
 d_{\overline{L}}|_{C^{\infty}(\mathcal{G}^{*}M)}f=\partial f\,\,,\,\,\,\,\,\,d_{L}|_{C^{\infty}(\overline{\mathcal{G}^{*}M})}f=\overline{\partial} f\,,   
\end{equation*} where $\partial f$ and $\overline{\partial} f$ are defined as in \eqref{del-delbar}\,.
 We can further restrict $d_{L}$ and $d_{\overline{L}}$ to $F_{M}(\wedge^{\bullet}\overline{\mathcal{G}^{*}M})$ and $F_{M}(\wedge^{\bullet}\mathcal{G}^{*}M)$ which we again denote by $d_{L}$ and $d_{\overline{L}}$, respectively. 
 Thus we can consider the following morphisms of sheaves
\begin{equation}\label{d-l,d-l bar}
    \begin{aligned}
     d_{L}:F_{M}(\wedge^{\bullet}\overline{\mathcal{G}^{*}M})\longrightarrow F_{M}(\wedge^{\bullet+1}\overline{\mathcal{G}^{*}M})\, ,\\
     d_{\overline{L}}:F_{M}(\wedge^{\bullet}\mathcal{G}^{*}M)\longrightarrow F_{M}(\wedge^{\bullet+1}\mathcal{G}^{*}M)\,.
    \end{aligned}
\end{equation}
Note that $d_{L}=\tilde{d}_{L}|_{F_{M}(\wedge^{\bullet}\overline{\mathcal{G}^{*}M})}$ and $d_{\overline{L}}=\tilde{d}_{\overline{L}}|_{F_{M}(\wedge^{\bullet}\mathcal{G}^{*}M)}$.
  They induce two differential complexes, namely $(F_{M}(\wedge^{\bullet}\overline{\mathcal{G}^{*}M}),d_{L})$ and $(F_{M}(\wedge^{\bullet}\mathcal{G}^{*}M),d_{\overline{L}})\,.$ Subsequently, we can naturally extend $d_{L}$ and $d_{\overline{L}}$ to $A^{\bullet,\bullet}$, again denoted by $d_{L}$ and $d_{\overline{L}}$ respectively, and get the following morphisms of sheaves
  \begin{equation}\label{dL-dL bar}
    \begin{aligned}
d_{L}:A^{\bullet,\bullet}\longrightarrow A^{\bullet,\bullet+1}\,;\\
d_{\overline{L}}:A^{\bullet,\bullet}\longrightarrow A^{\bullet+1,\bullet}\,.
    \end{aligned}
\end{equation}
 In particular, for any $\omega\in A^{p,q}$, locally
 $$\omega=\sum_{I,J}f_{IJ}\,dz_{I}\wedge d\overline{z_{J}}\,.$$ Then, 
 \begin{equation}\label{form of d-l bar}
 d_{{L}}\omega=\sum_{J}d_{{L}}f_{IJ}|_{F_{M}(\overline{\mathcal{G}^{*}M)}} \, \wedge \,dz_{I}\wedge d\overline{z_{J}}\,, \quad 
d_{\overline{L}}\omega=\sum_{I}d_{\overline{L}}f_{IJ}|_{F_{M}(\mathcal{G}^{*}M)}\, \wedge \,dz_{I}\wedge d\overline{z_{J}}\,.   
 \end{equation}
By the equations \eqref{form of D} and \eqref{form of d-l bar}, on $A^{\bullet,\bullet}$, we have
$$D=d_{\overline{L}}+d_{L}\quad\text{and}\quad D(A^{\bullet,\bullet})\subseteq A^{\bullet+1,\bullet}\oplus A^{\bullet,\bullet+1}\,.$$
 Similarly,  one can see that $\tilde{D}=\tilde{d}_{\overline{L}}+\tilde{d}_{L}$ where $\tilde{d}_{\overline{L}}$ and $\tilde{d}_{L}$ are considered as a morphism of sheaves between $\tilde{A}^{\bullet,\bullet}$ to $\tilde{A}^{\bullet+1,\bullet}$ and $\tilde{A}^{\bullet,\bullet+1}$, respectively.
\begin{definition}
    Any element $\omega\in\tilde{A}^{l}$ is called a generalized form of degree $l$ and any element in $\tilde{A}^{p,q}$ is called a generalized form of type $(p,q)$. Here $\tilde{A}^{p,q},\tilde{A}^{l}$ are as in \eqref{A}.
\end{definition}
\begin{definition}
    Any element $\omega\in A^{l}$ is called a transverse generalized form of degree $l$ and any element in $A^{p,q}$ is called a transverse generalized form of type $(p,q)$. Here $A^{p,q},A^{l}$ are as in \eqref{A}.
\end{definition}
Let $Z^{\bullet}=\ker(D:A^{\bullet}\longrightarrow A^{\bullet+1})$, i.e., the set of $D$-closed transverse generalized forms of degree $l$.  Let 
$B^{\bullet}(M)=\img(D:A^{\bullet-1}(M)\longrightarrow A^{\bullet}(M))$, i.e.,  the set of $D$-exact  transverse generalized forms of degree $l$. Then, the homology of the cochain complex $\{A^{\bullet}(M),D\}$  is called the $D$-cohomology of $M$, and it is denoted by
\begin{equation}\label{d-cohomo}
    H_{D}^{\bullet}(M) := \frac{Z^{\bullet}(M)}{B^{\bullet}(M)}=\frac{\ker(D:A^{\bullet}(M)\longrightarrow A^{\bullet+1}(M))}{\img(D:A^{\bullet-1}(M)\longrightarrow A^{\bullet}(M))}\,.
\end{equation}
\medskip

\begin{definition} Let $(\mathbf{\mathcal{G}^{*}M})^{p}:=\bigwedge^{p}_{\mathcal{O}_{M}}\mathbf{\mathcal{G}^{*}M}$ for $p \in \N$ and $(\mathbf{\mathcal{G}^{*}M})^{0}:=\mathcal{O}_{M}$. Note that $(\mathbf{\mathcal{G}^{*}M})^{\bullet}<F_{M}(\wedge^{\bullet}\mathcal{G}^{*}M)$. We say  that a transverse generalized form $\omega$ of type $(p,0)$ is a GH $p$-form if $d_{L}\omega=0$, that is, $\omega\in (\mathbf{\mathcal{G}^{*}M})^{p}$.   
\end{definition}

Let $N$ be another regular GC manifold and let $f: M\longrightarrow N$ be a GH map. Then it follows that 
\begin{enumerate}
\setlength\itemsep{0.5em}
    \item $f^{*}(A^{\bullet,\bullet}_{N})\subset A^{\bullet,\bullet}_{M}$,
    \item $f^{*}\circ d_{L_{N}}=d_{L_{M}}\circ f^{*}$.
\end{enumerate}

\begin{corollary}\label{imp corr}
    Let $M$ be a GC manifold. Given an open set $U\subseteq M$, a smooth map $\psi:(U,\mathcal{J}_{U})\longrightarrow\C$ is a GH function, that is, $f\in\mathcal{O}_{M}(U)$, if and only if $d_{L}f=0$ where $d_{L}$ as defined in \eqref{d-L}.
\end{corollary}
\begin{proof}
Follows from Lemma \ref{imp lemma2}.
\end{proof}
Let $Z^{\bullet,\bullet}=\ker(d_{L}:A^{\bullet,\bullet}\longrightarrow A^{\bullet,\bullet+1})$  and let 
$B^{\bullet,\bullet}(M)=\img(d_{L}:A^{\bullet,\bullet-1}(M)\longrightarrow A^{\bullet,\bullet}(M))$. Then the homology of the cochain complex $\{A^{\bullet,\bullet}(M),d_{L}\}$ is called \textit{$d_{L}$-cohomology of $M$} and it is denoted by
\begin{equation}\label{d-L-cohomo}
    H_{d_{L}}^{\bullet,\bullet}(M):=\frac{Z^{\bullet,\bullet}(M)}{B^{\bullet,\bullet}(M)} =\frac{\ker(d_{L}:A^{\bullet,\bullet}(M)\longrightarrow A^{\bullet,\bullet+1}(M))}{\img(d_{L}:A^{\bullet,\bullet-1}(M)\longrightarrow A^{\bullet,\bullet}(M))}\,.
\end{equation}
One can also consider the homology of the cochain complex $\{\tilde{A}^{\bullet}(M),\tilde{D}\}$ which is called the $\tilde{D}$-cohomology of $M$, and is denoted by
\begin{equation*}
 H_{\tilde{D}}^{\bullet}(M):=\frac{\ker(\tilde{D}:\tilde{A}^{\bullet}(M)\longrightarrow \tilde{A}^{\bullet+1}(M))}{\img(\tilde{D}:\tilde{A}^{\bullet-1}(M)\longrightarrow \tilde{A}^{\bullet}(M))}\,.    
\end{equation*}
Similarly, the homology of the cochain complex $\{\tilde{A}^{\bullet,\bullet}(M),\tilde{d}_{L}\}$ which will be called  \textit{$\tilde{d}_{L}$-cohomology of $M$}, and denoted by 
\begin{equation*}
  H_{\tilde{d}_{L}}^{\bullet,\bullet}(M):=\frac{\ker(\tilde{d}_{L}:\tilde{A}^{\bullet,\bullet}(M)\longrightarrow \tilde{A}^{\bullet,\bullet+1}(M))}{\img(\tilde{d}_{L}:\tilde{A}^{\bullet,\bullet-1}(M)\longrightarrow \tilde{A}^{\bullet,\bullet}(M))}\,.  
\end{equation*}
We know that locally (cf. \eqref{loc coordi}), 
$$C^{\infty}(\overline{\mathcal{G}^{*}M}|_{U})=\spn_{C^{\infty}(U)}\{d\overline{z_1},\ldots,d\overline{z_{k}}\}\,,$$ and 
$$C^{\infty}(\mathcal{G}^{*}M|_{U})=\spn_{C^{\infty}(U)}\{dz_1,\ldots,dz_{k}\}\,,$$ where $k$ is the type of $M$. Then, by following \cite[P-25\,,\,P-42]{griffiths}, one immediately obtains the result below.

\begin{prop}
Let $M$ be a regular GC manifold of type $k$. Then for any $q>0$,

\vspace{0.5em}
  \begin{enumerate}
  \setlength\itemsep{1em}
      \item $\tilde{d}_{L}$-Poincar\'{e} Lemma: For sufficiently small open set $U\subset M$, $ H_{\tilde{d}_{L}}^{\bullet,q}(U)=0\,.$
      \item $H^{q}(M,\tilde{A}^{\bullet,\bullet})=0\,.$
      \item $\tilde{D}$-Poincar\'{e} Lemma: For a sufficiently small open set $U\subset M$, $ H_{\tilde{D}}^{q}(U)=0\,.$
      \item $H^{q}(M,\tilde{A}^{\bullet})=0\,.$
  \end{enumerate}  
\end{prop}

\begin{definition}
    An open cover $\mathcal{U}=\{U_{\alpha}\}$ of $M$ is called a transverse good cover if $\mathcal{U}$ is a locally finite transverse open cover (cf. Definition \ref{basic open set}) and any finite intersection $\bigcap^{l}_{i=0}U_{\alpha_{i}}$ is diffeomorphic to a tubular neighborhood as in Theorem \ref{orbi thm}.
\end{definition}
\begin{prop}\label{prop}
 Let $M$ be a regular GC manifold of type $k$. Assume $M/\mathscr{S}$ has a smooth orbifold structure. Let $\mathcal{U}=\{U_{\alpha}\}$ be a sufficiently fine transverse good cover of $M$. Then for any $q>0$,
 \vspace{0.5em}
  \begin{enumerate}
  \setlength\itemsep{1em}
      \item $d_{L}$-Poincar\'{e} Lemma: For a sufficiently small transverse open set $U\subset M$, $$ H_{d_{L}}^{\bullet,q}(U)=0\,.$$
      \item $D$-Poincar\'{e} Lemma: For a sufficiently small transverse open set $U\subset M$, $$ H_{D}^{q}(U)=0\,.$$
      \item $H^{q}(\mathcal{U},A^{\bullet,\bullet})=0\,.$
      \item $H^{q}(\mathcal{U},A^{\bullet})=0\,.$
      \end{enumerate} 
\end{prop}
\begin{proof}
    By Theorem \ref{orbi thm}, there exists a transverse open set (tubular neighborhood) $U$ around a leaf $S$ which is diffeomorphic to $\tilde{S}\times_{\hol(S)}\C^{k}$ where $\tilde{S}$ is the universal cover of $S$ and $\hol(S)$ is the holonomy group of $S$. Since $\hol(S)$ is finite, it acts linearly. Recall that 
$\mathcal{N}^{*}\otimes\C=\mathcal{G}^{*}M\oplus\overline{\mathcal{G}^{*}M}$ where $\mathcal{N}$ is the normal bundle of $\mathscr{S}$. Taking $U$ to be sufficiently small, we have
    $$F_{M}(\overline{\mathcal{G}^{*}M}|_{U})=\spn_{F_{M}(U)}\{d\overline{z_1},\ldots,d\overline{z_{k}}\}\,,$$ and 
    $$F_{M}(\mathcal{G}^{*}M|_{U})=\spn_{F_{M}(U)}\{dz_1,\ldots,dz_{k}\}\,.$$
    Then following the proof in \cite[P-25\,,\,P-42]{griffiths}, we can prove $(1)$ and $(2)$.
    \vspace{0.3em}
    \\ 
    To prove $(3)$ and $(4)$, it is enough to show that there is a partition of unity subordinate to\, $\mathcal{U}$ \, such that they are constant along the leaves. This is obtained easily by pulling back a partition of unity for   
    $M/\mathscr{S}$ subordinate to $\hat{\mathcal{U}}=(\tilde{\pi}(U_{\alpha}))$ with respect to the quotient map $\tilde{\pi}:M\longrightarrow M/\mathscr{S}$.
    \end{proof}  

\begin{prop}(de Rham cohomology for regular GC manifold)\label{prop2}
Let $M$ be a regular GC manifold with induced symplectic foliation $\mathscr{S}$. Assume the leaf space $M/\mathscr{S}$ admits a smooth orbifold structure. Then for $q\geq 0$,
    $$H^{q}(\mathcal{U},\{\C\})\cong H_{D}^{q}(M)\,,$$ where $\{\C\}$ is the sheaf of locally constant $\C$-valued functions and $\mathcal{U}$ is a sufficiently fine transverse good cover of $M$.
\end{prop}
\begin{proof}
     By $D$-Poincar\'{e} Lemma, we have the following exact sequence of sheaves
\begin{equation}\label{exact ses }
  \begin{tikzcd}
0 \arrow[r] & \{\C\} \arrow[r, hookrightarrow] & A^{0} \arrow[r, "D"] &
A^{1} \arrow[r, "D"] &
\cdots
\end{tikzcd} 
\end{equation}
on $M$.  This gives the following exact sequence,
\begin{equation}\label{exact ses2}
\begin{tikzcd}
0 \arrow[r] & Z^{\bullet} \arrow[r, hookrightarrow] & A^{\bullet} \arrow[r, "D"] & Z^{\bullet+1} \arrow[r] & 0 \,.
\end{tikzcd} 
\end{equation}
 In particular, the sequence
\begin{equation}\label{exact ses1}
  \begin{tikzcd}
0 \arrow[r] & \{\C\} \arrow[r, hookrightarrow] & A^{0} \arrow[r, "D"]  & Z^{1} \arrow[r] & 0\,
\end{tikzcd} 
\end{equation} is exact.
By $(4)$ in Proposition \ref{prop}, $H^{q}(\mathcal{U},A^{\bullet})=0$ for all $q>0$. Thus, considering the associated long exact sequences in cohomology for these exact sequences of sheaves, as in \cite[pp. 40-41, 44]{griffiths}, we  obtain that for all $q\geq 0$ 
\begin{align*}
    H^{q}(\mathcal{U},\{\C\})&\cong H^{q-1}(\mathcal{U},Z^{1})\,\,\,\,\,(\text{by}\,\,\,\eqref{exact ses1})\\
    &\cong H^{q-2}(\mathcal{U},Z^{2})\,\,\,\,\,(\text{by}\,\,\,\eqref{exact ses2})\\
    &\vdots\\
    &\cong H^{1}(\mathcal{U},Z^{q-1})\,\,\,\,\,(\text{by}\,\,\,\eqref{exact ses2})\\
    &\cong \frac{H^{0}(\mathcal{U},Z^{q})}{D(H^{0}(\mathcal{U},Z^{q-1}))}\\
    &= \frac{Z^{q}(M)}{B^{q}(M)}=H_{D}^{q}(M)\,.
\end{align*}
\end{proof}
\begin{theorem}(Dolbeault cohomology for regular GC manifold)\label{main5}
    Let $M$ be a regular GC manifold with induced symplectic foliation $\mathscr{S}$. Assume that the leaf space $M/\mathscr{S}$ admits an orbifold structure. Then for any $p,q\geq 0$,
    $$H^{q}(M,(\mathbf{\mathcal{G}^{*}M})^{p})\cong H_{d_{L}}^{p,q}(M)\,.$$
\end{theorem}
\begin{proof}
By $d_{L}$-Poincar\'{e} Lemma, the following sequences of sheaves, 
    \begin{equation}\label{exact 1}
  \begin{tikzcd}
0 \arrow[r] & (\mathbf{\mathcal{G}^{*}M})^{\bullet} \arrow[r, hookrightarrow] & A^{\bullet,0} \arrow[r, "d_{L}"] & Z^{\bullet,1} \arrow[r] & 0\,,
\end{tikzcd} 
\end{equation}
\begin{equation}\label{exact 2}
\begin{tikzcd}
0 \arrow[r] & Z^{\bullet,\bullet} \arrow[r, hookrightarrow] & A^{\bullet,\bullet} \arrow[r, "d_{L}"] & Z^{\bullet,\bullet+1} \arrow[r] & 0
\end{tikzcd} 
\end{equation}
are exact. {Let $\mathcal{U}$ be a sufficiently fine  transverse good cover of $M$. By $(3)$ in Proposition \ref{prop}, $H^{r}(\mathcal{U},A^{\bullet,\bullet})=0$ for all $r>0$. Hence, using the long exact sequences in cohomology associated with \eqref{exact 1} and \eqref{exact 2} (cf. \cite[pp. 40-41]{griffiths}), we have,
\begin{align*}
    H^{q}(\mathcal{U},(\mathbf{\mathcal{G}^{*}M})^{p})&\cong H^{q-1}(\mathcal{U},Z^{p,1})\,\,\,\,\,(\text{by}\,\,\,\eqref{exact 1})\\
    &\cong H^{q-2}(\mathcal{U},Z^{p,2})\,\,\,\,\,(\text{by}\,\,\,\eqref{exact 2})\\
    &\vdots\\
    &\cong H^{1}(\mathcal{U},Z^{p,q-1})\,\,\,\,\,(\text{by}\,\,\,\eqref{exact 2})\\
    &\cong \frac{H^{0}(\mathcal{U},Z^{p,q})}{d_{L}(H^{0}(\mathcal{U},Z^{p,q-1}))}\\
    &= \frac{Z^{p,q}(M)}{B^{p,q}(M)}=H_{d_{L}}^{p,q}(M)\,.
\end{align*}
We can choose $\mathcal{U}=\{U_{\alpha}\}$ such that any finite intersection $V=\bigcap^{l}_{i=0}U_{\alpha_{i}}$ is diffeomorphic a tubular neighborhood as in Theorem \ref{orbi thm}. Fix such a $V$. Then $\mathcal{V}:=\{V\cap U_{\alpha}\}$ is a  transverse good cover of $V$. Note that $H^{q}(V,A^{\bullet,\bullet})=0$. Then, as above, $$H^q(\mathcal{V},(\mathbf{\mathcal{G}^{*}M})^{p}|_{V})=H_{d_{L}}^{p,q}(V)\,.$$ Since $H_{d_{L}}^{p,q}(W)=0$  for any finite intersection $W$ of elements in $\mathcal{V}$ by the $d_{L}$-Poincar\'{e} Lemma, using Leray's theorem we have, 
$$ H^q({V},(\mathbf{\mathcal{G}^{*}M})^{p}|_{V})= H^q(\mathcal{V},(\mathbf{\mathcal{G}^{*}M})^{p}|_{V})=H_{d_{L}}^{p,q}(V)\,.$$  
Again, by the $d_{L}$-Poincar\'{e} Lemma, $H_{d_{L}}^{p,q}(V)=0$. Thus, by Leray's Theorem, for all $p,q\geq 0$, $$H^q(\mathcal{U},(\mathbf{\mathcal{G}^{*}M})^{p})\cong H^q(M,(\mathbf{\mathcal{G}^{*}M})^{p})\,.$$}
\end{proof}

Let $E$ be an SGH vector bundle over $M$. We put 
\begin{equation}\label{A-E}
A_{E}:=A\otimes_{\mathcal{O}_{M}}\mathbf{E}\,,
\end{equation}
 where $A$ is as in \eqref{A}. Since $A$ is an $F_{M}$-module, it is also $\mathcal{O}_{M}$-module, and thus, $A_{E}$ is well-defined. We can naturally extend $d_{L}$ from $A$ to $A_{E}$,  denoted by $d_{L}^{'}$, as follows: For $f\in F_{M}$, $\alpha\in A$ and $\beta\in\mathbf{E}$, we can define $d_{L}^{'}(f\alpha\wedge\beta)=fd_{L}(\alpha)\wedge\beta\,.$ This definition is well-defined because if $f\in\mathcal{O}_{M}$, we have $f\alpha\wedge\beta=\alpha\wedge f\beta$. Then 
 \begin{align*}
     d_{L}^{'}(f\alpha\wedge\beta)&=fd_{L}\alpha\wedge\beta\quad(\text{as $d_{L}f=0$})\\
     &=d_{L}\alpha\wedge(f\beta)\\
     &=d_{L}^{'}(\alpha\wedge f\beta)\,.
 \end{align*}
 {For notational convenience, we again denote by $d_{L}\,$ or by $d_{L,E}\,$, the natural extension $d_L^{\prime}$ of the operator $d_L$.} We denote the component of type $(p,q)$ of the cohomology of the complex $ (H^0(M,A_E), d_L) $ by $H_{d_{L}}^{p,q}(M, E)\,.$ Note that tensoring the short exact sequences \eqref{exact 1} and \eqref{exact 2} with the locally free sheaf $\mathbf{E}$ again yields short exact sequences.  Then, following the proof of Theorem \ref{main5}, we get the following.
\begin{corollary}\label{cor5}
   $H_{d_{L}}^{p,q}(M,E)\cong H^{q}(M,(\mathbf{\mathcal{G}^{*}M})^{p}\otimes_{\mathcal{O}_{M}}\mathbf{E})$. 
\end{corollary}

\subsection{Cohomology class of the curvature}\label{curv sec}
Let $G$ be a complex Lie group with complex Lie algebra $\mathfrak{g}$. Let $G\hookrightarrow P\longrightarrow M$ be an SGH principal bundle. Then, by applying Corollary \ref{cor5} to $\mathbf{Ad(P)}$ where $Ad(P)$ is as in \eqref{atiyah GH bundle}, we have the following.
\begin{corollary}\label{main6}
   $H^{q}(M,(\mathbf{\mathcal{G}^{*}M})^{p}\otimes_{\mathcal{O}_{M}}\mathbf{Ad(P)})\cong H_{d_{L}}^{p,q}(M,Ad(P))$.
\end{corollary}
Let $\Theta = \{\Theta_{\alpha} \}\in \tilde{A}^{1,0}\otimes_{\mathcal{O}_{M}}\mathbf{Ad(P)}$ be a smooth generalized connection on $P$ (see  Definition \ref{def:gen conn} and Section \ref{connection}), where   
$$\Theta_{\alpha}\in C^{\infty}(U_{\alpha},
\Hom_{C^{\infty}_{M}}(C^{\infty}({\mathcal{G}M})\,,\,C^{\infty}(M_{\mathfrak{g}})))\,$$ on a trivializing neighborhood $U_{\alpha}$ of $P$.
Then, the curvature of this smooth generalized connection is defined on $U_{\alpha}$ by
\begin{equation}\label{crvtr}  \Omega_{\alpha}:=\tilde{D}\Theta_{\alpha}+\frac{1}{2}[\Theta_{\alpha},\Theta_{\alpha}]
\end{equation} where $\Theta_{\alpha}$ is considered as $\mathfrak{g}$ valued function. Then, by using either of the equations \eqref{co boundary equation1-2} or \eqref{co boundary equation2-2}, we get
\begin{equation}\label{curvature2}
    \begin{aligned} \Omega_{\alpha}&=\ad(\phi_{\alpha\beta}) \,\Omega_{\beta}\,\,\,\,\,\text{on}\,\,\,U_{\alpha}\cap U_{\beta}\,,\\
    &\text{or}\\
\Omega_{\beta}&=\ad(\phi_{\beta\alpha})\, \Omega_{\alpha}\,\,\,\,\,\text{on}\,\,\,U_{\alpha}\cap U_{\beta}\,.
    \end{aligned}
\end{equation}
So, after patching, we get an element $\Omega\in C^{\infty}(M,\wedge^2(\mathcal{G}M\oplus\overline{\mathcal{G}M})^{*})\otimes Ad(P))$. $\Omega$ is called the curvature of the smooth generalized connection $\Theta$.

\medskip
Now we consider a more special type of smooth generalized connection whose associated $1$-form is locally constant along the leaves. Let $\Theta\in A_{Ad(P)}^{1,0}$ be such that   $\Theta_{\alpha} \in A_{M_{\mathfrak{g}}}^{1,0}$ for every $\alpha$. 
This type of connection always exists due to the existence of partition of unity on the orbifold leaf space $M/\mathscr{S}$. We reformulate the definition of curvature (see \eqref{crvtr}) for such a smooth generalized connection  as
\begin{equation}\label{curvature}  \Omega_{\alpha}=D\Theta_{\alpha}+\frac{1}{2}[\Theta_{\alpha},\Theta_{\alpha}]\,\quad\text{on}\,\,\,U_{\alpha}\,.
\end{equation}
It follows  that the $(1,1)$ component of $\Omega_{\alpha}$, denoted by $\Omega^{1,1}_{\alpha}$, is given by 
\begin{equation}\label{curvature class}
\Omega^{1,1}_{\alpha}=d_{L}\Theta_{\alpha}\,\quad\text{on}\,\,\,U_{\alpha}\,.    
\end{equation}
By \eqref{curvature2}, we have
\begin{equation}\label{curvature3}
    \begin{aligned} \Omega^{1,1}_{\alpha}&=\ad(\phi_{\alpha\beta})\Omega^{1,1}_{\beta}\,\,\,\,\,\text{on}\,\,\,U_{\alpha}\cap U_{\beta}\,,\\
    &\text{or}\\
\Omega^{1,1}_{\beta}&=\ad(\phi_{\beta\alpha})\Omega^{1,1}_{\alpha}\,\,\,\,\,\text{on}\,\,\,U_{\alpha}\cap U_{\beta}\,.
    \end{aligned}
\end{equation}
After patching, we get a global element $\Omega\in\left(F_{M}(\wedge^2(\mathcal{G}^{*}M\oplus\overline{\mathcal{G}^{*}M}))\otimes_{\mathcal{O}_{M}}\mathbf{Ad(P)}\right)(M)$ whose $(1,1)$-component is $\Omega^{1,1}$.
Then from the equations \eqref{co boundary equation1-2}, \eqref{co boundary equation2-2}, \eqref{curvature class} and \eqref{curvature3}, we can see that the $d_{L}$-cohomology class $[\Omega^{1,1}]$ of $\Omega^{1,1}$ is independent of the choice of a smooth generalized connection of type $(1,0)$. Note that $[\Omega^{1,1}]$ maps to $a(P)$, as defined in Theorem \ref{main3}, via the isomorphism in Corollary \ref{main6}. We summarise our results as follows.
\begin{theorem}\label{main7}
    Let $P$ be an SGH principal $G$-bundle over a regular GC manifold $M$ where $G$ is a complex Lie group.  Assume that the leaf space of the induced symplectic foliation on $M$ admits a smooth orbifold structure. Let $\Theta$ be a smooth generalized connection of type $(1,0)$ on $P$, which is constant along the leaves. Let $\Omega^{1,1}$ denote the corresponding $(1,1)$ component of the curvature. Let $[\Omega^{1,1}]$ be the $d_{L}$-cohomology class in $H_{d_{L}}^{1,1}(M,Ad(P))$. Then $[\Omega^{1,1}]$ corresponds to $a(P)\in H^{1}(M,\mathbf{\mathcal{G}^{*}M}\otimes_{\mathcal{O}_{M}}\mathbf{Ad(P)})$ via the isomorphism in Corollary \ref{main6}.
\end{theorem}
\section{Generalized Chern-Weil Theory and characteristic class}\label{sec c-w}

Let $\sym^{k}(\mathfrak{g}^{*})$ denotes the set of all symmetric $k$-linear mappings $\mathfrak{g}\times\mathfrak{g}\times\cdots\times\mathfrak{g}\longrightarrow\C$ on the Lie algebra $\mathfrak{g}$.
Define the right adjoint action of the Lie group $G$ on $\sym^{k}(\mathfrak{g}^{*})$ by 
\begin{equation}\label{adj action}
\begin{aligned}
&(f, g) \mapsto \ad(g^{-1})f\\
\text{where}\quad(\ad(g^{-1})f)(x_1,&\ldots,x_k)=f(\ad(g^{-1})x_1,\ldots,\ad(g^{-1})x_{k})
\end{aligned}
\end{equation}
for any $f\in\sym^{k}(\mathfrak{g}^{*})$ and $g\in G$. Denote the space of  $\ad(G)$-invariant forms by
\begin{equation}\label{adj invr form}
\sym^{k}(\mathfrak{g}^{*})^{G}:=\{f\in\sym^{k}(\mathfrak{g}^{*})\,|\,\ad(g^{-1})f=f\,\,\forall\,\,g\in G\}    
\end{equation}
Now, given any $f\in\sym^{k}(\mathfrak{g}^{*})^{G}$, we define a $2k$-form in $A^{2k}$, of type $(k,k)$, by
\begin{equation}\label{f-form}
    f(\Omega^{1,1})(X_1,X_2\ldots,X_{2k}):=\frac{1}{(2k)!}\sum_{\sigma}\epsilon_{\sigma}f(\Omega^{1,1}(X_{\sigma(1)},X_{\sigma(2)}),\ldots,\Omega^{1,1}(X_{\sigma(2k-1)},X_{\sigma(2k)}))
\end{equation}
for $X_1,\ldots,X_{2k}\in F_{M}(\mathcal{G}M\oplus\overline{\mathcal{G}M})$, where $\sigma$ is an element of the symmetric group $S_{2k}$, $\epsilon_{\sigma}$ denotes the sign of the permutation $\sigma\in S_{2k}$, and $\Omega^{1,1}$ is defined as in Theorem \ref{main7}.

\medskip
Let $\C[\mathfrak{g}]$ be the algebra of $\C$-valued polynomials on $\mathfrak{g}$. Consider the same adjoint action of $G$ on $\C[\mathfrak{g}]$ as in \eqref{adj action}, and let $\C[\mathfrak{g}]^{G}$ denote the subalgebra of fixed points under this action. Then any $f\in\sym^{k}(\mathfrak{g}^{*})^{G}$ can be viewed as a homogeneous polynomial function of degree $k$ in $\C[\mathfrak{g}]^{G}$ that is, 
$$\sym^{k}(\mathfrak{g}^{*})^{G}<\C[\mathfrak{g}]^{G}\quad\text{for any $k\geq 0$}\,.$$ 
We set  $$\sym(\mathfrak{g}^{*})^{G}:=\bigoplus^{\infty}_{k=0}\sym^{k}(\mathfrak{g}^{*})^{G}\,.$$ Then, $\sym(\mathfrak{g}^{*})^{G}$ can be viewed as a sub-algebra of $\C[\mathfrak{g}]^{G}$.

\medskip
Since $d_{L}\Omega^{1,1}=0$, we can see that $d_{L}f(\Omega^{1,1})=0$ for any $f\in\sym^{k}(\mathfrak{g}^{*})^{G}$. Thus $f(\Omega^{1,1})\in H_{d_{L}}^{k,k}(M)$. We define a map
\begin{equation}\label{CW-hom map at k}
\begin{aligned}
\Phi_{k}:\sym^{k} (\mathfrak{g}^{*})^{G}\longrightarrow H_{d_{L}}^{k,k}(M)\,,\quad
f\mapsto [f(\Omega^{1,1})]\,.  
\end{aligned}
\end{equation}
Using the algebra structure of $\sym(\mathfrak{g}^{*})^{G}$, we extend the map in \eqref{CW-hom map at k} to an algebra homomorphism
\begin{equation}\label{CW-hom map}
\begin{aligned}
\Phi:\sym (\mathfrak{g}^{*})^{G}\longrightarrow H_{d_{L}}^{*}(M)\,,\quad
f\mapsto [f(\Omega^{1,1})]\,,  
\end{aligned}
\end{equation}
where $H_{d_{L}}^{*}(M):=\bigoplus_{k,l}H_{d_{L}}^{k,l}(M)$. Note that $\img\Phi\subseteq\bigoplus_{k\geq 0}H_{d_{L}}^{k,k}(M)$.
We follow the approach in classical Chern-Weil theory (see \cite{kobayashi69}) to show that $\Phi$ is independent of the choice of a smooth generalized connection of type $(1,0)$  constant along the leaves.  Consider two smooth generalized connections $\Theta$, $\Theta'$ of type $(1,0)$ on the SGH principal bundle $P$  over $M$ which are constant along the leaves. Define
\begin{align*}
    \omega :=\Theta-\Theta' \,, \quad {\rm and} \quad
    \omega_{t}:=\Theta'+t\omega\,\,\,\text{for $t\in [0,1]$}\,.
\end{align*}
From the equations \eqref{co boundary equation1-2} and \eqref{co boundary equation2-2}, one can see that $\omega_{t}$ is a $1$-parameter family of smooth generalized connections of type $(1,0)$ constant along the leaves. Let $\Omega_{t}$ be the curvature of $\omega_{t}$ and let $\Omega^{1,1}_{t}$ be the $(1,1)$ component of $\Omega_{t}$. 
By \eqref{curvature class}, we have
\begin{equation}\label{omega-t-omega}
    \begin{aligned}
\Omega^{1,1}_{t}&=d_{L}\omega_{t} = d_{L}\Theta'+t\,d_{L}\omega\, 
\implies\,\,\frac{d\Omega^{1,1}_{t}}{dt}&=d_{L}\omega\,.
    \end{aligned}
\end{equation}
Consider the transverse generalized  $(2k-1)$-form of type $(k,k-1)$, defined by 
\begin{equation}
\varphi=k\int_{0}^{1}f(\omega,\Omega^{1,1}_{t},\ldots,\Omega^{1,1}_{t})\,\,dt\,. 
\end{equation}
Using \eqref{omega-t-omega}, we have
\begin{equation}
    \begin{aligned}
k\,d_{L}f(\omega,\Omega^{1,1}_{t},\ldots,\Omega^{1,1}_{t}) = kf(d_{L}\omega,\Omega^{1,1}_{t},\ldots,\Omega^{1,1}_{t})\,
=\frac{d}{dt}(f(\Omega^{1,1}_{t},\ldots,\Omega^{1,1}_{t}))\,.
    \end{aligned}
\end{equation}
Hence,
$$d_{L}\varphi=\int_{0}^{1}\frac{d}{dt}(f(\Omega^{1,1}_{t},\ldots,\Omega^{1,1}_{t}))\,\,dt\,=\,f(\Omega^{1,1}_{1},\ldots,\Omega^{1,1}_{1})-f(\Omega^{1,1}_{0},\ldots,\Omega^{1,1}_{0})\,.$$

\begin{definition}\label{def:c-w}
    The algebra homomorphism $\Phi$, defined in \eqref{CW-hom map}, is called the \textit{generalized Chern-Weil} homomorphism.
\end{definition}
\subsection{Generalized Chern class}\label{char class of SGH PB}
Let $P\longrightarrow M$ be an SGH principal $G$-bundle over a regular GC manifold $M$. Let $G$ be a complex Lie group with a canonical faithful representation such as a classical complex Lie group. Then the complex Lie algebra $\mathfrak{g}$ is identified with a complex subalgebra of $M_{l}(\C)$ where $l$ is the dimension of the representation. For any $A\in\mathfrak{g}$\,, consider the following characteristic polynomial
\begin{equation}\label{char poly}
\det\left(I+t\frac{A}{2\pi\,i}\right)=\sum_{k=0}^{l}  f_{k}(A)\,t^{k}\,, 
\end{equation}
where $f_{k}\in\C[\mathfrak{g}]$ is an elementary symmetric polynomial of degree $k$ and $I$ is the identity matrix. Since the right hand side of $\eqref{char poly}$ is invariant under $\ad(G)$-action, we have $$f_{k}\in\sym^{k}(\mathfrak{g}^{*})^{G}\,.$$ 
Following  \cite[Example 32.3]{tu17}, we define an analogue of Chern classes for an SGH principal $G$-bundles where $G$ is a complex Lie group with a holomorphic faithful representation.
\begin{definition}
    The $k$-th \textit{generalized Chern class} of $P$, denoted by $\mathbf{g}c_{k}(P)$, is defined as the image of $f_{k}$ under the generalized Chern-Weil homomorphism, that is,
    $$\mathbf{g}c_{k}(P):=\Phi(f_{k})\,,$$ where $\Phi$ as defined in \eqref{CW-hom map}.
\end{definition}
 \begin{prop}
 $\mathbf{g}c_{1}(P)=\left[\left(\frac{1}{2\pi\,i}\right)\tr(\Omega^{1,1})\right]\,,$ where $\Omega^{1,1}$ as defined in \eqref{curvature class}.    
 \end{prop}
 \begin{proof} 
 Follows from a straightforward calculation (cf. \cite[Chapter XII]{kobayashi69}, \cite[Chapter 2]{kobayashi14}, \cite[Chapter 6]{tu17}).
 \end{proof}

 \section{Connection on SGH vector bundle}\label{sec sgh vb}
 Let $M$ be a regular GC manifold and $E$ be an SGH vector bundle over $M$. Set
 \begin{equation}\label{tilde A-E}
\tilde{A}_{E}:=\tilde{A}\otimes_{C^{\infty}_{M}}C^{\infty}(E)=\tilde{A}\otimes_{\mathcal{O}_{M}}\mathbf{E}\,,     
 \end{equation} where $\tilde{A}$ is defined as in \eqref{A}. 
 \begin{definition}
A smooth generalized connection on an SGH vector bundle $E$, is a $\C$-linear sheaf homomorphism 
$$\nabla:\tilde{A}^0_{E}\longrightarrow\tilde{A}^1_{E}$$ which satisfies the Leibniz rule
$$\nabla(fs)=\tilde{D}f\otimes s+f\nabla(s)$$ for any local function on $M$ and any local section $s$ of $E$.
 \end{definition}
Let $\{U_{\alpha},\phi_{\alpha}\}$ be a system of local trivializations of $E$. Then, on $U_{\alpha}$, we may write
$$\nabla|_{U_{\alpha}}=\phi^{-1}_{\alpha}\circ(\tilde{D}+\theta_{\alpha})\circ\phi_{\alpha}$$ where $\theta_{\alpha}$ is a matrix valued generalized $1$-form. On $U_{\alpha}\cap U_{\beta}$, we have
\begin{equation*}
    \begin{aligned}
&\phi^{-1}_{\beta}\circ(\tilde{D}+\theta_{\beta})\circ\phi_{\beta}=\phi^{-1}_{\alpha}\circ(\tilde{D}+\theta_{\alpha})\circ\phi_{\alpha}\,,\\
\implies &\phi^{-1}_{\beta}\circ\theta_{\beta}\circ\phi_{\beta}-\phi^{-1}_{\alpha}\circ\theta_{\alpha}\circ\phi_{\alpha}=\phi^{-1}_{\alpha}\circ\tilde{D}\circ\phi_{\alpha}-\phi^{-1}_{\beta}\circ\tilde{D}\circ\phi_{\beta}\,,\\
\implies &
\left\{
   \begin{array}{l}
\phi^{-1}_{\beta}\circ\theta_{\beta}\circ\phi_{\beta}-\phi^{-1}_{\alpha}\circ\theta_{\alpha}\circ\phi_{\alpha}=\phi^{-1}_{\beta}\circ(\phi^{-1}_{\alpha\beta}\circ\tilde{D}\circ \phi_{\alpha\beta}-\tilde{D})\circ\phi_{\beta}\,,\\
\quad\text{or}\\
\phi^{-1}_{\beta}\circ\theta_{\beta}\circ\phi_{\beta}-\phi^{-1}_{\alpha}\circ\theta_{\alpha}\circ\phi_{\alpha}=\phi^{-1}_{\alpha}\circ(\tilde{D}-\phi_{\alpha\beta}\circ\tilde{D}\circ \phi^{-1}_{\alpha\beta})\circ\phi_{\alpha}\,.\\
   \end{array}
\right.
\end{aligned}
\end{equation*}
Thus, we get the following co-boundary equation for $\nabla$.
\begin{equation}\label{nabla coboundary eq}
    \begin{aligned}
    \left\{
    \begin{array}{l}
    \theta_{\beta}-\ad(\phi_{\beta\alpha})\cdot\theta_{\alpha}=\phi^{-1}_{\alpha\beta}\tilde{D}( \phi_{\alpha\beta})\\
    \quad\text{or}\\
\ad(\phi_{\alpha\beta})\cdot\theta_{\beta}-\theta_{\alpha}=-\phi_{\alpha\beta}\tilde{D}( \phi^{-1}_{\alpha\beta})  
    \end{array}
  \right.\,\quad\text{on\,\,\,$U_{\alpha}\cap U_{\beta}$}\,.
    \end{aligned}
\end{equation}
A straightforward modification of the proof of \cite[Proposition 4.2.3]{hubrechts05} yields the following.
\begin{prop}\label{generalized prop}
Let $E$ be an SGH vector bundle over $M$. Then

\vspace{0.2em}
\begin{enumerate}
\setlength\itemsep{1em}
    \item For any two smooth generalized connections $\nabla,\nabla^{'}$ on an SGH vector bundle $E$, $\nabla-\nabla^{'}$ is $\tilde{A}^0$-linear.
    \item For any $\theta\in\tilde{A}^1_{\en(E)}(M)$, $\nabla+\theta$ is also a smooth generalized connection of $E$.
    \item The set of all smooth generalized connections on $E$, is an affine space over the (infinite-dimensional) $\C$-vector space $\tilde{A}^1_{\en(E)}(M)$.
\end{enumerate}
\end{prop}
Any SGH vector bundle $E$ is also a complex vector bundle and it admits a hermitian metric $h$. The pair $(E,h)$ is then known as a hermitian vector bundle.
\begin{definition}
 Given a hermitian vector bundle $(E,h)$, a smooth generalized connection $\nabla$ is called a generalized hermitian connection with respect to $h$ if for any two  local sections $s, \, s^{'}$, one has 
 \begin{equation}\label{hermitian conditn}
 \tilde{D}(h(s,s^{'}))=h(\nabla s,s^{'})+h(s,\nabla s^{'})\,.    
 \end{equation}
\end{definition}
Let $\theta$ be a element in $\tilde{A}^1_{\en(E)}(M)$ and $\nabla$ be a generalized hermitian connection. Then, by Proposition \ref{generalized prop}, $\nabla+\theta$ is also a smooth generalized connection. Now, observe that $\nabla+\theta$ satisfies \eqref{hermitian conditn}
if and only if $h(\theta s,s^{'})+h(s,\theta s^{'})=0$ for all smooth local sections $s,s^{'}$. 
Consider the subsheaf  
$$\en(E,h):=\{\theta\in C^\infty(\en(E))\,|\,h(\theta s,s^{'})+h(s,\theta s^{'})=0\,\,\forall\,\,\text{local sections $s,s^{'}$}\}\,$$ of $C^\infty(\en(E))$.
 Note that $\en(E,h)$ has the structure of a real vector bundle.
\begin{prop}
    The set of all generalized hermitian connections on $(E,h)$ is an affine space over the (infinite-dimensional) $\R$-vector space $\tilde{A}^1_{\en(E,h)}(M)$ where $\tilde{A}^1_{\en(E,h)}=\tilde{A}^1\otimes_{C^{\infty}_{M,\R}}C^{\infty}(\en(E,h))$, and $C^{\infty}_{M,\R}$ is the sheaf of $\R$-valued smooth functions. 
\end{prop}
\begin{proof}
    Follows from Proposition \ref{generalized prop} after considering $E$ as a real vector bundle.
\end{proof}
\begin{remark} (cf. \cite[Section 4.2]{hubrechts05})
    $\en(E,h)$ is not always an SGH vector bundle. It is not even always a complex vector bundle. For example, if $E=M\times\C$ is the trivial SGH vector bundle. Then $\en{E}$ is again $M\times\C$ but $\en(E,h)$ is just the $ M \times i\R  $.
\end{remark}
Now $\tilde{A}^1_{E}=\tilde{A}^{1,0}_{E}\oplus\tilde{A}^{0,1}_{E}$, as in \eqref{A-pq}. So, we can decompose any smooth generalized connection $\nabla$ into two components, $\nabla^{1,0}$ and $\nabla^{0,1}$ such that $\nabla=\nabla^{1,0}+\nabla^{0,1}$ where 
\begin{equation*}
\nabla^{1,0}:\tilde{A}^0_{E}\longrightarrow\tilde{A}^{1,0}_{E}\quad;\quad\nabla^{0,1}:\tilde{A}^0_{E}\longrightarrow\tilde{A}^{0,1}_{E}\,.    
\end{equation*}
Note that for any local function $f$ on $M$ and local section $s$ of $E$, $$\nabla^{0,1}(fs)=\tilde{d_{L}}f\otimes s+f\nabla^{0,1}(s)\,.$$
\begin{definition}
     A smooth generalized connection $\nabla$ on $E$ is compatible with the GCS if $\nabla^{0,1}=\tilde{d_{L}}$.
\end{definition}
After some straightforward modifications of the proofs in  \cite[Corollary 4.2.13, Proposition 4.2.14]{hubrechts05}, one obtains the following. 
\begin{prop}\label{generalized prop1}
Let $E$ be an SGH vector bundle over $M$ with a hermitian structure $h$.
 \begin{enumerate} 
 \setlength\itemsep{1em}
     \item The space of smooth generalized connections on $E$, compatible with the GCS, forms an affine space over the  $\C$-vector space $\tilde{A}^{1,0}_{\en(E)}(M)$.
     \item There exists a unique generalized hermitian connection $\nabla$ on $E$ with respect to $h$ which is also compatible with the $GCS$. This smooth generalized connection is called \textit{generalized Chern connection}. 
\end{enumerate}   
\end{prop}
\begin{definition}
   The curvature of a smooth generalized connection $\nabla$, denoted by $\Omega_{\nabla}$ and referred to as a smooth generalized curvature,  is the composition   $$\Omega_{\nabla}:=\nabla\circ\nabla:\tilde{A}^0_{E}\longrightarrow\tilde{A}^{1}_{E}\longrightarrow\tilde{A}^{2}_{E}\,.$$ 
\end{definition}
\begin{example}\label{generalized curvature}
Let $E=M\times\C^{l}$ be the trivial SGH vector bundle. By Proposition \ref{generalized prop}, any smooth generalized connection is of the form $\nabla=\tilde{D}+\theta$ where $\theta\in\tilde{A}^1_{\en(E)}(M)$. 
Note that $\tilde{D}: \tilde{A}^p \to \tilde{A}^{p+1} $ extends naturally to $\tilde{D}: \tilde{A}_E^p \to 
\tilde{A}_E^{p+1} $ by the Leibniz rule.
So, for any local section $s\in\tilde{A}^0_{E}$, we have
\begin{align*}
\Omega_{\nabla}(s)&=(\tilde{D}+\theta)(\tilde{D}s+\theta s)\\
&=\tilde{D}(\tilde{D}s)+(\tilde{D}(\theta s)+\theta\wedge\tilde{D}s)+\theta\wedge\theta(s)\\
&=(\theta\wedge\theta+\tilde{D}(\theta))(s)\,.
\end{align*}
\end{example}
{Given any smooth generalized connection $\nabla$ on an SGH vector bundle $E$ with local trivialization $\{U_{\alpha},\phi_{\alpha}\}$, using \eqref{nabla coboundary eq} and by a straightforward calculation (cf. \cite{griffiths, hubrechts05, tu17}), we have} 
$$\phi_{\beta}\circ\Omega_{\nabla}\circ\phi^{-1}_{\beta}=\ad(\phi_{\beta\alpha})(\phi_{\alpha}\circ\Omega_{\nabla}\circ\phi^{-1}_{\alpha})\,\,\,\text{on}\,\,\,U_{\alpha\beta}\,.$$
This implies that $\Omega_{\nabla}\in\tilde{A}^{2}_{\en(E)}(M)\,.$ 
Now, assume that $E$ admits a hermitian structure such that $\nabla$ is a generalized hermitian connection with respect to $h$. Without loss of generality, we can assume that $(E,h)|_{U_{\alpha}}$ is isomorphic to $U_{\alpha}\times\C^{l}$ with constant hermitian structure. Then we can easily see that, on $U_{\alpha}$, $\overline{\theta_{\alpha}}^{t}=-\theta_{\alpha}$ and so, by Example \ref{generalized curvature}, $\overline{\Omega_{\nabla}}^{t}=-\Omega_{\nabla}\,.$ Note that, using \eqref{hermitian conditn}, we have, for any local $s_{i}\in\tilde{A}^{k_{i}}_{E}\,\,(i=1,2)$, 
$$\tilde{D}h(s_1,s_2)=h(\nabla s_1,s_2)+(-1)^{k_1}h(s_1,\nabla s_2)\,.$$
This implies that for $s_{i}\in\tilde{A}^0_{E}$,
\begin{equation}\label{her-cur}
\begin{aligned}
    0&=\tilde{D}(\tilde{D}h(s_1,s_2))\\
    &=\tilde{D}(h(\nabla s_1,s_2)+h(s_1,\nabla s_2))\\
    &=h(\Omega_{\nabla} s_1,s_2)+h(s_1,\Omega_{\nabla} s_2)\,.
\end{aligned}
\end{equation}
If we further assume that $\nabla$ is compatible with the GCS, we get,
$$\Omega_{\nabla}=\nabla^2=(\nabla^{1,0})^2+\nabla^{1,0}\circ\tilde{d_{L}}+\tilde{d_{L}}\circ\nabla^{1,0}\,.$$ Thus, $h(\Omega_{\nabla} s_1,s_2)$ and $h(s_1,\Omega_{\nabla} s_2)$ are of type $(2,0)+(1,1)$ and $(1,1)+(0,2)$, respectively. So, by \eqref{her-cur}, $(\nabla^{1,0})^2=0$. We have proved the following.

\begin{prop}\label{generalized prop2}
 Let $E$ be an SGH vector bundle over $M$ with a hermitian structure $h$. Let $\nabla$ be a smooth generalized connection with curvature $\Omega_{\nabla}$.  Then
 \vspace{0.5em}
 \begin{enumerate}
 \setlength\itemsep{1em}
     \item If $\nabla$ is a generalized hermitian connection with respect to $h$, $\Omega_{\nabla}$ satisfies $$h(\Omega_{\nabla} s_1,s_2)+h(s_1,\Omega_{\nabla} s_2)=0\quad\text{for any sections $s_1,s_2$}\,.$$
     \item If $\nabla$ is compatible with the GCS, then $\Omega_{\nabla}$ has no $(0,2)$-part, that is, $$\Omega_{\nabla}\in(\tilde{A}^{2,0}_{\en(E)}\oplus\tilde{A}^{1,1}_{\en(E)})(M)\,.$$
     \item If $\nabla$ is a generalized Chern connection on $(E,h)$, $\Omega_{\nabla}$ is of type $(1,1)$, skew-hermitian and real.
 \end{enumerate}
\end{prop}
\medskip

Recall the transversely holomorphic symplectic foliation $\mathscr{S}$  of $M$ and the corresponding leaf space $M/\mathscr{S}$. We have seen that a smooth generalized connection on an SGH vector bundle $E$ over $M$ with trivializations $\{U_{\alpha},\phi_{\alpha}\}$ is equivalent to a family $\{\theta_{\alpha}\in\tilde{A}^{1}_{\en(E)|_{U_{\alpha}}}(U_{\alpha})\}$ satisfying \eqref{nabla coboundary eq}. If each $\theta_{\alpha}$ is constant along the leaves of $\mathscr{S}$, that is, if we replace $\tilde{A}^{\bullet}$ by $A^{\bullet}$, we get the following notion.

\begin{definition}\label{transvrs connectn}
Let $E$ be an SGH vector bundle on $M$.
\begin{enumerate}
\setlength\itemsep{1em}
    \item A \textit{transverse generalized connection} on $E$, is a $\C$-linear sheaf homomorphism 
$$\nabla:A^0_{E}\longrightarrow A^1_{E}$$ which satisfies the Leibniz rule
$$\nabla(fs)=Df\otimes s+f\nabla(s)$$ for any local function $f\in F_{M}$ and any local section $s$ of $F_{M}(E)$.
\item A \textit{transverse generalized curvature} is the curvature of a transverse generalized connection $\nabla$, denoted by $\Omega_{\nabla}$. Note that, $\Omega_{\nabla}\in A^2_{\en(E)}(M)$.
\end{enumerate}
\end{definition}
\begin{remark}
    A transverse generalized connection is also a smooth generalized connection in the sense that given a transverse generalized connection $\nabla$, we can consider a $\C$-linear sheaf homomorphism
    $$\tilde{\nabla}:\tilde{A}^0_{E}=C^{\infty}_{M}\otimes_{F_{M}}A^0_{E}\longrightarrow \tilde{A}^1_{E}=C^{\infty}_{M}\otimes_{F_{M}}A^1_{E}\,,$$ defined by $$\tilde{\nabla}(fs)=\tilde{D}f\otimes s+f\nabla s$$ for any local smooth function $f$ and any local section $s\in A^0_{E}$. One can check that $\tilde{\nabla}$ is a smooth generalized connection.
\end{remark}
\begin{remark}
A smooth generalized connection always exists. A transverse generalized connection exists locally. For it to exist globally we need a smooth partition of unity, which is constant along the leaves.  If we assume $M/\mathscr{S}$ is a smooth orbifold,  such a partition of unity exists. Henceforth, in this section, we always assume that $M/\mathscr{S}$ is a smooth orbifold.
\end{remark}

We can replicate all the definitions and results for smooth generalized connections in this section, except those concerning hermitian structure, to transverse generalized connections by making the following substitutions.
\begin{table}[H]
  \centering
\begin{tblr}{|Q[c,1.5cm]|Q[c,1.5cm]|Q[c,1.5cm]|Q[c,1.5cm]|Q[c,1.5cm]|Q[c,1.5cm]|} 
\hline
 & $C^{\infty}_{M}$ & $\tilde{A}^{\bullet}$ & $\tilde{A}^{\bullet,\bullet}$ & $\tilde{D}$ & $\tilde{d_{L}}$\\[1ex]\hline
 Replaced by & $F_{M}$ & $A^{\bullet}$ & $A^{\bullet,\bullet}$ & $D$ & $d_{L}$\\[1ex]\hline
\end{tblr}
\caption{Replacement table}
  \label{table}
\end{table}

For the results concerning the generalized hermitian connection, some extra care is needed.
Consider the trivial SGH vector bundle $E=M\times\C^{r}$ with a hermitian structure $h$ and let $\nabla$ be a smooth generalized connection which satisfies \eqref{hermitian conditn}. Any hermitian metric $h$ given on $E$ is given by a function, again denoted by $h$, on $M$ that associates to any $x\in M$, a positive-definite hermitian matrix $h(x)=(h_{ij}(x))$. So we can think of $h$ as a smooth global section of $E^*\otimes E^*$, that is,
$h\in C^{\infty}(M,E^*\otimes E^*)\,.$
Now $\nabla$ is of the form $\nabla=\tilde{D}+\theta$ for some $\theta=(\theta_{ij})\in\tilde{A}^1_{\en(E)}(M)$. Let $e_i$ be the constant $i$-th unit vector considered as a section of $E$. The compatibility of $\nabla$ with $h$ yields
$$\tilde{D}h(e_i,e_j)=h(\sum_{k}\theta_{ki}e_{k},e_{j})+h(e_{i},\sum_{l}\theta_{lj}e_{j})\,,$$ or equivalently $\tilde{D}h=\theta^{t}\cdot h+h\cdot\overline{\theta}$. Furthermore, if we assume that $\nabla$ is compatible with the GCS, we get $\theta$ is of type $(1,0)$. This implies $\tilde{d_{L}}h=h\cdot\overline{\theta}\,,$ and $\theta=\overline{h}^{-1}\tilde{d_{L}}h\,.$
So, the hermitian structure uniquely determines the smooth generalized connection. Thus, for a transverse generalized connection, we would like to have a hermitian metric which is constant along the leaves. 

\begin{definition}\label{thmetric}
 A hermitian metric $h$ on an SGH vector bundle $E$ is called a \textit{transverse hermitian metric} if $h\in F_{M}(M,E^{*}\otimes E^{*})$, that is, $h$ is constant along the leaves of $\mathscr{S}$.   
\end{definition}
\begin{remark}
    Our assumption that $M/\mathscr{S}$ is a smooth orbifold ensures that a transverse hermitian metric always exists.
\end{remark}
With this notion of transverse hermitian metric and using the substitutions in  Table \ref{table}, we can replicate all the relevant definitions and extend the results in Proposition \ref{generalized prop1} and Proposition \ref{generalized prop2} to the transverse generalized connections. In particular, we have the following.
\begin{theorem}\label{main generalized prop}
Let $E$ be an SGH vector bundle over $M$ such that $M/\mathscr{S}$ is a smooth orbifold. Let $h$ be a transverse hermitian metic on $E$.
\vspace{0.5em}
\begin{enumerate}
\setlength\itemsep{1em}
     \item There exists a unique transverse generalized hermitian connection $\nabla$ on $E$ with respect to $h$ which is also compatible with the $GCS$. This transverse generalized connection is called \textit{transverse generalized Chern connection}.
     \item The transverse generalized curvature of $\nabla$, $\Omega_{\nabla}$ is of type $(1,1)$, skew-hermitian and real.
     \item The space of transverse generalized connections on $E$, compatible with the GCS, forms an affine space over the  $\C$-vector space $A^{1,0}_{\en(E)}(M)$.
\end{enumerate}
\end{theorem}
\subsection{Generalized Chern class for SGH vector bundles}\label{sgh-vb chrn}
Let $E\longrightarrow M$ be an SGH vector bundle over $M$ of complex rank $l$. Then, following Subsection \ref{char class of SGH PB}, consider the following characteristic polynomial
$$\det\left(I-t\frac{A}{2\pi\,i}\right)=\sum_{j=0}^{l}  g_{j}(A)\,t^{j}\,,$$ 
where $g_{j}\in\C[M_{l}(\C)]$ is the elementary symmetric polynomial of degree $j$ and $I$ is the identity matrix. Then, we can define an analogue of Chern classes similar to the classical case, as follows.

\begin{definition} Let $E$ be an SGH vector bundle over $M$.
    The $j$-th \textit{generalized Chern class} of $E$, denoted by $\mathbf{g}c_{j}(E)$, is defined as the image of $g_{j}$ under the generalized Chern-Weil homomorphism, that is,
    $$\mathbf{g}c_{j}(E):=\Phi(g_{j})\,,$$ where $\Phi$ as defined in \eqref{CW-hom map}.
\end{definition}
\begin{example}\label{chern class fr GH-VB}
   Let $E$ be an SGH vector bundle over $M$ where the leaf space  $M/\mathscr{S}$
   admits an orbifold structure. Let $\nabla$ be the transverse generalized Chern connection and $\Omega_{\nabla}$ be its curvature. Then $\mathbf{g}c_{1}(E)=-\frac{1}{2\pi i}[ \tr(\Omega_{\nabla}) ]$.
\end{example}
\begin{remark}
    {If the leaf space $M/\mathscr{S}$ is a smooth manifold, and if we have an SGH vector bundle over $M$ which is the pullback of a holomorphic vector bundle over $M/\mathscr{S}$, then the generalized Chern classes of the SGH bundle are the pullbacks of the corresponding Chern classes of the holomorphic vector bundle. This is discussed further in Remark \ref{rmk chrn}.}
\end{remark}

\begin{remark}
 It is worth noting that if we substitute $F_M$ with $\mathcal{O}_M$ in Definition \ref{transvrs connectn} and refer to Remark \ref{VB rmk}, we get a GH connection on an SGH vector bundle as defined by Lang et al \cite[Definition 4.1]{lang2023}. In this framework, the subsequent result has been established concerning the existence of a GH connection on an SGH vector bundle.   
\end{remark}
\begin{theorem}(\cite[Sections 4.1--4.2]{lang2023})\label{thm VB}
    Let $E$ be an SGH vector bundle over a regular GC manifold. Then, the following are equivalent:
    \begin{enumerate}
    \setlength\itemsep{1em}
        \item $E$ admits a GH connection.
        \item The short exact sequence, as defined in \eqref{B(E)2}, splits.
        \item $b(E)=0$ where $b(E)$ is the Atiyah class of the SGH vector bundle $E$ as defined in Definition \ref{atiyah VB}.
    \end{enumerate}
\end{theorem}
\begin{theorem}\label{main8}
Consider $E$ as an SGH vector bundle over a regular GC manifold $M$. Let $P$ denote the corresponding SGH principal bundle, as in \eqref{e-p}. Then, $E$ admits a GH connection if and only if  $P$ admits a GH connection.
\end{theorem}
\begin{proof}
 Follows from Theorem \ref{main3}, Theorem \ref{main4}, and Theorem \ref{thm VB}. 
\end{proof}
\subsection{Generalized holomorphic Picard group}\label{sec picard}

Let $E$ be an SGH line bundle over a GC manifold $M$ with local trivializations $\{U_{\alpha},\phi_{\alpha}\}$. The transition functions $\phi_{\alpha\beta}$, as defined in Theorem \ref{main1}, are clearly non-vanishing GH functions by Lemma \ref{cmplx fiber}, that is, $\phi_{\alpha\beta}\in\mathcal{O}^{*}_{M}(U_{\alpha}\cap U_{\beta})$\,, and satisfy
\begin{equation}\label{cocycle conditn}
\phi_{\alpha\beta}\cdot\phi_{\beta\alpha}=1\,,\,\,\,\text{and}\,\,\,\phi_{\alpha\beta}\cdot\phi_{\beta\gamma}\cdot\phi_{\gamma\alpha}=1\,.
\end{equation}

On the other hand, given any collection of GH functions $\{\phi_{\alpha\beta}\in\mathcal{O}^{*}_{M}(U_{\alpha}\cap U_{\beta})\}\,,$ satisfying \eqref{cocycle conditn}, we can construct an SGH line bundle $E$ with transition functions $\{\phi_{\alpha\beta}\}$ by taking the union of $U_{\alpha}\times\C$ overall $\alpha$ and identifying ${z}\times\C$ in $U_{\alpha}\times\C$ and $U_{\beta}\times\C$ via multiplication by $\phi_{\alpha\beta}(z)$. 

\vspace{0.3em}
Note that the transition functions $\{\phi_{\alpha\beta}\in\mathcal{O}^{*}_{M}(U_{\alpha}\cap U_{\beta})\}\,$ of $E$ over $\{U_{\alpha}\}$ represent a \v{C}ech $1$-cocycle on $M$ with coefficients in $\mathcal{O}^{*}_{M}$. Moreover, any two cocycles $\{\phi_{\alpha\beta}\}$ and $\{\phi^{'}_{\alpha\beta}\}$ define isomorphic SGH line bundles if and only if $\{\phi_{\alpha\beta}\cdot(\phi^{'}_{\alpha\beta})^{-1}\}$ is a \v{C}ech co-boundary. This implies that any SGH bundle isomorphism class of an SGH line bundle over $M$ defines a unique element in $H^{1}(M,\mathcal{O}^{*}_{M})$ and vice versa.

\begin{definition}
 Consider the set $\mathscr{E}_{1}$ as defined in Proposition \ref{G-V prop}. We give $\mathscr{E}_{1}$ a group structure, denoted by $\tau$,   where multiplication is given by tensor product and inverses by dual bundles. Denote the group $(\mathscr{E}_{1}\,,\,\tau)$ by $\mathcal{G}\pic(M)\,.$ This group is called the \textit{generalized holomorphic (GH) Picard Group} of $M$. \end{definition}

\begin{theorem}\label{pic grp iso}
   For any GC manifold $M$, $\mathcal{G}\pic(M)\cong H^{1}(M,\mathcal{O}^{*}_{M})$ as groups.  
\end{theorem}

\begin{remark}\label{ses main}
Using Corollary \ref{imp corr}, for any GC manifold $M$ we have the following short exact sequence of sheaves, 
 \begin{equation}\label{ses-eq main}
\begin{tikzcd}
0 \arrow[r] & \{\Z\} \arrow[r, hookrightarrow] & \mathcal{O}_{M} \arrow[r, "\exp"] & \mathcal{O}^{*}_{M}\arrow[r] & 0 
\end{tikzcd}\,.
\end{equation}
However, given an SGH line bundle over $M$, the image of its isomorphism class in
$H^1(M, \mathcal{O}^{*}_{M} )$, under the connecting homomorphism in the corresponding long exact sequence, may not give the first generalized Chern class of the bundle. This is because the first generalized Chern class lies in $H^{1,1}_{d_L} (M)$ which under suitable conditions lies in $H^2_D(M)$. But the latter basically describes the cohomology of the leaf space of the GCS and
may not be the same as the de Rham cohomology of $M$.  If the bundle is the pullback of a holomorphic line bundle on the leaf space of the symplectic foliation, then there is no such discrepancy.   
\end{remark}

\begin{remark}
    Note that Gualtieri \cite{Gua2} has given a different notion of Picard group corresponding to his notion of GH line bundles.
\end{remark}

\section{Dualities and vanishing theorems for SGH vector bundles}\label{sec duality}

In this section, we extend some classical results like Serre duality,  Poincar\'{e} duality, Hodge decomposition and vanishing theorems to the cohomology theory of Section \ref{sec cohomolgy} following the approach of \cite{asaoka14} and \cite{hubrechts05}.

\subsection{Generalized Serre duality and Hodge decomposition}

Let $M^{2n}$ be a compact regular GC manifold of type $k$.  Then the leaf space $M/\mathscr{S}$, as defined in \eqref{leaf sp}, is a compact space. 
Let us assume $M/\mathscr{S}$ is a smooth orbifold. 
Then, by the integrability condition of the GCS, 
$M/\mathscr{S}$ is a complex orbifold, and hence, 
orientable. Thus the cohomology 
$H^{2k} (M/\mathscr{S})$ is nontrivial. Therefore, there exists  a $(2n-2k)$-form $\chi$ on $M$ (see \cite[Section 2.8]{asaoka14}) which restricts to a volume form on each leaf such that for any $X_1,\ldots, X_{2n-2k}\in C^{\infty}(T\mathscr{S})$ and $Y\in C^{\infty}(TM)$,
\begin{equation}\label{proprty 1}
    d\chi(X_1,\ldots,X_{2n-2k},Y)=0\,.
\end{equation}

Fix a Riemannian metric on the leaf space. This induces a transverse Riemmaninan metric on $M$. We can complete the transverse metric by a Riemannian metric along the leaves to obtain a Riemannian metric on $M$ for which the leaves are minimal. In fact, $\chi$ is associated to this metric. 

\medskip
Now, define a Hodge-star operator on $A^{\bullet}$, 
\begin{equation}\label{hodge-start}
    \star: A^{\bullet}\longrightarrow A^{2k-\bullet}\,,
\end{equation} as follows: Let $U$ be an open set in $M$ on which the GCS is equivalent to a product GCS
(see Theorem \ref{darbu thm}). This implies that the symplectic foliation on $U$ is trivial. Let $e_1,\ldots,e_{2k}$ be transverse generalized $1$-forms such that $\{e_1,\ldots,e_{2k}\}$ is an orthonormal frame of $A^1(U)$. Then, for any $r>0$,  $\star:A^{r}(U)\longrightarrow A^{2k-r}(U)$ is defined by, 
\begin{equation}\label{star}
\star(e_{i_1}\wedge e_{i_2}\wedge\cdots\wedge e_{i_{r}})=\sign(i_{1},\ldots,i_{r},j_{1},\ldots,j_{2k-r})\,e_{j_1}\wedge e_{j_2}\wedge\cdots\wedge e_{j_{2k-r}}    
\end{equation}
where $\{j_{1},\ldots,j_{2k-r}\}$ is the increasing complementary sequence of $\{i_{1},\ldots,i_{r}\}$ in the set $\{1,2,\ldots,2k\}$ and $\sign(i_{1},\ldots,i_{r},j_{1},\ldots,j_{2k-r})$ denotes the sign of the permutation $\{i_{1},\ldots,i_{r},j_{1},\ldots,j_{2k-r}\}$. A simple calculation will show that 
\begin{equation}\label{**}
    \star\star=(-1)^{r(2k-r)}\,.
\end{equation}
 Define a hermitian product on $A^{r}(M)$, by
 \begin{equation}\label{hermitian prdct}   h(\alpha,\beta):=\int_{M}\alpha\wedge\overline{\star\beta}\wedge\chi\,.
 \end{equation}
 Define another operator $D^{*}:A^{r}\longrightarrow A^{r-1}$ by 
 \begin{equation}\label{adjnt of D}
  D^{*}:=(-1)^{2k(r-1)-1}\star\,D\,\star\,.   
 \end{equation}
For any $\alpha\in A^{r-1}(M)$ and $\beta\in A^{r}(M)$,
\begin{align*}
d(\alpha\wedge\star\beta\wedge\chi)=D\alpha\wedge\star\beta\wedge\chi-\alpha\wedge\star D^{*}\beta\wedge\chi+(-1)^{2k-1}\alpha\wedge\star\beta\wedge d\chi\,.
\end{align*}
 Using \eqref{proprty 1} and integrating both sides, we get $h(D\alpha,\beta)=h(\alpha,D^{*}\beta)$. The operator $D^{*}$ is called the formal adjoint of $D$.

\medskip
Since $M/\mathscr{S}$ is a complex orbifold, $\mathscr{S}$ is hermitian as well. The operator $\star$ induces a (vector space) isomorphism between $A^{p,q}(M)$ and $A^{k-q\,,\,k-p}(M)$, that is,
$$A^{p,q}(M)\cong A^{k-q\,,\,k-p}(M)\,,\quad\text{(as $\C$-vector spaces)}\,,$$
where $A^{p,q}$ as defined in \eqref{Apq}. Moreover, $D=d_{L}+d_{\overline{L}}$ on $A^{p,q}$ where $d_{L},d_{\overline{L}}$ are defined as in \eqref{dL-dL bar}.
Then the operator $D^{*}$, restricted to $A^{p,q}$, decomposes into the sum of two operators
\begin{align*}
    d^{*}_{L}:=-\star\,d_{\overline{L}}\,\star\,,\quad\text{and}\quad
   d^{*}_{\overline{L}}:=-\star\,d_{L}\,\star\,,
\end{align*} of type $(-1,0)$ and $(0,-1)$, respectively. Observe that $d^{*}_{L}$ and $d^{*}_{\overline{L}}$ are the formal adjoints of $d_{\overline{L}}$ and $d_{L}$, respectively. Define the following operators
\begin{equation}\label{harmonic operatr}
    \begin{aligned}
  \Delta_{D}:=D^{*}D+DD^{*}\,;\quad
\Delta_{d_{\overline{L}}}:=d_{\overline{L}}d^{*}_{L}+d^{*}_{L}d_{\overline{L}}\,;\quad
\Delta_{d_{L}}:=d^{*}_{\overline{L}}d_{L}+d_{L}d^{*}_{\overline{L}}\,.
    \end{aligned}
\end{equation}
Note that, similar to the classical case, $\Delta_{D}$, $\Delta_{d_{\overline{L}}}$, and $\Delta_{d_{L}}$ are self-adjoint operators.

\medskip
For any $p,q,r\geq 0$, define 
\begin{equation}\label{GH harmonic form}
    \begin{aligned}  &\mathcal{H}^{r}_{D}:=\ker(\Delta_{D})=\{\alpha\in A^{r}(M)\,|\,D\alpha=D^{*}\alpha=0\}\,;\\ 
    &\\
&\mathcal{H}^{p,q}_{d_{\overline{L}}}:=\ker(\Delta_{d_{\overline{L}}})=\{\alpha\in A^{p,q}(M)\,|\,d_{\overline{L}}\alpha=d^{*}_{L}\alpha=0\}\,;\\
&\\
&\mathcal{H}^{p,q}_{d_{L}}:=\ker(\Delta_{d_{L}})=\{\alpha\in A^{p,q}(M)\,|\,d_{L}\alpha=d^{*}_{\overline{L}}\alpha=0\}\,.
    \end{aligned}
\end{equation}
\begin{definition}
    A form $\alpha\in\mathcal{H}^{r}_{D}$ is called a transverse GH harmonic form of degree $r$ and if $\alpha\in\mathcal{H}^{p,q}_{d_{L}}$, it's called transverse GH form of type $(p,q)$.
\end{definition}
\begin{theorem}\label{harmonic thm}
    Let $M$ be a compact regular GC manifold of type $k$. Let $\mathscr{S}$ be the induced transversely holomorphic foliation. Assume that $M/\mathscr{S}$ is a smooth orbifold. Then we have the following. 
    \vspace{0.3em}
    \begin{enumerate}
    \setlength\itemsep{1em}
        \item $\mathcal{H}^{\bullet}_{D}$ and $\mathcal{H}^{\bullet,\bullet}_{d_{L}}$ both are finite dimensional.
        \item There are orthogonal decompositions 
        \vspace{0.5em}
        \begin{itemize}
        \setlength\itemsep{1em}
            \item[(a)] $A^{\bullet}(M)=\mathcal{H}^{\bullet}_{D}\oplus\img(\Delta_{D})=\mathcal{H}^{\bullet}_{D}\oplus\img(D)\oplus\img(D^{*})$\,,
            \item[(b)] $A^{\bullet,\bullet}=\mathcal{H}^{\bullet,\bullet}_{d_{L}}\oplus\img(\Delta_{d_{L}})=\mathcal{H}^{\bullet,\bullet}_{d_{L}}\oplus\img(d_{L})\oplus\img(d^{*}_{\overline{L}})$\,.
        \end{itemize}
    \end{enumerate}
\end{theorem}
\begin{proof}
Since $M/\mathscr{S}$ is a compact complex orbifold, $A^{\bullet}$ and $A^{\bullet,\bullet}$ both are hermitian vector bundles over $M$. Now, a simple local coordinate calculation shows that both $\Delta_{D}$ and $\Delta_{d_{L}}$ are strongly elliptic operators. This implies that the complexes $\{A^{\bullet},D\}$ and $\{A^{\bullet,\bullet},d_{L}\}$ both are transversely elliptic. Thus by \cite[Theorem 2.7.3]{asaoka14}, we are done. 
\end{proof}
\begin{corollary}
Let $H^{\bullet}_{D}(M)$ and $H^{\bullet,\bullet}_{d_{L}}(M)$ are defined as in \eqref{d-cohomo}, \eqref{d-L-cohomo}, respectively. Then 
    $H^{\bullet}_{D}(M)$ and $H^{\bullet,\bullet}_{d_{L}}(M)$ are finite dimensional and isomorphic to $\mathcal{H}^{\bullet}_{D}$ and $\mathcal{H}^{\bullet,\bullet}_{d_{L}}$, respectively. 
\end{corollary}
\begin{proof}
    Follows from Theorem \ref{harmonic thm}.
\end{proof}
The operator $\star\,$ induces a $\C$-linear isomorphism
$$\star:\mathcal{H}^{\bullet,\ast}_{d_{L}}(M)\cong\mathcal{H}^{k-\ast,\,k-\bullet}_{d_{\overline{L}}}(M)\,.$$
On the other hand, consider the following hermitian map
$$\tilde{h}:A^{\bullet}(M)\times A^{2k-\bullet}(M)\longrightarrow\C$$ defined by $\tilde{h}(\alpha,\beta)=\int_{M}\alpha\wedge\beta\wedge\chi\,.$ It induces a  non-degenerate pairing 
$$\Phi:H^{\bullet}_{D}(M)\times H^{2k-\bullet}_{D}(M)\longrightarrow\C\,.$$ 
\begin{theorem}\label{harmonic thm 2}
 Let $M$ be a compact regular GC manifold of type $k$. Let $\mathscr{S}$ be the induced transversely holomorphic foliation. Assume that $M/\mathscr{S}$ is a smooth orbifold. Then 
 \begin{enumerate}
 \setlength\itemsep{1em}
     \item $H^{\bullet}_{D}(M)$ satisfies the generalized Poincar\'{e} duality, that is, $$H^{\bullet}_{D}(M)\cong (H^{2k-\bullet}_{D}(M))^{*}\,.$$
     \item $H^{\bullet,\bullet}_{d_{L}}(M)$ satisfies the generalized Serre duality, that is, $$H^{\bullet,\bullet}_{d_{L}}(M)\cong (H^{k-\bullet,k-\bullet}_{d_{L}}(M))^{*}\,.$$
     \item Moreover, if $\mathscr{S}$ is also transversely K\"{a}hlerian, we have a generalized Hodge decomposition,
      $$H^{\bullet}_{D}(M)=\bigoplus_{p+q=\bullet}H^{p,q}_{d_{L}}(M)\,.$$
 \end{enumerate}
\end{theorem}
\begin{proof}
$(1)$ and $(2)$ follows from the preceding discussion.

\vspace{0.2em}
\begin{itemize}
    \item[(3)] Since $\mathscr{S}$ is transversely K\"{a}hlerian, we can prove, analogous to the classical K\"{a}hler case, that $\Delta_{D}=2\Delta_{d_{L}}$. Since $A^{\bullet}(M)=\bigoplus_{p+q=\bullet}A^{p,q}(M)$, every $\alpha$ is a transverse GH harmonic form of degree $\bullet$ if and only if its each component is a transverse GH harmonic form of type $(p,q)$ where $p+q=\bullet$. Then using Theorem \ref{harmonic thm}, we have the direct decomposition
    $$H^{\bullet}_{D}(M)=\bigoplus_{p+q=\bullet}H^{p,q}_{d_{L}}(M)\,.$$
\end{itemize}
\end{proof}

As the complex conjugation operator induces an isomorphism (of real vector spaces) $H^{\bullet,\ast}_{d_{L}}(M)\cong H^{\ast,\bullet}_{d_{L}}(M)\,,$ the operator $\star$ also induces a unitary isomorphism 
$$\overline{\star}:H^{\bullet,\ast}_{d_{L}}(M)\longrightarrow H^{k-\bullet,\,k-\ast}_{d_{L}}(M)$$ defined by $\overline{\star}(\alpha)=\star(\overline{\alpha})=\overline{\star(\alpha)}$ where $k$ is the type of $M$.

\medskip
Let $E$ be an SGH vector bundle on $M$ with a transverse hermitian structure $H$. Then $H$ can be considered as a $\C$-antilinear isomorphism $H:E\cong E^{*}$. Consider the following operators
\begin{enumerate}
\setlength\itemsep{1em}
    \item \begin{equation}\label{star E}
    \overline{\star}_{E}: A^{\bullet,\bullet}_{E}\longrightarrow A^{k-\bullet,k-\bullet}_{E^{*}}
\end{equation} defined by $\overline{\star}_{E}(\phi\otimes s)=\overline{\star}(\phi)\otimes H(s)$ for any local sections $\phi\in A^{\bullet,\bullet}$ and $s\in F_{M}(E)$ where $\overline{\star}(\phi)=\star(\overline{\phi})$. 
\item \begin{equation}\label{dL-E}
    d^{*}_{L,E}:A^{\bullet,\bullet}_{E}\longrightarrow A^{\bullet,\bullet-1}_{E}
\end{equation} defined by $d^{*}_{L,E}=-\overline{\star}_{E^{*}}\circ d_{L,E^{*}}\circ\overline{\star}_{E}$ where $d_{L,E^{*}}:A^{\bullet,\bullet}_{E^{*}}\longrightarrow A^{\bullet,\bullet+1}_{E^{*}}$ is the natural extension of $d_{L}$, as defined in \eqref{dL-dL bar}, described in Section \ref{sec cohomolgy}.
\item $\Delta_{d_{L,E}}:=d^{*}_{L,E}\,d_{L,E}+d_{L,E}\,d^{*}_{L,E}\,.$
\end{enumerate}
\vspace{0.5em}
Consider a natural hermitian scalar product on $A^{\bullet,\bullet}_{E}(M)$, similarly as in \eqref{hermitian prdct}, defined as
\begin{equation}\label{hermitian prdct-E}
  h_{E}(\alpha,\beta):=\int_{M}\alpha\wedge\overline{\star}_{E}(\beta)\wedge\chi\,.  
\end{equation} for any local section $\alpha,\beta\in A^{\bullet,\bullet}_{E}(M)$ where $\wedge$ is the exterior product on $A^{\bullet,\bullet}$ and the evaluation map $E\otimes E^{*}\longrightarrow\C$ in bundle part. Then, similarly, we can prove that 
\begin{lemma}
$d^{*}_{L,E}$ is the formal adjoint of $d_{L,E}$ and $\Delta_{d_{L,E}}$ is self-adjoint.   
\end{lemma}
Set $\mathcal{H}^{\bullet,\bullet}_{d_{L,E}}:=\ker(\Delta_{d_{L,E}})=\{\alpha\in A^{\bullet,\bullet}_{E}(M)\,|\,d_{L,E}\,\alpha=d^{*}_{L,E}\,\alpha=0\}\,.$ 

\begin{theorem}\label{harmonic thm 3}(Generalized Hodge decomposition for SGH vector bundle)
Let $(E, H)$ be an SGH vector bundle with a transverse hermitian structure $H$, over a compact regular GC manifold $M$. Assume $M/\mathscr{S}$ is a smooth orbifold. Then
\begin{enumerate}
\setlength\itemsep{1em}
    \item $\mathcal{H}^{\bullet,\bullet}_{d_{L,E}}$ is finite dimensional.
\item $A^{\bullet,\bullet}_{E}(M)=\mathcal{H}^{\bullet,\bullet}_{d_{L,E}}\oplus\img(\Delta_{d_{L,E}})=\mathcal{H}^{\bullet,\bullet}_{d_{L,E}}\oplus\img(d_{L,E})\oplus\img(d^{*}_{L,E})$
\end{enumerate}
\end{theorem}
\begin{proof}
    Follows from Theorem \ref{harmonic thm} by replacing $\{A^{\bullet,\bullet},d_{L}\}$ and $\Delta_{d_{L}}$ with $\{A^{\bullet,\bullet}_{E},d_{L,E}\}$ and $\Delta_{d_{L,E}}$, respectively.
\end{proof}
Consider the natural pairing 
$$\tilde{h}_{E}:A^{\bullet,\bullet}_{E}(M)\times A^{k-\bullet,k-\bullet}_{E^{*}}(M)\longrightarrow\C\,\,\,\text{defined by}\,\,\,\tilde{h}_{E}(\alpha,\beta)=\int_{M}\alpha\wedge\beta\wedge\chi\,,$$ where $\wedge$ is the exterior product on $A^{\bullet,\bullet}$ and the evaluation map $E\otimes E^{*}\longrightarrow\C$ in bundle part. 
\begin{theorem}(Generalized Serre duality for SGH vector bundle)\label{gc serre}
Let $E$ be an SGH vector bundle with the same assumption as in Theorem \ref{harmonic thm 3}. Then, there exists a natural $\C$-linear isomorphism between $H^{\bullet,\bullet}_{d_{L}}(M,E)$ and $(H^{k-\bullet,k-\bullet}_{d_{L}}(M,E^{*}))^{*}$, that is,
$$H^{\bullet,\bullet}_{d_{L}}(M,E)\cong(H^{k-\bullet,k-\bullet}_{d_{L}}(M,E^{*}))^{*}\quad(\text{as $\C$-vector spaces})\,.$$

\end{theorem}
\begin{proof}
Consider the natural pairing $\tilde{h}_{E}$. It induces a pairing
$$\Phi_{E}:H^{\bullet,\bullet}_{d_{L}}(M,E)\times H^{k-\bullet,k-\bullet}_{d_{L}}(M,E^{*})\longrightarrow\C\,.$$ defined as $\Phi_{E}([\alpha],[\beta])=\tilde{h}_{E}(\alpha,\beta)$ where $[\alpha],[\beta]$ denote the classes of $\alpha,\beta$, respectively. One can easily check that this is well-defined. To show that $\Phi_{E}$ is non-degenerate, by Theorem \ref{harmonic thm 3}, it is enough to show that for any $0\neq\alpha\in\mathcal{H}^{\bullet,\bullet}_{d_{L,E}}$, there exist a $\beta\in\mathcal{H}^{\bullet,\bullet}_{d_{L,E^{*}}}$ such that $\int_{M}\alpha\wedge\beta\wedge\chi\neq0$. Note that, $\overline{\star}_{E}$ induces a $\C$-antilinear isomorphism $\overline{\star}_{E}:\mathcal{H}^{\bullet,\bullet}_{d_{L,E}}\longrightarrow\mathcal{H}^{k-\bullet,k-\bullet}_{d_{L,E^{*}}}$. This implies there exist $\beta$ s.t $\overline{\star}_{E}(\alpha)=\beta$. Thus $\Phi_{E}([\alpha],[\beta])=h_{E}(\alpha,\alpha)\neq0$ and this proves the theorem.
\end{proof}
\medskip

\subsection{Generalized Vanishing Theorems}
Let $g$ be a transversely hermitian metric and $I$ be the transverse complex structure corresponding to the GCS on $M$ where $M$ and $M/\mathscr{S}$ satisfy the same conditions as before with one exception, namely, $M$ need not be compact. Define a transverse generalized form of type $(1,1)$ by
$$\omega:= g(I(\cdot),\cdot)\in A^{1,1}(M)\,.$$ This form is called the transverse generalized fundamental form. We define four operators, in particular, an analogue $\mathcal{L}$ of the Lefschetz operator, and a corresponding dual Lefschetz operator $\Lambda$, as follows.
\setlength{\jot}{0.5em}
\begin{align*}
     &(1)\,\,\mathcal{L}:A^{\bullet}\longrightarrow A^{\bullet+2}\,;\quad \alpha\mapsto\alpha\wedge\omega\,,\\
     &(2)\,\,\Lambda:=\star^{-1}\circ\mathcal{L}\circ\star:A^{\bullet}\longrightarrow A^{\bullet-2}\,,\\
     &(3)\,\,d^{*}_{L}:=-\star\,d_{\overline{L}}\,\star\,,\\
     &(4)\,\,d^{*}_{\overline{L}}:=-\star\,d_{L}\,\star\,,
\end{align*}
where $\star$ is defined in \eqref{hodge-start}. Note that $d^{*}_{L}$ and $d^{*}_{\overline{L}}$ are well defined even if $M$ is not compact. But if $M$ is compact, they are formal adjoints with respect to the hermitian inner product $h$ (see \eqref{hermitian prdct}).

\medskip
Now, assume $D\omega=0$. This implies that $\mathscr{S}$ is transversely K\"{a}hlerian with transversely K\"{a}hler metric $g$. Thus, $M/\mathscr{S}$ is a K\"{a}hler orbifold. Trivial modification of the proofs of \cite[Proposition 1.2.26, Proposition 3.1.12]{hubrechts05} yields the following identities analogous to the K\"{a}hler identities in the classical case.
\begin{prop}\label{kahler id}
Let $M$ be a regular GCS such that the leaf space $M/\mathscr{S}$ is a K\"{a}hler orbifold. Then
\vspace{0.3em}
\begin{enumerate}
\setlength\itemsep{1em}
    \item $[\Lambda,\mathcal{L}]=(k-(p+q))\id_{A^{p,q}}$ 
    \item $[d_{L},\mathcal{L}]=[d_{\overline{L}},\mathcal{L}]=0$ and $[d^{*}_{L},\Lambda]=[d^{*}_{\overline{L}},\Lambda]=0\,.$
    \item $[d^{*}_{L},\mathcal{L}]=id_{\overline{L}}\,,$ $[d_{\overline{L}},\mathcal{L}]=-id_{L}$ and $[\Lambda,d_{L}]=-id^{*}_{\overline{L}}\,,$ and $[\Lambda,d_{\overline{L}}]=id^{*}_{L}\,.$
\end{enumerate}
\end{prop}
For the rest of this section, assume that $M$ is compact and $M/\mathscr{S}$ is a K\"{a}hler orbifold. Let $E$ be an SGH vector bundle over $M$. Consider the natural extension of $\mathcal{L}$, $\Lambda$ on $A^{\bullet,\bullet}_{E}$, which will be again denoted by $\mathcal{L}$, $\Lambda$, respectively. Fix a transverse hermitian structure on $E$ (see Definition \ref{thmetric}). Let $\nabla_{E}$ be the transverse generalized Chern connection and $\Omega_{\nabla}$ be its curvature. Let $\{U_{\alpha},\phi_{\alpha}\}$ be an orthonormal trivialization of $E$. Then, on $U_{\alpha}\times\C^{r}$, with respect to such trivialization, 

\vspace{0.2em}
\begin{enumerate}
\setlength\itemsep{1em}
    \item $\overline{\star}_{E}$ can be identified with the complex conjugate $\overline{\star}$ of the operator $\star$ defined in \eqref{star}.
    \item $\nabla_{E}=D+\theta_{\alpha}\,,$ $\nabla^{1,0}_{E}=d_{\overline{L}}+\theta_{\alpha}^{1,0}$ and $\theta_{\alpha}^{*}=-\theta_{\alpha}$\,.
    \item  \begin{align*}
        (\nabla^{1,0}_{E})^{*}&=-\overline{\star}\circ\nabla^{1,0}_{E^{*}}\circ\overline{\star}\quad(\text{by \eqref{dL-E}})\\
        &=-\overline{\star}\circ(d_{\overline{L}}-\theta^{1,0}_{\alpha})\circ\overline{\star}\quad(\text{by \eqref{dL-E}})\\
        &=d^{*}_{\overline{L}}-(\theta^{1,0}_{\alpha})^{*}\,.
    \end{align*}
\end{enumerate}
\begin{align*}
    [\Lambda,\nabla^{0,1}_{E}]+i(\nabla^{1,0}_{E})^{*}&=[\Lambda,d_{L}]+[\Lambda,\theta^{0,1}_{\alpha}]+id^{*}_{\overline{L}}-i(\theta^{1,0}_{\alpha})^{*}\\
    &=[\Lambda,\theta^{0,1}_{\alpha}]-i(\theta^{1,0}_{\alpha})^{*}\quad(\text{by $(3)$ Proposition \ref{kahler id}})\,.
\end{align*}
So, the global operator $[\Lambda,\nabla^{0,1}_{E}]+i(\nabla^{1,0}_{E})^{*}$ is linear. For any point $x\in M$, we can always choose an orthonormal trivialization $\{U_{\alpha},\phi_{\alpha}\}$ such that $x\in U_{\alpha}$ and $\theta_{\alpha}(x)=0\,.$ Since $M$ is compact, $(\nabla^{1,0}_{E})^{*}=-\overline{\star}_{E^{*}}\circ\nabla^{1,0}_{E^{*}}\circ\overline{\star}_{E}$ is the formal adjoint of $\nabla^{1,0}_{E}$. 

\begin{lemma}\label{generalized lemma1}
    Let $\nabla_{E}$ be the transverse generalized Chern connection on $E$ and $\Omega_{\nabla}$ be its curvature. Then, we have

    \vspace{0.3em}
    \begin{enumerate}
    \setlength\itemsep{1em}
        \item $[\Lambda,\nabla^{0,1}_{E}]=-i( (\nabla^{1,0}_{E})^{*})=i(\overline{\star}_{E^{*}}\circ\nabla^{1,0}_{E^{*}}\circ\overline{\star}_{E})\,.$
        \vspace{0.5em}
        \item For an arbitrary $\alpha\in\mathcal{H}^{\bullet,\bullet}_{d_{L,E}}\,,$
        $$\frac{i}{2\pi}h_{E}(\Omega_{\nabla}\Lambda\alpha,\alpha)\leq 0;\quad\text{and}\quad\frac{i}{2\pi}h_{E}(\Lambda\Omega_{\nabla}\alpha,\alpha)\geq 0\,,$$ where $h_{E}$ is the natural hermitian product, defined in \eqref{hermitian prdct-E}.
    \end{enumerate}
\end{lemma}
\begin{proof}
    $(1)$ follows from the preceding discussion.
    
    \medskip
    $(2)$ By Proposition \ref{generalized prop}, $\Omega_{\nabla}=\nabla^{1,0}_{E}\circ d_{L,E}+d_{L,E}\circ\nabla^{1,0}_{E}\,.$ Let $\alpha$ be an element in $\mathcal{H}^{p,q}_{d_{L,E}}$. Since $\Lambda\alpha\in A^{p-1,q-1}_{E}(M)$, $\Omega_{\nabla}\Lambda\alpha \in A^{p,q}_{E}(M)$. So, we can compute
    \begin{align*}
        h_{E}(i\Omega_{\nabla}\Lambda\alpha,\alpha)&=ih_{E}(\nabla^{1,0}_{E} d_{L,E}\Lambda\alpha,\alpha)+ih_{E}(d_{L,E}\nabla^{1,0}_{E}\Lambda\alpha,\alpha)\\
        &=h_{E}(d_{L,E}\Lambda\alpha,-i(\nabla^{1,0}_{E})^{*}\alpha)+0\quad(\text{as $\alpha\in\mathcal{H}^{p,q}_{d_{L,E}}$})\\
        &=h_{E}(d_{L,E}\Lambda\alpha,[\Lambda,\nabla^{0,1}_{E}]\alpha)\quad(\text{by $(1)$})\\
        &=-h_{E}(d_{L,E}\Lambda\alpha,d_{L,E}\Lambda\alpha)\quad(\text{as $\nabla^{0,1}_{E}=d_{L,E}$})\\
        &\leq 0\,.
    \end{align*}
    Similarly, we can show $h_{E}(i\Lambda\Omega_{\nabla}\alpha,\alpha)\geq 0\,.$
\end{proof}

\vspace{0.2em}
\begin{definition}
~
\vspace{0.2em}
    \begin{enumerate}
     \setlength\itemsep{1em}
        \item A real $(1,1)$-transverse generalized form $\alpha$ (that is, $\alpha=\overline{\alpha}$) is called \textit{(semi-) positive} if for all GH tangent vectors $0\neq v\in\mathcal{G}M$, one has
        $$-i\alpha(v,\overline{v})>0\,(\geq 0)\,.$$
        \item Let $\nabla$ be a transverse generalized hermitian connection with respect to a transverse hermitian structure $H$ on $E$ such that $\Omega_{\nabla}\in A^{1,1}_{\en(E)}(M)$. The transverse generalized curvature $\Omega_{\nabla}$ is \textit{(Griffiths-) positive} if, for any local section $0\neq s\in F_{M}(E)$, one has
        $$H(\Omega_{\nabla}(s),s)(v,\overline{v})>0$$ for all $0\neq v\in\mathcal{G}M$.
    \end{enumerate}
\end{definition}
\begin{definition}
    An SGH line bundle $E$ over $M$ is positive if its first generalized Chern class $\mathbf{g}c_{1}(E)$ ($\in H^2_{D}(M)$ by Theorem \ref{harmonic thm 2}) can be represented by a closed positive  $(1,1)$-transverse generalized form where $\mathbf{g}c_{1}(E)$ is defined in Example \ref{chern class fr GH-VB}.
\end{definition}
\begin{theorem}\label{vanishing thm}
Let $M$ be a compact regular GC manifold of type $k$. Let the leaf space $M/\mathscr{S}$ of the induced foliation be a K\"{a}hler orbifold. Let $E$ be a positive SGH line bundle on $M$. Then, we have the following
\vspace{0.5em}
\begin{enumerate}
 \setlength\itemsep{1em}
    \item (Generalized Kodaira vanishing theorem) $$H^{q}(M,(\mathbf{\mathcal{G}^{*}M})^{p}\otimes_{\mathcal{O}_{M}}\mathbf{E})=0\quad\text{for $p+q>k$}\,.$$
    \item (Generalized Serre's theorem) For any SGH vector bundle $E^{'}$ on $M$, there exists a constant $m_0$ such that 
    $$H^{q}(M,\mathbf{E^{'}}\otimes_{\mathcal{O}_{M}}\mathbf{E}^{m})=0\quad\text{for $m\geq m_0$ and $q>0$}\,.$$
\end{enumerate}
\end{theorem}
\begin{proof}
Choose a transverse hermitian structure on $E$ such that the curvature of the transverse generalized Chern connection $\nabla_{E}$ is positive, that is, $\frac{i}{2\pi}\Omega_{\nabla_{E}}$ is a transverse K\"{a}hler form (that is, $D$-closed transverse generalized fundamental form) on $M$. We endow $M$ with this corresponding transverse K\"{a}hler structure.

\medskip
  $(1)$ With respect to this transverse K\"{a}hler structure, the operator $\mathcal{L}$ is nothing but the curvature operator $\frac{i}{2\pi}\Omega_{\nabla_{E}}$. Then, for $\alpha\in\mathcal{H}^{p,q}_{d_{L,E}}\,,$
  \begin{align*}
      0\leq\,&h_{E}(\frac{i}{2\pi}[\Lambda,\Omega_{\nabla_{E}}]\alpha,\alpha)\quad(\text{by $(2)$ Lemma \ref{generalized lemma1}})\\
      &=h_{E}([\Lambda,\mathcal{L}]\alpha,\alpha)\\
      &=(k-(p+q))h_{E}(\alpha,\alpha)\quad(\text{by $(1)$ Proposition \ref{kahler id}})\,.
  \end{align*}
  By Corollary \ref{cor5} and Theorem \ref{harmonic thm 3}, we get
  $$\mathcal{H}^{p,q}_{d_{L,E}}\cong H_{d_{L}}^{p,q}(M,E)\cong H^{q}(M,(\mathbf{\mathcal{G}^{*}M})^{p}\otimes_{\mathcal{O}_{M}}\mathbf{E})\,.$$
  Hence $H^{q}(M,(\mathbf{\mathcal{G}^{*}M})^{p}\otimes_{\mathcal{O}_{M}}\mathbf{E})=0$ for $p+q>k$.

\medskip
  $(2)$ Let $m\neq 0$. Choose a transverse hermitian structure on $E^{'}$ and denote its associated transverse generalized Chern connection by $\nabla_{E^{'}}$. Then we have an induced transverse Chern connection on $E^{''}:=E^{'}\otimes E^{m}$, denoted by $\nabla$, corresponding to the induced transverse hermitian structure,
  $$\nabla=\nabla_{E^{'}}\otimes 1+1\otimes\nabla_{E^{m}}\,,$$ where $\nabla_{E^{m}}$ is induced by $\nabla_{E}$. Its curvature is of the form
  \begin{align*}
      \frac{i}{2\pi}\Omega_{\nabla}&= \frac{i}{2\pi}\Omega_{\nabla_{E^{'}}}\otimes 1+ \frac{i}{2\pi}(1\otimes\Omega_{\nabla_{E^{m}}})\\
      &=\frac{i}{2\pi}\Omega_{\nabla_{E^{'}}}\otimes 1+m(1\otimes\frac{i}{2\pi}\Omega_{\nabla_{E}})\quad(\text{as $\Omega_{\nabla_{E^{m}}}=m\Omega_{\nabla_{E}}$})
  \end{align*}
By $(2)$ in Lemma \ref{generalized lemma1}, for $\alpha\in\mathcal{H}^{p,q}_{d_{L,E^{''}}}\,,$ we have
\begin{align*}
    0\leq\,\,&\frac{i}{2\pi}h_{E^{''}}([\Lambda,\Omega_{\nabla}]\alpha,\alpha)\\
    &=\frac{i}{2\pi}h_{E^{''}}([\Lambda,\Omega_{E^{'}}]\alpha,\alpha)+m\,h_{E^{''}}([\Lambda,\mathcal{L}]\alpha,\alpha)\\
    &= \frac{i}{2\pi}h_{E^{''}}([\Lambda,\Omega_{E^{'}}]\alpha,\alpha)+m(k-(p+q))h_{E^{''}}(\alpha,\alpha)\quad(\text{by $(1)$ Proposition \ref{kahler id}})\,.
\end{align*}
Since $h_{E^{''}}$ is a positive-definite hermitian matrix on each fiber of $E^{''}$, we can consider the fiber wise Cauchy-Schwarz inequality 
$$|h_{E^{''}}([\Lambda,\Omega_{E^{'}}]\alpha,\alpha)|\leq ||[\Lambda,\Omega_{E^{'}}]||\cdot h_{E^{''}}(\alpha,\alpha)\,.$$ By compactness of $M$, we have a global upper bound $C$ for the operator norm 
$||[\Lambda,\Omega_{E^{'}}]||$, independent of $m$, and  
 a corresponding global inequality. 
Thus, we get,
\begin{align*}
     0\leq\,\,&|\frac{i}{2\pi}h_{E^{''}}([\Lambda,\Omega_{E^{'}}]\alpha,\alpha)|+(m(k-(p+q)))\,h_{E^{''}}(\alpha,\alpha)\\
     &=\left(\frac{C}{2\pi}+(m(k-(p+q)))\right)\,h_{E^{''}}(\alpha,\alpha)\,.
\end{align*} 
\newline
Hence, if $C+2\pi\,m(k-(p+q))<0$, then $\alpha=0$. When $p=k$ and $q>0$,   $m>\frac{C}{2\pi}\geq\frac{C}{2\pi\,q}\,$ ensures $\alpha=0$. So, if we take $m_0>\frac{C}{2\pi}$, by Corollary \ref{cor5} and Theorem \ref{harmonic thm 3}, we get
$$\mathcal{H}^{k,q}_{d_{L,E^{''}}}\cong H^{q}(M,(\mathbf{\mathcal{G}^{*}M})^{k}\otimes_{\mathcal{O}_{M}}(\mathbf{E^{'}}\otimes_{\mathcal{O}_{M}}\mathbf{E^{m}}))=0\quad\text{for $m\geq m_0$ and $q>0$}\,.$$
Now, we apply these arguments to the SGH bundle $(\mathcal{G}M)^{k}\otimes E^{'}$ instead of $E^{'}$. The constant $m_0$ might change in the process but this will prove the assertion.
\end{proof}

\section{Strong generalized Calabi-Yau manifold and its leaf space}\label{sec calabi}

In this section, we give some criteria on the GCS so that the leaf space of the associated symplectic foliation is a smooth torus, and therefore, satisfies the
hypothesis that the leaf space is an orbifold, used in most of our results in this manuscript. This is a generalization of a result of Bailey et al.  \cite[Theorem1.9]{cav17}.

\medskip
Let $M^{2n}$ be a GC manifold with $+i$-eigenbundle $L$ of $(TM\oplus T^{*}M)\otimes\C$. 
 Consider the bundle $\bigwedge^{\bullet} T^{*}M\otimes\C$ as a spinor bundle for $(TM\oplus T^{*}M)\otimes\C$ with the following Clifford action 
$$(X+\eta)\cdot\rho=i_{X}(\rho)+\eta\wedge\rho\quad\text{for $X+\eta\in (TM\oplus T^{*}M)\otimes\C$}\,.$$ Then, there exits a unique line subbundle $U_M$ of $\bigwedge^{\bullet}T^{*}M\otimes\C$, called the {\it canonical line bundle} associated to the GCS, which is  annihilated by $L$ under the above Clifford action.

\medskip
At each point of $M$, $U_{M}$ is generated by a $$\rho=e^{B+i\omega}\Omega\,,$$ where $B,\omega$ are real $2$-forms and $\Omega=\theta_1\wedge\cdots\wedge\theta_{k}$ is a complex decomposable $k$-form where $k$ is the type of the GCS at that point. By \cite{Gua, Gua2}, the condition $L\cap\overline{L}=\{0\}$ is equivalent to the non-degeneracy condition
\begin{equation}\label{nondegen} \omega^{n-k}\wedge\Omega\wedge\overline{\Omega}\neq 0\,.
\end{equation}
The involutivity of $L$, with respect to the Courant bracket, is equivalent to the following condition on any local trivialization $\rho$ of $U_{M}$, $$d\rho=(X+\eta)\cdot\rho\,,$$ for some $X+\eta\in C^{\infty}((TM\oplus T^{*}M)\otimes\C)\,.$ 
\begin{definition}
    A GC manifold $M$ of type $k$ is said to be a \textit{generalized Calabi-Yau} manifold if its canonical bundle $U_{M}$ is a trivial bundle admitting a nowhere-vanishing global section $\rho$ such that $d\rho = 0$ (cf. \cite{Gua,Gua2}). $M$ is called a \textit{strong generalized Calabi-Yau} manifold if, in addition, $\rho$ is such that $\Omega=\theta_1\wedge\cdots\wedge\theta_{k}$ is globally decomposable and $d\theta_{j}=0$ for $1\leq j\leq k$.
\end{definition}
\begin{remark}
    Note that any generalized Calabi-Yau manifold is orientable because we get a global nowhere-vanishing volume form $\omega^{n-k}\wedge\Omega\wedge\overline{\Omega}\,.$
\end{remark}

\vspace{0.2em}
\begin{example}
~
\begin{enumerate}
\setlength\itemsep{1em}
    \item Any type $1$ generalized Calabi-Yau manifold is a strong generalized Calabi-Yau manifold.
    \item Any $6$-dimensional nilmanifold with $(b_1,b_2)\in\{(4,6),(4,8),(5,9),(5,11),(6,15)\}$ admits a type $2$ strong generalized Calabi-Yau structure (cf. \cite[Table 1]{cav04}) where $b_1$ and $b_2$ are the first and second betti numbers, respectively.
\end{enumerate} 
\end{example}

Let $M^{2n}$ be a compact connected strong generalized Calabi-Yau manifold of type $k$. Under some assumptions on the leaves of the induced foliation, we show that the foliation is simple. To show this, we need to use an extended version of the techniques used in \cite[Section 1.2]{cav17}.

\medskip
Let $\rho$ be a nowhere-vanishing closed section of the corresponding canonical line bundle $U_{M}$. We can express $\rho$ in the following form $$\rho=e^{B+i\omega}\wedge\Omega$$ where $B,\omega$ are real $2$-forms and $\Omega=\theta_1\wedge\cdots\wedge\theta_{k}$ is a complex decomposable $k$-form with $d\theta_{j}=0$ for $1\leq j\leq k$. Fix $p\in\{1,\ldots,k\}$. Let $\theta_{p}=\re(\theta_{p})+i\im(\theta_{p})$ where $\re(\theta_{p})$ and $\im(\theta_{p})$ denote the real and imaginary parts of $\theta_{p}$, respectively. First, we show that  $[\re(\theta_{p})]$ and $[\im(\theta_{p})]$ are linearly independent in $H_{dR}^1(M,\R)$. 

\medskip
If possible, let there exist a nontrivial linear combination, say $$\lambda_{R}[\re(\theta_{p})]+\lambda_{I}[\im(\theta_{p})]=0\,.$$ Fix $m_0\in M$ and define the map $f:M\longrightarrow\R$ as
 $$f(m)=\int_{[m_0,m]}(\lambda_{R}\re(\theta_{p})+\lambda_{I}\im(\theta_{p}))\,,$$ where integral is taken over any  path connecting $m_0$ to $m\in M$. The function $f$ is well-defined as $\lambda_{R}\re(\theta_{p})+\lambda_{I}\im(\theta_{p})$ is exact. Without loss of generality, let $\lambda_{R}\neq 0$. Note that $\theta_{p}\wedge\overline{\theta_{p}}=-2i\re(\theta_{p})\wedge\im(\theta_{p})$. Then the non-degeneracy condition $$\omega^{n-k}\wedge\left(\bigwedge^{k}_{j=1}\theta_{j}\wedge\overline{\theta_{j}}\right)\neq 0\quad\text{implies that}\quad\omega^{n-k}\wedge\left(\bigwedge^{k}_{\substack{j=1 \\ j\neq p}}\theta_{j}\wedge\overline{\theta_{j}}\right)\wedge df\neq 0\,.$$ This shows that $df$ is nowhere vanishing and so, $f$ is a submersion. Therefore, $f(M)$ is open. However,  $f(M)$ is also closed since $M$ is compact. Thus $f(M)=\R$ which is a contradiction. Hence $[\re(\theta_{p})]$ and $[\im(\theta_{p})]$ are linearly independent.
 
 \medskip
 Now, the non-degeneracy condition \eqref{nondegen} is an open condition that gives us the freedom to choose $\theta_{j}\in\Omega^1(M,\C)$ $(1\leq j\leq k)$  such that $[\re(\theta_{j})]$ and $[\im(\theta_{j})]$ are still linearly independent in $H^1(M,\Q)$. Then, we can consider the map $$\tilde{f}:M\longrightarrow\C^{k}/\Gamma\cong\prod_{j}\T^2\quad\text{defined as}\quad \tilde{f}(m)=\bigoplus_{j}\int_{[m_0,m]}\theta_{j}\,,$$ where the integral is taken over any path connecting $m_0$ to $m$ and $$\Gamma=\oplus_{j}[\theta_{j}](H_1(M,\Z))$$ is a co-compact lattice in $\C^{k}$. As before, using the non-degeneracy condition, one can show that $\tilde{f}$ is a surjective submersion. 
 
\vspace{0.5em}
Suppose $S$ is a leaf of the induced foliation $\mathscr{S}$ which is closed. Then $S$ is a compact embedded submanifold in $M$. {Let $X_{j}$ be a complex vector field on $M$ such that $$\theta_{l}(X_{j})=\delta_{lj}\quad\text{and}\quad\overline{\theta_{l}}(X_{j})=0\,,$$ where $\delta_{lj}$ is Kronecker delta and $l,j\in\{1,\ldots,k\}\,.$ This is possible since the normal bundle $\mathcal{N}$ of the foliation is trivial and the transverse holomorphic structure induces an integrable complex structure on $\mathcal{N}$ so that $C^{\infty}(\mathcal{N}^{1,0 *})=\langle\theta_{j}\,|\,j=1,\ldots,k\rangle$ and $C^{\infty}(\mathcal{N}^{0,1 *})=\langle\overline{\theta_{j}}\,|\,j=1,\ldots,k\rangle$ where $\mathcal{N}\otimes\C=\mathcal{N}^{1,0}\oplus\mathcal{N}^{0,1}$ as defined in \eqref{nrml bundle}.} Let $\re(X_{j})$ and $\im(X_{j})$ be the real and imaginary parts of $X_{j}$, respectively. Note that $\re(X_{j})$ and $\im(X_{j})$ are pointwise linearly independent and $\mathcal{L}_{Y}\Omega=\mathcal{L}_{Y}\overline{\Omega}=0$ where $Y\in\{\re(X_{j}),\im(X_{j}) : 1\leq j\leq k\}$. Therefore, these vector fields preserve the foliation $\mathscr{S}$ which is determined by $\ker (\Omega\wedge\overline{\Omega})$.

\medskip
Consider the map 
$$\psi_{S}:S\times\R^{2k}\longrightarrow M\,,$$ defined by 
$$\psi_{S}(s,\lambda_1,\ldots,\lambda_{2k})=\exp\left(\sum^{k}_{j=1}\bigg(\lambda_{2j-1}\re(X_{j})+\lambda_{2j}\im(X_{j})\bigg)\right)(s)\,.$$  $\psi_{S}$ is a local diffeomorphism as $(\psi_{S})_{*}(TS\oplus\R^{2k})=TM$. Since $\re(X_{j})$ and $\im(X_{j})$ preserve the foliation $\mathscr{S}$,  $\sum^{k}_{j=1}\bigg(\lambda_{2j-1}\re(X_{j})+\lambda_{2j}\im(X_{j})\bigg)$ also preserves $\mathscr{S}$. We conclude that all leaves in a neighborhood of $S$ are diffeomorphic to $S$. More precisely, since $S$ is compact, $\psi_{S}$ provides a leaf preserving local diffeomorphism between a tubular neighborhood of $S$ in $M$ and $S\times\prod_{j}\D^{2}$. Here $\D^{2}\subset\R^2$ is an open disk. 

\medskip
Let $V$ be the set of points in $M$ that lie in leaves that are diffeomorphic to $S$. Then $V$ is an open subset of $M$. Let $q\in\overline{V}$. Let $\alpha:\D^{2n-2k}\longrightarrow M$ be a local parametrization of the leaf through $q$ such that $\alpha(0)=q$. Then, the map $\psi:\D^{2n-2k}\times\prod_{j}\D^{2}\longrightarrow M$ defined by 
$$\psi(s,\lambda_1,\ldots,\lambda_{2k})=\exp\left(\sum^{k}_{j=1}\bigg(\lambda_{2j-1}\re(X_{j})+\lambda_{2j}\im(X_{j})\bigg)\right)(\alpha(s))$$ is a again a leaf preserving local diffeomorphism. Here, $\D^{2n-2k}\times\prod_{j}\D^{2}$ is considering with the product GCS.  So, $\img(\psi)\cap V\neq\emptyset\,.$ Let $q^{'}\in\img(\psi)\cap V$ and  $S^{'}$ be the compact leaf through it. For some $(s,\lambda_1,\ldots,\lambda_{2k})\in\D^{2n-2k}\times\prod_{j}\D^{2}$, we have $$q^{'}=\psi(s,\lambda_1,\ldots,\lambda_{2k})\,.$$ By taking the inverse of $\exp$, we can shows that $\psi_{S^{'}}(q^{'},-\lambda_1,\ldots,-\lambda_{2k})\in\img(\alpha)\,.$ Therefore, $\img(\alpha)\cap\img(\psi_{S^{'}})\neq\emptyset\,.$ So, $S^{'}$ is diffeomorphic to the leaf through $q$ via $\psi_{S^{'}}$, which shows that $q\in V\,.$ Since $V$ is both open and closed and $M$ is connected, we have $V=M\,.$ This conclude that $M$ is a fibration (fibre bundle) $M\longrightarrow B$ over a compact connected $2k$-dimensional smooth manifold. {Now, $\theta_{j}$ $(1\leq j\leq k)$ vanishes when restricted to a leaf by \cite[Corollary 2.8]{Gua2}. Since $\theta_{j}$ is also closed, it is basic for this fibration, that is, $B$ has $2k$-number of linearly independent nowhere-vanishing closed real $1$-forms.} 
\begin{prop}\label{diff prop}
    Let $B$ be any smooth compact connected $2k$-dimensional manifold. Suppose $B$ has $2k$ linearly independent nowhere-vanishing closed real $1$-forms. Then $B$ is diffeomorphic to a product of $2$-dimensional tori $\prod^{k}_{j=1}\T^{2}\,.$
\end{prop}
\begin{proof}
  Let $\{\theta_1,\ldots,\theta_{2k}\}$ be linearly independent nowhere-vanishing closed $1$-forms on $B$. Note that $\wedge_{j}\theta_{j}$ is a volume form for $B$, which is an open condition. So, we can choose $\theta_{j}$'s  such that $\theta_{j}$'s are linearly independent in $H^{1}(B,\Q)$. Fix $m_0\in B$ and consider the following map 
  $$\phi:B\longrightarrow\R^{2k}/\Gamma\cong\prod^{k}_{j=1}\T^2\quad\text{defined as}\quad\phi(m)=\bigoplus^{2k}_{j=1}\int_{[m_0,m]}\theta_{j}\,,$$ where integral is taken over any path connecting $m_0$ to $m$. Let $\Gamma=\oplus^{2k}_{j=1}[\theta_{j}](H_1(B,\Z))$. Then $\Gamma$ is a co-compact lattice in $\R^{2k}\,$. One can easily see that 
  $$\left(\bigwedge^{2k}_{\substack{j=1 \\ j\neq p}}\theta_{j}\right)\wedge d\phi_{p}\neq 0\,,$$ where $\phi_{p}:B\longrightarrow\R/[\theta_{p}](H_1(B,\Z))\cong S^1$ is the natural projection of $\phi$ onto the $p$-th component. This implies that $\phi_{j}$ $(1\leq j\leq 2k)$ is a submersion. Hence, $\phi$ is a submersion.  It follows that $\phi$ is a local diffeomorphism. Since $B$ is compact, $\phi$ is a proper map. Therefore, $\phi:B\longrightarrow\prod^{k}_{j=1}\T^2$ is a covering map and it induces an injective map 
  $$\pi_1(\phi):\pi_1(B)\longrightarrow\pi_1(\prod^{k}_{j=1}\T^2)\cong\bigoplus_{2k}\Z\,.$$ Then $\pi_1(B)\cong\bigoplus_{l}\Z$ for some $l\leq 2k$. Using the de Rham isomorphism and the universal coefficient theorem, we have $H^1_{dR}(B,\R)\cong\Hom(H_1(B,\R),\R)$ and $ H_1(B,\R) \cong H_1(B,\Z)\otimes_{\Z}\R$, respectively. Since the rank of $H^1_{dR}(B,\R)$ is $2k$, $\rank(H_1(B,\Z))=2k$. As $\pi_1(B)\cong H_1(B,\Z)$ (since $\pi_1(B)$ is abelian), $\pi_1(B)\cong\bigoplus_{2k}\Z\,.$  So, there exists a smooth covering map $\tilde{\phi}:\R^{2k}\longrightarrow B$, such that $B$ is diffeomorphic to $\R^{2k}/\pi_1(B) \cong \prod^{k}_{j=1}\T^2\,.$ 
\end{proof}

We have proved the following result.

\begin{theorem}\label{eg thm2}
    Let $M$ be a compact connected strong generalized Calabi-Yau manifold of type $k$. Let $\mathscr{S}$ be the induced foliation. Then, the following statements hold.
    
    \medskip
    \begin{enumerate}
    \setlength\itemsep{0.5em}
        \item There exists a smooth surjective submersion $\tilde{f}:M\longrightarrow\prod^{k}_{j=1}\T^2\,.$ 
        \item Suppose $\mathscr{S}$ has a closed leaf. Then, we have:
        
        \medskip
        \begin{enumerate}
        \setlength\itemsep{0.5em}
            \item All leaves are diffeomorphic and compact. Their holonomy group is trivial.
            \item The leaf space $M/\mathscr{S}$ is a smooth manifold.
            \item The submersion $\tilde{f}$ can be chosen so that the components of the fibers of $\tilde{f}$ are the symplectic leaves of $\mathscr{S}\,.$
        \end{enumerate}
    \end{enumerate}
\end{theorem}

\section{Nilpotent Lie groups, nilmanifolds, and SGH bundles}\label{sec nilpotent}
In this section, we give a complete characterization of the leaf space of a left invariant GCS on a simply connected nilpotent Lie group and its associated nilmanifolds. Finally, we construct some examples of nontrivial SGH bundles on the Iwasawa manifolds which show that the category of SGH bundles is in general different from the category of holomorphic bundles on the leaf space. The study of left invariant GCS on Lie groups was initiated in \cite{cav04}, and it subsequently motivated further developments; for instance, see \cite{alek12,fino1,fino2}. 

\medskip
Let $G^{2n}$ be a simply connected nilpotent Lie group and $\mathfrak{g}$ be its real lie algebra. Suppose $G$ has a left-invariant GCS. Since $G$ is diffeomorphic to $\mathfrak{g}$ via the exponential map, any left-invariant GCS is regular of constant type, say $k\,.$ The canonical line bundle, corresponding to a left-invariant GCS, is trivial as $G$ is contractible. So, we can choose a global trivialization of the form 
\begin{equation}\label{rho}
\rho=e^{B+i\omega}\wedge\Omega\,,   
\end{equation}
where $B,\omega$ are real left invariant $2$-forms and $\Omega=\theta_1\wedge\cdots\wedge\theta_{k}$ is a complex decomposable $k$-form with left invariant complex $1$-forms $\theta_{j}$ $(1\leq j\leq k)\,.$ 

\medskip
Let $\mathscr{S}$ be the induced foliation and $\mathcal{N}$ be its normal bundle. Then we know that $T\mathscr{S}=\ker(\Omega\wedge\overline{\Omega})\,.$ Using \cite[Theorem 4]{alek12}, the left-invariant GCS corresponds to a real Lie subalgebra $\mathfrak{s}\subset\mathfrak{g}$ such that $\mathfrak{s}\cong T_{id}\mathscr{S}$ where $id\in G$ is the identity element. Since any (simply connected) nilpotent Lie group is solvable, $S=\exp(\mathfrak{s})$ is a closed simply connected Lie subgroup of $G$ by \cite[Section II]{claude41}. Note that, by the closed subgroup theorem, $S$ is also an embedded submanifold of $G$ and so, $T_{id}S=T_{id}\mathscr{S}\,.$ Observe that
$$TG\cong G\times\mathfrak{g}\quad\text{and}\quad T\mathscr{S}\cong G\times\mathfrak{s}\,.$$  Thus any leaf of $\mathscr{S}$ is diffeomorphic to $S$ via the left multiplication map. This implies that $G$ is foliated by the left cosets of $S$, that is, the leaf space $G/\mathscr{S}$ is $G/S\,.$ Since $S$ is closed, $G/S$ is a smooth homogeneous manifold such that the quotient map $\pi_{S}:G\longrightarrow G/S$ is a smooth submersion.

\medskip
Contractibily of $G$ also implies that $\mathcal{N}\cong G\times\R^{2k}\,.$ Let $\langle\,,\,\rangle$ be an inner product on $\R^{2k}\,.$ Consider the metric $\langle\,,\,\rangle^{'}$ on $G\times\R^{2k}$ defined as $$\langle(g,v),(h,w)\rangle^{'}=\langle v,w\rangle\quad\text{for all $(g,v),(h,w)\in G\times\R^{2k}$}\,.$$
Note that $\langle\,,\,\rangle^{'}$ is $G$-invariant. Then there exists a left-invariant metric, say $h$ on $\mathcal{N}$ such that $(\mathcal{N},h)$ is isometric to $(G\times\R^{2k},\langle\,,\,\rangle^{'})\,.$ This $h$ is a  left-invariant transverse metric on $G$.

\vspace{0.3em}
Hence, we have established the following result. 

\begin{theorem}\label{eg thm}
    Let $G$ be a simply connected nilpotent Lie group with $\mathfrak{g}$ as its real lie algebra. Suppose $G$ has a left-invariant GCS. Let  $\mathscr{S}$ be the foliation induced by the GCS.  Then, the following hold.

    \medskip
    \begin{enumerate}
    \setlength\itemsep{0.5em}
        \item All leaves of $\mathscr{S}$ are diffeomorphic to the leaf through the identity element of $G$.
        \item $\mathscr{S}$ is a Riemannian foliation. In particular, $G$ admits a transverse left-invariant metric.
        \item $G$ is foliated by the left cosets of $S$ where $S\subset G$ is a closed simply connected Lie subgroup. The leaf space $G/\mathscr{S}$ is the homogeneous manifold $G/S\,.$
    \end{enumerate}
\end{theorem}
Let $\Gamma\subset G$ be a maximal lattice (that is, cocompact, discrete subgroup). Malcev (cf. \cite{malcev49}) showed that such a lattice exists if and only if $\mathfrak{g}$ has rational structure constants in some basis. Let $M^{2n}=\Gamma\backslash G$ be a nilmanifold with a left-invariant GCS. Using \cite[Theorem 3.1]{cav04}, we can say that this left-invariant GCS is generalized Calabi-Yau. This GCS on $M$ is induced from a left-invariant GCS on $G$. Let $\rho$ be a global trivialization for the canonical line bundle of the left-invariant GCS on $G$ as defined in \eqref{rho}. It also induces a global trivialization for the canonical line bundle of the left-invariant GCS on $M$. Let $\mathscr{S}_{M}$ be the induced foliation corresponding to this GCS on $M$ and $S_{M}$ be the leaf through the coset $ \Gamma \in M$. Note that, since $\mathscr{S}$ is $\Gamma$-invariant, $\mathscr{S}_{M}$ is just induced by $\mathscr{S}$, that is, $\mathscr{S}_{M}=\Gamma\backslash\mathscr{S}\,.$  

\vspace{0.3em}
Now, the quotient map $\pi_{\Gamma}: G\longrightarrow M$ is a principal $\Gamma$-bundle as well as a covering map. It induces a principal $(S\cap\Gamma)$-bundle $\pi_{\Gamma}|_{S}:S\longrightarrow S_{M}$ and so the fundamental group $\pi_1(S_{M})=S\cap\Gamma\,.$ Therefore, we can identify $S_{M}=(S\cap\Gamma)\backslash S$. Note that $\pi_{\Gamma}^{-1}(S_{M})=\Gamma S$. By \cite[Theorem 1.13]{raghunathan72}, $S\cap\Gamma$ is a maximal lattice in $S$ if and only if $\Gamma S$ is closed. Thus,
$S_{M}=(S\cap\Gamma)\backslash S\subset M$ is a compact leaf if and only if $\Gamma S$ is closed.  The transverse left-invariant metric on $G$ (see Theorem \ref{eg thm}),  is preserved by $\Gamma$-action. Therefore, it induces a transverse metric on $M$. This implies that $\mathscr{S}_{M}$ is a Riemannian foliation. 

\vspace{0.3em}
Consider the natural left $\Gamma$-action on $G/S$ defined as $\eta\cdot gS=(\eta g)S$ and its quotient space $\widehat{M}:=\Gamma\backslash G/S$ with the quotient topology such that $\widehat{\pi}:G/S\longrightarrow\widehat{M}$ is continuous. Note that this map is also open. So, we have the following diagram,  \[\begin{tikzcd}[ampersand replacement=\&]
	\& G \\
	M \&\& {G/S} \\
	\\
	{M/\mathscr{S}_{M}} \&\& {\widehat{M}}
	\arrow["{\pi_{\Gamma}}"', from=1-2, to=2-1]
	\arrow["{\pi_{S}}", from=1-2, to=2-3]
	\arrow["{\tilde{\pi}}"', from=2-1, to=4-1]
	\arrow["{\widehat{\pi}}", from=2-3, to=4-3]
\end{tikzcd}\]
where $M/\mathscr{S}_{M}$ is the leaf space and $\tilde{\pi}$ is the quotient map as defined in \eqref{leaf sp map}. We will use $\tilde{S}_{x}$ to denote the leaf through $x\in M\,.$ Let $g\in G$ and consider the map 
\begin{equation}\label{map}
    \Phi:M/\mathscr{S}_{M}\longrightarrow\widehat{M}\quad\text{defined as}\quad\Phi(\tilde{S}_{\pi_{\Gamma}(g)})=\widehat{\pi}(\pi_{S}(g))\,.
\end{equation}
Let $g,g^{'}\in G$ such that $\pi_{\Gamma}(g)$ and $\pi_{\Gamma}(g^{'})$ are in the same leaf, that is, $\tilde{S}_{\pi_{\Gamma}(g)}=\tilde{S}_{\pi_{\Gamma}(g^{'})}\,.$ To show $\Phi$ is well-defined, we need to show that $\widehat{\pi}(\pi_{S}(g))=\widehat{\pi}(\pi_{S}(g^{'}))\,.$ 

\medskip
Let $\gamma:[0,1]\longrightarrow\tilde{S}_{\pi_{\Gamma}(g)}$ be a path such that $\gamma(0)=\pi_{\Gamma}(g)$ and $\gamma(1)=\pi_{\Gamma}(g^{'})\,.$ Since $G$ is the universal cover of $M$, the path $\gamma$ lifts to a unique path $\tilde{\gamma}$ such that $\tilde{\gamma}(0)=g$ and $\tilde{\gamma}(1)=g^{''}$ with $\pi_{\Gamma}(g^{'})=\pi_{\Gamma}(g^{''})\,.$ Now the path $\tilde{\gamma}$ is contained in one of the connected components of $\pi_{\Gamma}^{-1}(\tilde{S}_{\pi_{\Gamma}(g)})$, which is a leaf of $\mathscr{S}$, say, $\tilde{g}S$ for some $\tilde{g}\in G$, such that $\pi_{\Gamma}|_{\tilde{g}S}:\tilde{g}S\longrightarrow\tilde{S}_{\pi_{\Gamma}(g)}$ is a universal covering map. So, we get the following commutative diagram,
\[\begin{tikzcd}[ampersand replacement=\&]
	\&\& {\tilde{g}S} \&\& G \\
	{[0,1]} \\
	\&\& {\tilde{S}_{\pi_{\Gamma}(g)}} \&\& M
	\arrow["{{\pi_{\Gamma}|_{\tilde{g}S}}}"', from=1-3, to=3-3]
	\arrow["{{\pi_{\Gamma}}}", from=1-5, to=3-5]
	\arrow["{{\tilde{\gamma}}}", from=2-1, to=1-3]
	\arrow["\gamma"', from=2-1, to=3-3]
	\arrow["i", hook, from=1-3, to=1-5]
	\arrow["{\tilde{i}}", from=3-3, to=3-5]
\end{tikzcd}\] where $i$ is inclusion and $\tilde{i}$ is injective immersion. Since $g,g^{''}\in \tilde{g}S$, we have $\pi_{S}(g)=\pi_{S}(g^{''})\,.$ Now $g^{'}$ and $g^{''}$ are in the same fiber of the universal covering map $\pi_{\Gamma}$, and as $\pi_1(M)=\Gamma$, there exists $\eta\in\Gamma$ such that $\eta\cdot g^{''}=g^{'}\,.$ Note that $\Gamma$ preserves the foliation $\mathscr{S}$ and so, $\pi_{S}(\eta\cdot g^{''})=\eta\cdot\pi_{S}(g^{''})\,.$ Therefore, we can see that $\pi_{S}(g^{'})=\eta\cdot\pi_{S}(g)$ which implies that $\widehat{\pi}(\pi_{S}(g))=\widehat{\pi}(\pi_{S}(g^{'}))\,.$  Hence, the map $\Phi$ is well defined.

\medskip
Now define the inverse of $\Phi$ as $$\Phi^{-1}\bigg(\widehat{\pi}(\pi_{S}(g))\bigg)=\tilde{S}_{\pi_{\Gamma}(g)}\,.$$ We need to show that $\Phi^{-1}$ is well-defined. For that, let $g,g^{'}\in G$ with the condition that $\widehat{\pi}(\pi_{S}(g))=\widehat{\pi}(\pi_{S}(g^{'}))\,.$ It is enough to show that $\tilde{S}_{\pi_{\Gamma}(g)}=\tilde{S}_{\pi_{\Gamma}(g^{'})}\,.$ 
The given condition on $\pi_{S}(g)$ and $\pi_{S}(g^{'})$ implies that there exists $\eta\in\Gamma$ such that $$\pi_{S}(g)=\eta\cdot\pi_{S}(g^{'})=\pi_{S}(\eta\cdot g^{'})\,.$$ Then, there exists $\tilde{g}\in G$ such that $g,\eta\cdot g^{'}\in\tilde{g}S\,,$ that is, they belong to the same leaf of $\mathscr{S}\,.$ This implies that $\pi_{\Gamma}(g)=\pi_{\Gamma}(\eta\cdot g^{'})=\pi_{\Gamma}(g^{'})\,.$ Hence $\tilde{S}_{\pi_{\Gamma}(g)}=\tilde{S}_{\pi_{\Gamma}(g^{'})}\,.$ So $\Phi^{-1}$ is well-defined.

\medskip
Note that $\pi_{\Gamma},\tilde{\pi},\pi_{S}$ and $\widehat{\pi}$ are open maps and we have the following commutative diagram: 
\[\begin{tikzcd}[ampersand replacement=\&]
	\& G \\
	M \&\& {G/S} \\
	\\
	{M/\mathscr{S}_{M}} \&\& {\widehat{M}}
	\arrow["{\pi_{\Gamma}}"', from=1-2, to=2-1]
	\arrow["{\pi_{S}}", from=1-2, to=2-3]
	\arrow["{\tilde{\pi}}"', from=2-1, to=4-1]
	\arrow["{\widehat{\pi}}", from=2-3, to=4-3]
	\arrow["\Phi", from=4-1, to=4-3]
\end{tikzcd}\]
This implies both $\Phi$ and $\Phi^{-1}$ are continuous and so, $\Phi$ is a homeomorphism.

\medskip
Let $x,y\in\widehat{M}$. There exist $g,g^{'}\in G$ such that $\widehat{\pi}^{-1}(x)=\Gamma gS$ and $\widehat{\pi}^{-1}(y)=\Gamma g^{'}S\,.$ Now the map $\Gamma gS\longrightarrow\Gamma g^{'}S$ defined as $\eta gs\longmapsto\eta g^{'}s$ is a diffeomorphism. In particular, both orbits are diffeomorphic to $\Gamma S\,.$ Suppose $\Gamma S$ is closed. Set 
$$\ker(\widehat{\pi}):=\left\{(gS,g^{'}S)\,|\,\widehat{\pi}(gS)=\widehat{\pi}(g^{'}S)\right\}\subset G/S\times G/S\,.$$ To show $\widehat{M}$ is Hausdroff, it is enough to show that $\ker(\widehat{\pi})$ is closed, because $\widehat{\pi}$ is an open surjection.

\medskip
Let $\{(g^1_{n}S,g^2_{n}S)\}_{n}\in\ker(\widehat{\pi})$ be a sequence converging to $(g^1S,g^2S)\,.$ Then $\{g^{j}_{n}S\}_{n}$ is converging to $g^{j}S$ for $j=1,2$. By the assumption on $\{g^{j}_{n}S\}_{n}$ $(j=1,2)$, they belong to the same $\Gamma$-orbit, in other words, there exist $\tilde{g}$ such that $g^{j}_{n}S\in\Gamma\tilde{g}S$ $(j=1,2)$. Since any two $\Gamma$-orbits are diffeomorphic, and $\Gamma S$ is closed, $\Gamma\tilde{g}S$ is also closed. This implies that $g^{j}S\in\Gamma\tilde{g}S$ for $j=1,2$. Therefore, $\widehat{\pi}(g^1 S)=\widehat{\pi}(g^2 S)\,.$ This implies $(g^1S,g^2S)\in\ker(\widehat{\pi})$ and $\ker(\widehat{\pi})$ is closed. Hence, $M/\mathscr{S}_{M}$ is Hausdroff. So, each leaf is closed as well as compact in $M$. Since $\mathscr{S}_{M}$ is a Riemannian foliation, the holonomy group of any leaf is finite, and $M/\mathscr{S}_{M}$ is a smooth orbifold. Hence we have proved the following.

\begin{theorem}\label{eg thm3}
 Let $M=\Gamma\backslash G$ be a nilmanifold with a left-invariant GCS. Let $\mathscr{S}_{M}$ be the induced foliation. Then, the following hold.

 \medskip
 \begin{enumerate}
 \setlength\itemsep{0.5em}
     \item $\mathscr{S}_{M}$ is a Riemannian foliation.
     \item $M/\mathscr{S}_{M}$ is homeomorphic to $\Gamma\backslash G/S$ where $S\subset G$ is a closed simply connected Lie subgroup.
     \item $M/\mathscr{S}_{M}$ is a compact smooth orbifold $ \iff \Gamma S$ is closed $\iff (S\cap\Gamma)\backslash S$ is compact.
 \end{enumerate}
\end{theorem}
\begin{example}\label{imp exmple3}
Consider the complex Heisenberg group 
$$G=\left\{
      \begin{pmatrix}
      1 & z_1 & z_3\\
      0 & 1 & z_2\\
      0 & 0 & 1
     \end{pmatrix}\,\Big|\,z_{j}\in\C\,(j=1,2,3)
    \right\}\,.$$
 $G$ is a $6$-dimensional simply connected nilpotent lie group. Consider a maximal lattice  $\Gamma\subset G$  defined as  
$$\Gamma=\left\{
      \begin{pmatrix}
      1 & a_1 & a_3\\
      0 & 1 & a_2\\
      0 & 0 & 1
     \end{pmatrix}\,\Big|\,a_{j}\in\Z\oplus i\Z\,(j=1,2,3)
    \right\}\,.$$

Then, $\Gamma$ acts on $G$ by left multiplication and the corresponding nilmanifold $M=\Gamma\backslash G$ is known as the Iwasawa manifold. Let $\mathfrak{g}$ be the real lie algebra of $G$. Choose a basis $\{e_1,e_2,\ldots,e_6\}\in\mathfrak{g}^{*}$ by setting
$$dz_1=e_1+ie_2\,,\quad dz_2=e_3+ie_4\,,\quad\text{and}\quad z_1dz_2-dz_3=e_5+ie_6\,.$$
These real $1$-forms are pullbacks of the corresponding $1$-forms on $M$, which we denote by the same symbols. They satisfy the following equations:
\begin{align*}
    &de_{j}=0\quad\forall\,\,1\leq j\leq 4\,.\\
    &de_5=e_{13}+e_{42}\,\quad\text{and}\quad de_6=e_{14}+e_{23}\,.
\end{align*}
Here, we make use of the notation $e_{jl}=e_{j}\wedge e_{l}$ for all $j,l\in\{1,\ldots, 6\}\,.$

\medskip
Consider the mixed complex form $$\rho=e^{i(e_{56})}(e_1+ie_2)\wedge(e_3+ie_4)\quad\text{on $M$}\,.$$ Note that $de_5\wedge de_6=0$ and $(e_1+ie_2)\wedge(e_3+ie_4)=de_5+i\,de_6$. Then, we have,
\begin{align*}
   d\rho&=e^{i(e_{56})}\wedge d(ie_{56})\wedge(e_1+ie_2)\wedge(e_3+ie_4)\\
   &=e^{i(e_{56})}\wedge d(ie_{56})\wedge(de_5+i\,de_6)\\
   &=-e^{i(e_{56})}\wedge(e_6+ie_5)\wedge de_5\wedge de_6\\
   &=0\,,
\end{align*}
and \begin{align*}
    e_{56}\wedge(e_1+ie_2)\wedge(e_3+ie_4)\wedge(e_1-ie_2)\wedge(e_3-ie_4)&=e_{56}\wedge(de_5+i\,de_6)\wedge(de_5-i\,de_6)\\
    &=-ie_{56}\wedge de_5\wedge de_6\\
    &\neq 0\,.
\end{align*}
By \cite[Theorem 3.38, Theorem 4.8]{Gua}, $M$ admits a type $2$ strong generalized Calabi-Yau structure whose canonical line bundle is generated by $\rho$. It is straightforward to see that $\rho$, when considered as a mixed form on $G$, gives a left-invariant GCS on $G$ which is a strong generalized Calabi-Yau structure.

\medskip
Let $f:G\longrightarrow\C^{2}$ be the natural projection defined as
$$\tilde{\pi}(\begin{pmatrix}
      1 & z_1 & z_3\\
      0 & 1 & z_2\\
      0 & 0 & 1
     \end{pmatrix})=(z_1,z_2)\,.$$
One can see that $\Gamma$-acts on $(z_1, z_2)$ via left translation by $\Z\oplus i\Z$. This shows that $f$ induces a surjective submersion $$\tilde{f}:M\longrightarrow\bigoplus^2_{j=1}\C/\Z\oplus i\Z\cong\T^2\times\T^2\,,$$ that satisfies the following commutative diagram,
\[\begin{tikzcd}[ampersand replacement=\&]
	G \&\& {\C^2} \\
	\\
	M \&\& {\T^2\times\T^2}
	\arrow["f", from=1-1, to=1-3]
	\arrow[from=1-3, to=3-3]
	\arrow["{\pi_{\Gamma}}"', from=1-1, to=3-1]
	\arrow["{\tilde{f}}"', from=3-1, to=3-3]
\end{tikzcd}\]
where $\C^2\longrightarrow\T^2\times\T^2$ is the natural quotient map.
So, there exist $\theta_1,\theta_2\in\Omega^2(\T^2\times\T^2,\C)$ such that $\tilde{f}^{*}(\theta_1)=e_1+ie_2$ and $\tilde{f}^{*}(\theta_2)=e_3+ie_4\,.$ Now, each fiber of this submersion is diffeomorphic to $\C/\Z\oplus i\Z\cong\T^2\,.$ Therefore, the foliation induced by the strong generalized Calabi-Yau structure on $M$ is simple with leaf space  $\T^2\times\T^2$ and with the fibers as leaves.
\end{example}

Let $M$ be a type $k$ regular GC manifold such that the leaf space $M/\mathscr{S}$ of the induced foliation $\mathscr{S}$ is a smooth manifold. Then $\mathscr{M}=M/\mathscr{S}$, as defined in \eqref{leaf sp}, becomes a complex manifold of complex dimension $k$ and the quotient map $\tilde{\pi}:M\longrightarrow\mathscr{M}$, as defined in \eqref{leaf sp map}, becomes a smooth surjective submersion. In particular, $\tilde{\pi}$ is an open map.

\medskip
For any open set $V\subseteq M$, consider the map $\tilde{\pi}^{\#}:\tilde{\pi}^{-1}\mathcal{O}_{\mathscr{M}}\longrightarrow\mathcal{O}_{M}$ defined as
\begin{equation}\label{pull back eq}
\tilde{\pi}^{\#}(f)=f\circ\tilde{\pi}\quad\text{for $f\in\mathcal{O}_{\mathscr{M}}(\tilde{\pi}(V))$}\,,    
\end{equation}
where $\mathcal{O}_{\mathscr{M}}$ is the sheaf of holomorphic functions on $\mathscr{M}$. To show $\tilde{\pi}^{\#}$ is an isomorphism, it is enough to show $\tilde{\pi}^{\#}_{x}:(\tilde{\pi}^{-1}\mathcal{O}_{\mathscr{M}})_{x}\longrightarrow(\mathcal{O}_{M})_{x}$ is an isomorphism for any $x\in M$.

\medskip
Let $x\in M$ and set $y=\tilde{\pi}(x)$. Let $\{U,\phi\}$ be a co-ordinate chart around $y$ in $\mathscr{M}$, and let $S_{x}=\tilde{\pi}^{-1}(y)$ be the fiber (leaf) over $y$. Then, choosing $U$ sufficiently small, we have the following commutative diagram by Theorem \ref{orbi thm},
\[\begin{tikzcd}[ampersand replacement=\&]
	{x\in V=\tilde{\pi}^{-1}(U)} \&\& {\tilde{S_{x}}\times U^{'}} \\
	\\
	U\subset\mathscr{M} \&\& {U^{'}\subset\C^{k}}
	\arrow["{\tilde{\phi}}", from=1-1, to=1-3]
	\arrow["{\tilde{\pi}}"', from=1-1, to=3-1]
	\arrow["\phi"', from=3-1, to=3-3]
	\arrow["{\pr_2}", from=1-3, to=3-3]
\end{tikzcd}\]
where $\tilde{\phi}$ is a GH homeomorphism, $U^{'}\subset\C^{k}$ is an open set, and $\tilde{S_{x}}$ is the universal cover of $S_{x}$. Note that $\mathcal{O}_{V}$ is isomorphic to $\tilde{\phi}^{-1}\mathcal{O}_{S_{x}\times\C^{k}}$ via $\tilde{\phi}^{\#}$, defined in a similar manner as in \eqref{pull back eq}, and $\mathcal{O}_{S_{x}\times U^{'}}=\pr_2^{-1}\mathcal{O}_{U^{'}}\,.$ Using commutativity of the diagram and the fact that $\mathcal{O}_{U}$ is isomorphic to $\phi^{-1}\mathcal{O}_{U^{'}}$ via $\phi^{\#}$, defined similarly as in \eqref{pull back eq}, we can show that $\tilde{\pi}^{-1}\mathcal{O}_{U}$ is isomorphic to $\mathcal{O}_{V}$ via $\tilde{\pi}^{\#}\,.$ Therefore, $\tilde{\pi}^{\#}_{x}$ is isomorphism and so is $\tilde{\pi}^{\#}\,.$ Similarly one can show that $\tilde{\pi}^{\#}$ is also an isomorphism even when we replace $\mathcal{O}_{M}$ by $F_{M}\,,$, that is, 
\begin{equation}\label{pi FM}
\tilde{\pi}^{\#}:\tilde{\pi}^{-1}C^{\infty}_{\mathscr{M}}\longrightarrow F_{M}\quad\text{is an isomorphism}\,.    
\end{equation}
Let $GM$ and $G^{*}M$ be the GH tangent and GH cotangent bundle of $M$, as defined in \eqref{GH tangent bundle} and \eqref{GH cotangent bundle}, respectively. Let $\{U_{\alpha}\}$ be a coordinate atlas of $\mathscr{M}$ such that $\tilde{\pi}^{-1}U_{\alpha}\cong \tilde{S_{\alpha}}\times U^{'}_{\alpha}$ via a GH homeomorphism  for some leaf $S_{\alpha}$ and $U^{'}_{\alpha}\subset\C^{k}$ open set where $\tilde{S_{\alpha}}$ is the universal cover of $S_{\alpha}\,.$ Note that, $$F_{M}(G^{*}M)|_{\tilde{\pi}^{-1}U_{\alpha}}=\spn_{F_{M}(\tilde{\pi}^{-1}U_{\alpha})}\{dz_1,\ldots,dz_{k}\}$$ where $z_{j}$ $(1\leq j\leq k)$ are holomorphic coordinates on $U_{\alpha}^{\prime}$. Then \eqref{pi FM} naturally induces an isomorphism $$\tilde{\tilde{\pi}}^{\#}|_{U_{\alpha}}:C^{\infty}(T^{1,0*}\mathscr{M})|_{U_{\alpha}}\longrightarrow F_{M}(G^{*}M)|_{\tilde{\pi}^{-1}U_{\alpha}}\,,$$ which gives rise to a sheaf isomorphism 
$$\tilde{\tilde{\pi}}^{\#}:\tilde{\pi}^{-1}C^{\infty}(T^{1,0*}\mathscr{M})\longrightarrow F_{M}(G^{*}M)\,,$$ where $T^{1,0}\mathscr{M}$ is the holomorphic tangent bundle of $\mathscr{M}\,.$ Replacing $F_{M}(G^{*}M)$ by $F_{M}(\overline{G^{*}M})$, one can show, similarly, that 
$$\tilde{\tilde{\pi}}^{\#}:\tilde{\pi}^{-1}C^{\infty}(T^{0,1*}\mathscr{M})\longrightarrow F_{M}(\overline{G^{*}M})\,.$$
Also, similarly, we can show that $F_{M}(GM)\cong\tilde{\pi}^{-1}C^{\infty}(T^{1,0}\mathscr{M})$ because $F_{M}(GM)=Hom_{F_{M}}(F_{M}(G^{*}M), F_{M})\,.$ 

\medskip
This shows that, for $l,p,q\geq 0$, $A^{l}$ (respectively, $A^{p,q}$), as defined in \eqref{A-pq}, is isomorphic to $\tilde{\pi}^{-1}(\Omega^{l}_{\mathscr{M}})$ (respectively, $\tilde{\pi}^{-1}(\Omega^{p,q}_{\mathscr{M}})$) where $\Omega^{l}_{\mathscr{M}}$ is the sheaf of $\C$-valued smooth $l$-forms on $\mathscr{M}$. In particular, The map $\tilde{\tilde{\pi}}^{\#}$ induces the pullback map $\tilde{\pi}^{*}$ from $\Omega^{l}_{\mathscr{M}}({\mathscr{M}})$ (respectively, $\Omega^{p,q}_{\mathscr{M}}({\mathscr{M}})$) to $A^{l}(M)$ (respectively, $A^{p,q}(M)$) which is an isomorphism of $\C$-vector spaces. By the definitions of $D$ and $d_{L}$ (see \eqref{D-def} and \eqref{dL-dL bar}), we have the following commutative diagrams.
\[\begin{tikzcd}[ampersand replacement=\&]
	{\Omega^{l}_{\mathscr{M}}(\mathscr{M})} \&\& {A^{l}(M)} \& {\Omega^{p,q}_{\mathscr{M}}(\mathscr{M})} \&\& {A^{p,q}(M)} \\
	\\
	{\Omega^{l+1}_{\mathscr{M}}(\mathscr{M})} \&\& {A^{l}(M)} \& {\Omega^{p,q+1}_{\mathscr{M}}(\mathscr{M})} \&\& {A^{p,q+1}(M)}
	\arrow["\tilde{\pi}^{*}(\cong)", from=1-1, to=1-3]
	\arrow["d"', from=1-1, to=3-1]
	\arrow["D", from=1-3, to=3-3]
	\arrow["\tilde{\pi}^{*}(\cong)"', from=3-1, to=3-3]
	\arrow["{\overline{\partial}}"', from=1-4, to=3-4]
	\arrow["\tilde{\pi}^{*}(\cong)", from=1-4, to=1-6]
	\arrow["\tilde{\pi}^{*}(\cong)"', from=3-4, to=3-6]
	\arrow["{d_{L}}", from=1-6, to=3-6]
\end{tikzcd}\] This shows that we have surjective homomorphisms at the level of de Rham cohomology and Dolbeault cohomology, respectively:
\begin{equation*}
  \tilde{\pi}^{*}:H^{l}_{dR}(\mathscr{M},\C)\longrightarrow H^{l}_{D}(M)\quad\text{and}\quad\tilde{\pi}^{*}:H^{p,q}_{\overline{\partial}}(\mathscr{M})\longrightarrow H^{p,q}_{d_{L}}(M)\,.
\end{equation*}
 Since $ \tilde{\pi}$ is a submersion, $\tilde{\pi}^{*}$ is one-to-one. Hence, $\tilde{\pi}^{*}$ is an isomorphism of $\C$-vector spaces. We summarize our results as follows.
 
\begin{theorem}\label{leaf sp mfld thm}
Let $M$ be a regular GC manifold such that the leaf space of the induced foliation is a smooth manifold $\mathscr{M}$. Let $GM$ and $G^{*}M$ be the GH tangent and GH cotangent bundle of $M$. Let $\tilde{\pi}:M\longrightarrow\mathscr{M}$ be the quotient map,  and let $T^{1,0}_{\mathscr{M}}$ be the sheaf holomorphic sections of the holomorphic tangent bundle of $\mathscr{M}$. Then the following hold.

\medskip
\begin{enumerate}
\setlength\itemsep{0.6em}
    \item $F_{M}\cong\tilde{\pi}^{-1}C^{\infty}_{\mathscr{M}}$ and $\mathcal{O}_{M}\cong\tilde{\pi}^{-1}\mathcal{O}_{\mathscr{M}}\,.$
    \item $F_{M}(GM)\cong\tilde{\pi}^{-1}C^{\infty}(T^{1,0}\mathscr{M})$ and $F_{M}(G^{*}M)\cong\tilde{\pi}^{-1}C^{\infty}(T^{1,0*}\mathscr{M})\,.$ In particular,  $$GM\cong\tilde{\pi}^{*}T^{1,0}\mathscr{M}\quad\text{and}\quad G^{*}M\cong\tilde{\pi}^{*}T^{1,0*}\mathscr{M}\,.$$ 
    \item $\mathbf{\mathcal{G}M}\cong\tilde{\pi}^{-1}T^{1,0}_{\mathscr{M}}$ and $\mathbf{\mathcal{G}^{*}M}\cong\tilde{\pi}^{-1}T^{1,0*}_{\mathscr{M}}\,.$
    \item $H^{\bullet}_{dR}(\mathscr{M},\C)\cong H^{\bullet}_{D}(M)$ and $H^{\bullet,\bullet}_{\overline{\partial}}(\mathscr{M})\cong H^{\bullet,\bullet}_{d_{L}}(M)\,.$
\end{enumerate}

\end{theorem}
\begin{remark}
    Theorem \ref{leaf sp mfld thm} implies that the pullback of any holomorphic vector bundle on the leaf space is an SGH vector bundle of $M$. A natural question is whether all SGH vector bundles arise in this way.  The following two examples demonstrate that this is not always the case. 
\end{remark}
\begin{example}\label{counter eg1}
    Let $M_1$ be a complex manifold and $M_2$ be a symplectic manifold. Consider the natural product GCS on $M_1\times M_2\,.$ Consider the SGH vector bundle $\otimes_{i}\pr_{i}^{*}V_{i}$ over $M_1\times M_2$, as defined in Example \ref{imp exmple2} where $V_1$ is a holomorphic vector bundle over $M_1$ and $V_2$ is flat vector bundle over $M_2$. This bundle is not a pullback of a holomorphic vector bundle over $M_1$ unless $V_2$ is trivial. 
\end{example}
\begin{example}\label{counter eg2}
    Let $G$ be the Heisenberg group and $M=\Gamma\backslash G$ be the Iwasawa manifold with the left-invariant GCS as defined in Example \ref{imp exmple3}. Let $\rho:\Gamma\longrightarrow GL_{l}(\C)$ be a nontrivial (faithful) representation. Let $G\times_{\rho}\C^{l}$ be the SGH bundle over $M$ as defined in Example \ref{imp exmple}. Let $S$ ($\cong\T^2$) be a leaf of the induced foliation. Considering $\pi_1(S)<\Gamma\,,$  $(G\times_{\rho}\C^{l})|_{S}$ is isomorphic to $\R^2\times_{\rho^{'}}\C^{l}$ where $\rho^{'}=\rho|_{\pi_1(S)}$ is a non-trivial representation. If possible let, there exist a holomorphic vector bundle $W$ over $\T^2\times\T^2$ such that $\tilde{f}^{*}W=G\times_{\rho}\C^{l}$. But, then the restriction of $\tilde{f}^{*}W$ to any of the fibers of $\tilde{f}$ is a trivial bundle which is not possible as the fibers of $\tilde{f}$ are the leaves of the induced foliation by the left-invariant GCS. Hence $G\times_{\rho}\C^{l}$ is an SGH vector bundle on $M$ which is not a pullback of any holomorphic vector bundle on $\T^2\times\T^2\,.$
\end{example}

Let $V$ be a holomorphic vector bundle over $\mathscr{M}$. Let $H^{\bullet,\bullet}_{\overline{\partial}}(\mathscr{M},V)$ denote the $V$-valued Dolbeault cohomology of $\mathscr{M}\,.$ Consider the SGH vector bundle $E:=\tilde{\pi}^{*}V$ over $M$. Then, by preceding discussions and Theorem \ref{leaf sp mfld thm}, we get the following.
\begin{corollary}\label{cor mfld1}
Let $M$ be a regular GC manifold such that the leaf space $\mathscr{M}$ of the induced foliation is a smooth manifold. Let $\tilde{\pi}:M\longrightarrow\mathscr{M}$ be the quotient map. Let $V$ be a holomorphic vector bundle on $\mathscr{M}$. Then,
 $$H^{\bullet,\bullet}_{\overline{\partial}}(\mathscr{M},V)\cong H^{\bullet,\bullet}_{d_{L}}(M,E)\,\,\,\text{(via $\tilde{\pi}^{*}$)\,\,\,where}\,\,\,E=\tilde{\pi}^{*}V\,.$$
\end{corollary}
\begin{remark}
    Corollary \ref{cor mfld1} is useful for studying $\dim_{\C}H^{\bullet,\bullet}_{d_{L}}(M,E)\,.$ This can be illustrated by the following example.
\end{remark}
\begin{example} 
Let $M$ be a regular GC manifold with the leaf space $\mathbb{CP}^{n}$.
    Let $\mathcal{O}_{\mathbb{CP}^{n}}(m)$ denote the holomorphic line bundle of degree $m$ over $\mathbb{CP}^{n}$. Consider the  SGH vector bundle $\mathcal{O}_{M}(m):=\tilde{\pi}^{*}\mathcal{O}_{\mathbb{CP}^{n}}(m)$ over $M$ where $\tilde{\pi}:M\longrightarrow\mathbb{CP}^{n}$ is the quotient map. By Corollary \ref{cor5} and Corollary \ref{cor mfld1},  we have
$$\dim_{\C}H^{\bullet}(M,(\mathbf{\mathcal{G}^{*}M})^{\bullet}\otimes_{\mathcal{O}_{M}}\mathcal{O}_{M}(m))=\dim_{\C}H^{\bullet}(\mathbb{CP}^{n},\Omega^{\bullet}_{\mathbb{CP}^{n}}\otimes_{\mathcal{O}_{\mathbb{CP}^{n}}}\mathcal{O}_{\mathbb{CP}^{n}}(m))\,,$$ where $\Omega^{\bullet}_{\mathbb{CP}^{n}}$ denotes the sheaf of holomorphic $\bullet$-forms on $\mathbb{CP}^{n}\,.$ Then, using the Bott formula (cf. \cite{bott57} and \cite[Chapter 1]{okonek11}), for $p,q\geq 0$, we get, 
 \begin{align*}
    \dim_{\C}H^{q}(M,(\mathbf{\mathcal{G}^{*}M})^{p}\otimes_{\mathcal{O}_{M}}\mathcal{O}_{M}(m)) =
    \begin{cases}
    \binom{m+n-p}{m}\binom{m-1}{p} &\text{for } q=0, 0\leq p\leq n, m>p\,;\\
    \binom{-m+p}{-m}\binom{-m-1}{n-p} &\text{for } q= n, 0\leq p\leq n, m<p-n\,;\\
     1 &\text{for } m= 0, 0\leq p=q\leq n\,;\\
     0 &\text{otherwise }\,.
    \end{cases}
\end{align*}
In particular, for $p=0$, we have
\begin{align*}
    \dim_{\C}H^{q}(M,\mathcal{O}_{M}(m)) =
    \begin{cases}
    \binom{m+n}{m} &\text{for } q=0, m\geq 0\,;\\
    \binom{-m-1}{-m-1-n} &\text{for } q= n, m\leq -n-1\,;\\
     0 &\text{otherwise }\,.
    \end{cases}
\end{align*}
\end{example}
\begin{remark}\label{rmk chrn}
    Note that, by Corollary \ref{cor mfld1}, for an SGH vector bundle $E$ over $M$ that is a pullback of a holomorphic vector bundle $V$ over $\mathscr{M}$, the generalized Chern classes of $E$ (see Subsection \ref{sgh-vb chrn}) are pullback of the Chern classes of $V$, that is $$\mathbf{g}c_{j}(E)=\tilde{\pi}^{*}(c_{j}(V))\,\,\,\text{for}\,\,\,0\leq j\leq l\,,$$ where $c_{j}(V)$ is the $j$-th Chern class of $V$, $\tilde{\pi}:M\longrightarrow\mathscr{M}$ is the quotient map, and $l$ is the complex rank of $E$. It is an interesting question if there is an SGH bundle whose generalized Chern class is not the pullback of a Chern class of a holomorphic bundle over the leaf space.
\end{remark}

\end{document}